\documentclass[11pt]{amsart}

\usepackage{amsmath,color,amssymb}
\usepackage{amsthm}
\usepackage{hyperref}

\usepackage{xcolor}
\hypersetup{
	colorlinks,
	linkcolor={blue!95!black},
	citecolor={green!70!black},
	urlcolor={blue!95!black}
}

\usepackage[margin=1.0in]{geometry}
\usepackage{tikz,tikz-cd}

\let\oldtocsection=\tocsection
\let\oldtocsubsection=\tocsubsection
\renewcommand{\tocsection}[2]{\hspace{0em}\oldtocsection{#1}{#2}}
\renewcommand{\tocsubsection}[2]{\hspace{1em}\oldtocsubsection{#1}{#2}}
 
\usepackage{mathtools}
\mathtoolsset{showonlyrefs}

\numberwithin{equation}{section}

\newtheorem{prop}{Proposition}
\newtheorem{lemma}[prop]{Lemma}

\newtheorem{thm}[prop]{Theorem}
\newtheorem{cor}[prop]{Corollary}

\numberwithin{prop}{section}

\theoremstyle{definition}
\newtheorem{defn}[prop]{Definition}
\newtheorem{ex}[prop]{Example}
\newtheorem{rmk}[prop]{Remark}

\renewcommand{\geq}{\geqslant}
\renewcommand{\leq}{\leqslant}

\renewcommand{\Re}{\mathrm{Re}}

\newcommand{\del}{\partial}
\newcommand{\dt}{\frac{\partial}{\partial t}}
\newcommand{\brs}[1]{\left| #1 \right|}

\newcommand{\gG}{\Gamma}
\renewcommand{\gg}{\gamma}

\newcommand{\gd}{\delta}
\newcommand{\gs}{\sigma}

\newcommand{\gl}{\lambda}

\newcommand{\gw}{\omega}
\newcommand{\ga}{\alpha}
\newcommand{\gb}{\beta}

\renewcommand{\ge}{\epsilon}
\newcommand{\N}{\nabla}

\newcommand{\til}[1]{\widetilde{#1}}
\renewcommand{\hat}{\widehat}

\renewcommand{\bar}[1]{\overline{#1}}

\renewcommand{\i}{\sqrt{-1}}
\newcommand{\C}{\mathbb{C}}
\newcommand{\R}{\mathbb{R}}
\newcommand{\Z}{\mathbb{Z}}

\newcommand{\mc}{\mathcal}

\newcommand{\bj}{\bar{j}}

\DeclareMathOperator{\spn}{span}

\DeclareMathOperator{\Rc}{Rc}
\DeclareMathOperator{\Rm}{Rm}

\DeclareMathOperator{\tr}{tr}
\DeclareMathOperator{\Ker}{Ker}

\DeclareMathOperator{\Id}{Id}

\DeclareMathOperator{\Isom}{Isom}

\newcommand\w[1]{\makebox[2em]{$#1$}}

\begin{document}

\title[The Gibbons-Hawking ansatz in generalized K\"ahler geometry]{The Gibbons-Hawking ansatz\\in generalized K\"ahler geometry}

\begin{abstract} We derive a local ansatz for generalized K\"ahler surfaces with nondegenerate Poisson structure and a biholomorphic $S^1$ action which generalizes the classic Gibbons-Hawking ansatz for invariant hyperK\"ahler manifolds, and allows for the choice of one arbitrary function.  By imposing the generalized K\"ahler-Ricci soliton equation, or equivalently the equations of type IIB string theory, the construction becomes rigid, and we classify all complete solutions with the smallest possible symmetry group.
\end{abstract}

\author{Jeffrey Streets}
\address{Rowland Hall\\
         University of California, Irvine, CA}
\email{\href{mailto:jstreets@uci.edu}{jstreets@uci.edu}}

\author{Yury Ustinovskiy}
\address{Chandler-Ullman Hall, Lehigh University, Bethlehem, PA}
\email{\href{mailto:yuraust@gmail.com}{yuraust@gmail.com}}

\date{\today}

\maketitle

\maketitle

\tableofcontents

\section{Introduction}
%
%
%
%
%
The classification of Einstein metrics and Ricci solitons is a central problem in geometry. Gradient shrinking and steady Ricci solitons play a crucial role in the study of Ricci flow since they arise as blow-up limits of singularities. Recently, motivated by research in mathematical physics and complex geometry, a generalization of Ricci solitons was introduced.
A smooth manifold $M^n$ equipped with a Riemannian metric $g$, closed 3-form $H$ and function $f$ define a \emph{generalized Ricci soliton} if
\begin{align}\label{f:gk_soliton_intro}
\Rc - \tfrac{1}{4} H^2 + \N^2 f =&\ 0,\\
d^* H + i_{\N f} H =&\ 0,
\end{align}
where $\Rc$ is the Ricci curvature of $(M,g)$ and $(H^2)_{ij}:=H_{ikl}H_{j}^{\ kl}$.  These equations generalize the Ricci soliton system, arising as the Euler-Lagrange equations for Einstein-Hilbert functional
\[
\mc{EH}(g,H,f):=\int_M\left(R-\frac{1}{12}|H|_g^2+|\N f|^2_g\right)e^{-f}dV_{g},
\]
where $R$ is the scalar curvature of $g$.

The generalized Ricci soliton system~\eqref{f:gk_soliton_intro} describes canonical geometric structures involving torsion (cf. \cite{gf-st-20}), and arises in physics as the type IIB string equations \cite{strings-85}.  While in~\eqref{f:gk_soliton_intro} the metric $g$ and 3-form $H$ are unrelated, there many natural geometric settings in which $g$ and $H$ are tied together.  One important special case occurs when $(M,g,I)$ is a complex Hermitian manifold with an integrable almost complex structure $I$ such that $H = -d^c \gw_I$, where $\gw_I:=g(I\cdot,\cdot)$. Since $H$ is closed, we necessarily have $dd^c\gw_I=0$, i.e. $(M,g,I)$ is a \textit{pluriclosed} manifold. There is a more refined case where $g$ is part of a generalized K\"ahler (GK) structure (cf.\,\S \ref{ss:GK_structures}), with the data then referred to as a \emph{generalized K\"ahler-Ricci soliton}, or simply GK soliton.  In the context of string theory this corresponds to imposing supersymmetric constraints on the solution of~\eqref{f:gk_soliton_intro}, and the corresponding structures were first recognized in~\cite{ga-hu-ro-84}.
Mathematically, generalized K\"ahler structures were discovered by Gualtieri~\cite{gu-10,gu-14} in the context of Hitchin's \textit{generalized geometry}~\cite{hi-10}.
These extra geometric assumptions on $g$ and $H$ relate the classification of generalized Ricci solitons to deep questions in complex geometry. For instance, in~\cite{PCF} and~\cite{GKRF} Tian and the first-named author introduced natural geometric flows~--- pluriclosed flow and generalized K\"ahler-Ricci flow~--- extending the K\"ahler-Ricci flow to the non-K\"ahler setting. Similarly to the Ricci flow setup, pluriclosed and GK solitons appear as limits of the respective flows. Thus it is essential to classify all solutions to~\eqref{f:gk_soliton_intro} as the first step in understanding the long-term behavior of a relevant flow. In \cite{st-19-soliton}, \cite{SU} we gave a conjecturally exhaustive list of generalized K\"ahler-Ricci solitons on compact complex surfaces, by constructing GK solitons on all diagonal Hopf surfaces. Since the pointed limits of the above flows are not necessarily compact, we are forced to consider also solutions to the GK soliton system with a \textit{complete} background $(M,g)$. This is the primary motivation for the present paper.

In this work we give a partial classification of complete generalized K\"ahler-Ricci solitons in dimension four.  The first step is to observe the presence of two commuting Killing fields associated to a given soliton.  The case of primary interest is when these two Killing fields are aligned, and generate a biholomorphic $S^1$ action.  Furthermore, associated to any generalized K\"ahler structure is a real Poisson tensor (cf.\,\S \ref{ss:Poisson}), which we will assume does not vanish identically.  These structural hypotheses roughly speaking represent the generic case, while the other cases are more rigid, and will be treated elsewhere.

In the locus where the Poisson tensor in nondegenerate, it is possible to describe a GK structure in terms of a triple of symplectic forms which includes hyperK\"ahler triples as a special case \cite{AGG} (cf.\,Proposition \ref{p:holo_symplectic-acs}).  In the hyperK\"ahler case, the (tri-Hamiltonian) $S^1$ action determines locally a moment map $\pmb \mu$ to $\mathbb R^3$, and we use the notation $\mathbb R^3_{\pmb \mu}$ to denote the image space of such a moment map.  Using these moment map coordinates, the hyperK\"ahler metric has an explicit local description given by the famous Gibbons-Hawking ansatz \cite{GibbonsHawking}:
\begin{align*}
g = W \left( dx^2 + dy^2 + dz^2 \right) + W^{-1} \eta^2,
\end{align*}
where $W$ is a harmonic function on flat $\mathbb R^3_{\pmb\mu}$, and $\eta$ is a principal $S^1$ connection with curvature $d \eta = *d W$.  This construction is reversible, and by choosing $W$ to be a linear combination of a constant and suitably normalized Green's functions based at points in $\mathbb R^3_{\pmb \mu}$, one obtains a hyperK\"ahler $4$-manifold admitting a free isometric circle action, which is incomplete due to the poles of $W$.  These singularities are however removable, and by adding a single point for each pole of $W$, one obtains finally a complete hyperK\"ahler $4$-manifold admitting an effective $S^1$ action, with the added points being the fixed point set of the action.

As shown by Bielawski \cite{Bielawski}, the construction described above yields an exhaustive list of all possible simply connected hyperK\"ahler 4-manifolds admitting an effective circle action.  A key point in the analysis, described in more detail below, is to determine the global structure of the moment map for an arbitrary hyperK\"ahler manifold with effective $S^1$ action.  It is shown that the assumption of completeness implies that the image of the moment map, endowed with the \emph{background} horizontal metric, by which we mean the Euclidean metric on $\mathbb R^3_{\pmb\mu}$, is also \emph{complete}.  This implies that the image must be all of $\mathbb R^3$, thus the harmonic function with poles $W$ is defined globally, and results from elliptic theory imply that it is a linear combination of a constant and Green's functions, as expected.

In the more general setting of generalized K\"ahler structures, the symplectic triple together with $S^1$ action still allows us to locally define a moment map $\pmb \mu$ to $\mathbb R^3_{\pmb \mu}$. We use this moment map to derive an explicit local description of the generalized K\"ahler structure.  As above we let $W^{-1}$ denote the square length of the canonical vector field generating the $S^1$ action. Then one can express the GK metric locally as 
\begin{align*}
g = W h + W^{-1} \eta^2,
\end{align*}
where $W$ satisfies an explicit linear second order elliptic PDE, and $\eta$ is a principal $S^1$ connection whose curvature is determined by $W$.  The background horizontal metric $h$ now depends on an extra scalar function $p$, which is known as the \emph{angle function} associated to the GK structure. Specifically we have the following (cf. Theorem~\ref{t:nondegenerate_gk_description} below):
\begin{thm}[Generalized-K\"ahler Gibbons-Hawking ansatz] \label{t:mainGH}
	Fix a smooth 3-dimensional manifold $N$ and consider
	\begin{enumerate}
		\item an open map $\iota\colon N\to \R^3_{\pmb\mu}$,
		\item smooth functions
		\[
		p\colon N\to (-1,1),\quad W\colon N\to (0,+\infty)
		\]
		solving the equation
		\begin{equation*}
		W_{11}+W_{22}+W_{33}+2(pW)_{23}=0,
		\end{equation*}
		such that the closed differential form $\beta\in\Lambda^2(N,\R)$
		\[
		\gb=(W_3+(pW)_2)d\mu_1\wedge d\mu_2-(W_2+(pW)_3)d\mu_1\wedge d\mu_3+W_1d\mu_2\wedge d\mu_3
		\]
		represents a class in $H^2(N,2\pi\Z)$,
		\item a connection form $\eta$ with curvature $\beta$ in the principal $S^1$-bundle $\pi\colon M\to N$ determined by $[\beta]$.
	\end{enumerate}
	Then the total space of the principal $S^1$-bundle $M$
	admits a nondegenerate GK structure
	\[
	(M,g,I,J)
	\]
	with
	\[
	g=Wh+W^{-1}\eta^2,\quad h=(1-p^2)d\mu_1^2+d\mu_2^2+d\mu_3^2-2p\,d\mu_2d\mu_3,
	\]
	and $I$, $J$ the unique almost complex structures such that the complex-valued 2-forms
	\[
	\begin{split}
	\Omega_{I}:=&
	(-d\mu_1+\sqrt{-1}d\mu_2)\wedge(\eta+\sqrt{-1}W(d\mu_3-pd\mu_2)),\\
	\Omega_{J}:=&
	(-d\mu_1+\sqrt{-1}d\mu_3)\wedge(\eta+\sqrt{-1}W(-d\mu_2+pd\mu_3)),
	\end{split}
	\]
	are holomorphic with respect to $I$ and $J$ respectively.
	
	Conversely any nondegenerate GK manifold $(M,g,I,J)$ with a free isometric tri-Hamiltonian $S^1$ action arises via this construction.
\end{thm}

This is a local description of any nondegenerate GK structure with a tri-Hamiltonian $S^1$ action. If we further impose the soliton equations, then function $p$ is determined explicitly up to two real parameters $a_{\pm}$, yielding explicit choices of metric $h$ defined globally on $\mathbb R^3_{\pmb \mu}$.  Crucially however, these metrics are not complete, with the completion points corresponding to the possible zeros of the background Poisson tensor.  Furthermore, the $S^1$ action on the nondegeneracy locus will only be tri-Hamiltonian after lifting to a certain infinite covering space.  Thus to determine the possible image spaces of the moment map for a complete GK soliton, we must determine the relevant deck transformation group $\gG$ and its action on $\mathbb R^3_{\pmb \mu}$, and understand the completions of the quotient spaces, i.e. $\bar{\mathbb R^3_{\pmb \mu} / \gG}$.  We prove that there are three possibilities:

\begin{enumerate}
\item One has $a_- = 0$, $a_+ = \frac{2}{k_+}$, $k_+\in\Z$ and $\bar{\mathbb R^3_{\pmb \mu} / \gG} \cong N(a_+, 0)$ is an orbifold diffeomorphic to a global quotient of $\mathbb C \times \mathbb R$ by a linear $\mathbb Z_{|k_+|}$-action on $\mathbb C$.  The metric $h$ extends as a smooth orbifold metric with cone angle $2\pi/ |k_+|$ along $\{0\} \times \mathbb R$.
\item One has $a_- = 0$, $a_+ = \frac{2}{k_+}$, and $\bar{\mathbb R^3_{\pmb \mu} / \gG} \cong N(a_+, 0) / \mathbb Z$, where $N(a_+, 0)$ is as in item (1) and the $\mathbb Z$ action is generated by a translation in the $\mathbb R$ factor and (possibly trivial) rotation in $\C_{|k_+|}$ factor.
\item One has $a_+=2/k_+$, $a_-=2/k_-$, $k_{\pm}\in\Z$ and $\bar{\mathbb R^3_{\pmb \mu} / \gG} \cong N(a_+, a_-)$ is an orbifold diffeomorphic to a product $S^2(k_+,k_-)\times \R$, where $S^2(k_+,k_-)$ is a \emph{spindle} 2-sphere with cone angles $2\pi/ |k_+|$ and $2\pi/|k_-|$.
\end{enumerate}

With this in hand we can now state the main existence theorem for GK solitons:

\begin{thm} \label{t:existence} The following hold:
\begin{enumerate}
\item Let $(k_+, l_+)$ be a pair of coprime integers, $0\leq l_+<|k_+|$ and denote $a_+ = \frac{2}{k_+}$. Fix a collection of points $\{z_1, \dots, z_n\}$ in the smooth locus of $N(a_+, 0)$.  There exists a one-parameter family of complete rank one generalized K\"ahler-Ricci solitons $(M^4, g, I, J)$ admitting an isometric $S^1$ action, such that $M / S^1 \cong N(a_+, 0)$, and $M^{S^1} / S^1 = \{z_1, \dots, z_n\}$.  The manifold $M \backslash M^{S^1}$ is a Seifert fibration over $N(a_+, 0) \backslash \{z_1, \dots, z_n\}$, and the lift of the orbifold locus in this domain consists of points of type $(k_+, l_+)$.
\item Let $(k_+,l_+)$ be a pair of coprime integers, $0\leq l_+<|k_+|$ and denote $a_+ = \frac{2}{k_+}$. There exists a two-parameter family of complete rank one generalized K\"ahler-Ricci solitons $(M^4, g, I, J)$ admitting an effective isometric $S^1$ action, such that $M / S^1 \cong N(a_+, 0) / \mathbb Z$.  The manifold $M$ is a Seifert fibration over $N(a_+, 0) / \mathbb Z$, and the lift of the orbifold locus in this domain consists of points of type $(k_+, l_+)$.
\item Let $(k_+, l_+), (k_-, l_-)$ be two pairs of coprime integers, $0\leq l_\pm<|k_\pm|$ and denote $a_\pm = \frac{2}{k_\pm}$. Fix a collection of points $\{z_1, \dots, z_n\}$ in the smooth locus of $N(a_+, a_-)$.
There exists a two-parameter family of complete rank one generalized K\"ahler-Ricci solitons $(M^4, g, I, J)$ admitting an effective isometric $S^1$ action, such that $M / S^1 \cong N(a_+, a_-)$, and $M^{S^1} / S^1 = \{z_1, \dots, z_n\}$.  The manifold $M \backslash M^{S^1}$ is a Seifert fibration over $N(a_+, a_-) \backslash \{z_1, \dots, z_n\}$, and the lift of the two orbifold loci in this domain consists of points of type $(k_{\pm}, l_{\pm})$, respectively.
\end{enumerate}
\end{thm}

The situation is summarized in Figures \ref{f:fig1} and \ref{f:fig2}.  The quantity $\mu_1$ is a certain coordinate in $\mathbb R^3_{\pmb \mu}$ whose periods generate a primitive subgroup $\mathbb Z \subset \gG$, so that $\mu_1$ descends to an $S^1$-valued coordinate on $\R^3_{\pmb\mu}/\Z$.

\begin{figure}[ht]
\begin{tikzpicture}

\draw  (-2.75,0) ellipse (0.18 and 0.6);
\draw [white,fill=white] (-2.75,1) rectangle (-2.5,-1);
\draw  [dotted] (-2.75,0) ellipse (0.18 and 0.6);

\draw[->] (-3.5,0) .. controls (-2,1.5) and (0.25,1.85) .. (0.5,1.9);
\draw[->] (-3.5,0) .. controls (-2,-1.5) and (0.25,-1.85) .. (0.5,-1.9);

\draw[<->] (-4.1,0) to (-3.1,0);
\node at (-3.7175,0.1881) {\tiny{$\mathbb R$}};

\draw [->](-2.5,0.3) .. controls (-2.46,0.1) and (-2.46,-0.1) .. (-2.5,-0.3);
\node at (-2.25,0) {\tiny{$\mu_1$}};

\node (v1) at (-3.8537,-1.1406) {\tiny{$\mathbf{T_+} = \{p = 1\}$}};
\draw[dotted,->] (v1) to (-3.7386,-0.1001);

\draw [<->](-0.6084,1.0546) to (-0.9459,2.1143);
\node at (-1.0865,1.7506) {\tiny{$\mathbb R$}};

\node (v3) at (-3.8,1.3) {\tiny{$(D^2/\mathbb Z_{k_+})\times \R$}};
\draw[dotted,->] (v3) to (-3.2,0.4);

\node (v5) at (1.0,-0.8) {\tiny{$\{z_1, \dots, z_n\}$}};

\draw[fill=black] (-0.5,0.4) circle (0.03);
\draw[fill=black] (-1.1,-0.5) circle (0.03);
\draw[fill=black] (-2.2,0.5) circle (0.03);
\draw [dotted,->] (v5) to (-0.42,0.32);
\draw [dotted,->] (v5) to (-0.99,-0.52);
\draw [dotted,->] (v5) to (-2.1,0.5);
\end{tikzpicture}
\caption{Configuration space $N(a_+, 0)\simeq \C/{\Z_{k_+}}\times \R$}
\label{f:fig1}

\bigskip
\bigskip

\begin{tikzpicture}

\draw (0.4,0) ellipse (0.2 and 0.88);
\draw [white,fill=white] (0.4,1) rectangle (1,-1);
\draw [dotted] (0.4,0) ellipse (0.2 and 0.88);

\draw  (-2.75,0) ellipse (0.18 and 0.6);
\draw [white,fill=white] (-2.75,1) rectangle (-2.5,-1);
\draw  [dotted] (-2.75,0) ellipse (0.18 and 0.6);

\draw (-3.5,0) .. controls (-2,1.5) and (0.25,2) .. (1,0);
\draw (-3.5,0) .. controls (-2,-1.5) and (0.25,-2) .. (1,0);

\draw[<->] (-4.1,0) to (-3.1,0);
\node at (-3.7175,0.1881) {\tiny{$\mathbb R$}};

\draw [->](-2.5,0.3) .. controls (-2.46,0.1) and (-2.46,-0.1) .. (-2.5,-0.3);
\node at (-2.25,0) {\tiny{$\mu_1$}};

\draw[<->] (0.7,0) to (1.6,0);
\node at (1.2175,0.2092) {\tiny{$\mathbb R$}};

\draw [<-](0.1,0.3) .. controls (0.06,0.1) and (0.06,-0.1) .. (0.1,-0.3);
\node at (-0.1,0) {\tiny{$\mu_1$}};

\node (v1) at (-3.8537,-1.1406) {\tiny{$\mathbf{T_+} = \{p = 1\}$}};
\draw[dotted,->] (v1) to (-3.7386,-0.1001);

\node (v2) at (1.5216,-1.1054) {\tiny{$\mathbf{T}_- = \{p = -1\}$}};
\draw[dotted,->] (v2) to (1.2474,-0.0931);

\draw [<->](-0.8,0.8) to (-0.8,1.8);
\node at (-1.0251,1.5097) {\tiny{$\mathbb R$}};

\node (v3) at (-4.5,1.1) {\tiny{$(D^2/\mathbb Z_{k_+})\times \R$}};
\draw[dotted,->] (v3) to (-3.2,0.4);

\node (v4) at (2.3,1.3) {\tiny{$(D^2/\mathbb Z_{k_-})\times \R$}};
\draw[dotted,->] (v4) to (0.801,0.6521);

\node (v5) at (-2.5,2.0) {\tiny{$\{z_1,\dots,z_n\}$}};

\draw[fill=black] (-0.5,0.4) circle (0.03);
\draw[fill=black] (-1.1,-0.5) circle (0.03);
\draw[fill=black] (-2.2,0.5) circle (0.03);
\draw [dotted,->] (v5) to (-0.58,0.48);
\draw [dotted,->] (v5) to (-1.16,-0.42);
\draw [dotted,->] (v5) to (-2.22,0.58);
\end{tikzpicture}
\caption{Configuration space $N(a_+, a_-)\simeq S^2(k_+,k_-)\times \R$}
\label{f:fig2}
\end{figure}

\begin{rmk}
\begin{enumerate}
\item For the definition of points of type $(k, l)$ see Remark \ref{r:local_s1_quotient}.  For the definition of rank of a soliton see Definition \ref{d:GKsoliton}.
\item One parameter in cases (1) and (2) corresponds to the weight on the constant function in the decomposition of $W$. This phenomenon also occurs in the original Gibbons-Hawking construction, where one may include an arbitrary nonnegative constant in the choice of $W$. Depending on whether this constant is zero or not, one obtains multi-Eguchi-Hanson or multi-Taub-NUT metrics.
\item The second parameter in case (2) corresponds to a choice of a connection $\eta$ on the Seifert bundle over $N(a_+,0)/\Z$ with a prescribed curvature form.
\item In case (3) there are two real parameters, given by weights on the constant function and an anomalous smooth solution, which further satisfy an integrality constraint  (see~\eqref{f:W_quantization}).
\item The global topology of these examples is implicit in the construction, with the manifolds arising via completion of explicit Seifert bundles by adding finitely many points.
\end{enumerate}
\end{rmk}

Next, crucially, we show that our construction classifies all possible solitons in four dimensions with the \emph{smallest} possible isometry groups.  Those with larger symmetry groups arise via constructions in toric geometry and will be treated elsewhere.

\begin{thm} \label{t:uniqueness} Let $(M^4, g, I, J)$ be a complete generalized K\"ahler-Ricci soliton with generically nondegenerate Poisson tensor, Ricci curvature bounded below and $\dim H^*(M,\R)<\infty$. If $\dim \Isom(g) \leq 1$, then exactly one of the following holds.
\begin{enumerate}
\item $(M^4, g, I)$ is hyperK\"ahler.
\item $\Isom(g) = S^1$ and $(M^4, g, I, J)$ is isomorphic to one of the examples constructed in Theorem \ref{t:existence}.
\item $\Isom(g)=\R$ and $(M^4, g, I, J)$ is isomorphic to a $\Z$-cover of an example constructed in Theorem \ref{t:existence} with a trivial $S^1$-bundle structure.
\end{enumerate}
\end{thm}

\begin{rmk}
The hypothesis of a lower bound on Ricci curvature is needed for technical reasons described below, and it should be possible to remove this. The assumption on the finite topological type of $M$ is used to ensure that an $S^1$-action on $M$ necessarily has finitely many fixed points in the same way as in the work of Bielawski~\cite{Bielawski}. It should be noted that the existence statement of Theorem~\ref{t:existence} holds even for an infinite number of poles $\{z_\alpha\}\subset N(a_+,a_-)$ as long as $\{z_\alpha\}$ are sparse enough ensuring a uniform convergence of the sum of the corresponding Green's functions (see Proposition~\ref{p:elliptic_aneq0}) on compact sets. This observation is well-known in the hyperK\"ahler case, see e.g.,~\cite{an-kr-le-89}.
\end{rmk}

Let us give a brief description of the proofs of Theorems \ref{t:existence} and \ref{t:uniqueness}, expanding on the discussion above, which will also serve as an outline for the remainder of the paper.  In \S \ref{s:background} we recall fundamental properties and examples of generalized K\"ahler structures, which are biHermitian triples $(g, I, J)$ satisfying certain integrability conditions.  We recall the associated Poisson tensor $\gs = \tfrac{1}{2} g^{-1} [I,J]$, and briefly indicate the geometric significance of the rank of this tensor.  In the nondegeneracy locus we obtain a symplectic form $\Omega = \gs^{-1}$, and recall that in dimension $4$ the entire GK structure is described in terms of $\Omega$, $I \Omega$ and $J \Omega$.  We also recall some fundamental topological and geometric results on Seifert fibrations, specialized to the case of almost free $S^1$ actions on 4-manifolds.  

In \S \ref{s:GKsymm} we study generalized K\"ahler 4-manifolds admitting a free $S^1$ action.  As GK structures on 4-manifolds are described by a triple of symplectic forms, in \S \ref{ss:inv_gk_str} we are able to locally define a moment map generalizing that of the Gibbons-Hawking ansatz (Theorem~\ref{t:nondegenerate_gk_description}).  In the image of the moment map, the horizontal geometry is determined by a single arbitrary function $p$, which determines the angle between the complex structures $I$ and $J$.  The length of the circle fiber is determined by a function $W$ which solves a certain linear elliptic PDE depending on $p$.  Turning to the case of effective $S^1$ actions in \S \ref{ss:NGKeff}, we first derive the local structure of the fibration near fixed points, which must be that of the standard Hopf fibration, then derive the local blowup rate of $W$ near these fixed points.  Given this we prove removable singularities results in Propositions \ref{p:removable_singularity} and \ref{p:removable_singularity_GK} which say roughly that given the data of our ansatz on a punctured $3$-ball satisfying the necessary blowup rates, we can extend the total space of the fibration topologically by adding a point, with the generalized K\"ahler structure extending in $C^{1,1}$ sense across this point. Having established these basic structural results, we end \S \ref{s:GKsymm} by reviewing some known examples, including diagonal Hopf surfaces and a new description of LeBrun's generalized K\"ahler structures on parabolic Inoue surfaces \cite{LeBrun}.

So far the discussion has allowed for the angle function $p$ to be essentially arbitrary.  In \S \ref{s:GKRS} we impose the generalized K\"ahler-Ricci soliton equation, and observe that it determines $p$.  For K\"ahler-Ricci solitons, it is well-known that the associated soliton vector field $\N f$ is real holomorphic, with $J \N f$ a Killing field.  Through a local analysis of the generalized K\"ahler-Ricci soliton equation, we show that aside from some exceptional and rigid cases, either the metric has a two-dimensional isometry group or there is a vector field $X$ preserving all the generalized K\"ahler structure, leading to the definition of a \emph{rank one} soliton (cf.\,Definition \ref{d:GKsoliton}).  We next derive a kind of scalar reduction for the soliton equation, showing that the angle function $p$ must be determined by $\Omega$-Hamiltonian potential functions for the vector fields $I X$ and $JX$, and thus in the rank one case, in the image of the moment map we obtain an explicit formula for $p$ depending on two real parameters $a_{\pm}$.

The description of GK structures and solitons in sections \S \ref{s:GKsymm}--\S\ref{s:GKRS} is purely local and \textit{flexible}. To turn it into a reasonable classification problem, we need to introduce the key assumption of $(M,g)$ being \textit{complete}.
In \S \ref{s:completion} we provide a description of the topological and geometric structure of the possible completions of the quotient space $M/S^1$ for rank one solitons under natural geometric hypotheses on $M$.  To begin we analyze the blowup behavior of the meromorphic $(2,0)$-form $\Omega_I$, showing in particular that it has a pole of order $1$ at the degeneracy locus. Using this we begin our analysis of the global properties of the moment map. We establish a key completeness property for the induced horizontal geometry in \S \ref{ss:compquot}. In particular, as derived earlier, the horizontal metric is naturally of the form $W h$, where $W^{-1}$ is the squared length of the circle fiber.  Completeness of $M$ immediately implies completeness of $W h$, however, a priori we do not know the global structure of $h$.  As discussed above, establishing the completeness of the domain with respect to the metric $h$ plays a key role in determining the global structure of the so-far only locally defined moment map.  For hyperK\"ahler metrics with symmetry, this property was established by Bielawski \cite{Bielawski}, and plays a key role in proving their classification.  Bielawski's proof exploits structural results of Schoen-Yau \cite{SchoenYau} on conformal immersions of manifolds of nonnegative scalar curvature, relying in a delicate way on the fact that $h$ is flat, which holds in the hyperK\"ahler setting.  In our setup the metric $h$ is fairly explicit, but not flat, and thus we must rely on more robust elliptic theory.  We establish the relevant completeness property in Proposition \ref{p:s1_bundle_complete}, relying on the gradient estimate of Cheng-Yau \cite{ChengYau} and certain ideas from the previously mentioned work of Schoen-Yau \cite{SchoenYau}.  These arguments are where the technical condition of a lower Ricci curvature bound and finite topological type enter.

With this completeness property at hand we can analyze the global properties of the moment map.  In particular, a priori the moment map is defined only after we remove the degeneracy loci of $\gs$, and take an appropriate $\mathbb Z^k$ cover to render the $S^1$ action tri-Hamiltonian.  Using the structure of $\Omega_I$ and our technical result on the completeness of the horizontal geometry, in Proposition \ref{p:deck_transform_mu} we analyze the $\mathbb Z^k$ covering action, and show that the moment map is naturally defined on the $\mathbb Z^k$ quotient, with the image space an explicit quotient of $\mathbb R^3_{\pmb \mu}$.  Next in \S \ref{ss:metcomp} we explicitly describe the completions of these $\mathbb R^3_{\pmb \mu}$ quotients.  Specifically we observe that by adding the degeneracy loci of $\gs$, the metric extends smoothly in the orbifold sense, with cone angles determined by the parameters $a_{\pm}$.  Given this, we finish in \S \ref{ss:globalmoment} by proving that the moment map is globally defined on $M$, with the images given by one of the spaces $N(a_+, 0), N(a_+, 0) / \mathbb Z$, or $N(a_+, a_-)$ described above.

Having now determined the possible geometry and topology of the images of the moment map, it remains to describe the possible choices of $W$.  In \S \ref{s:construction} we first observe that solutions to the linear elliptic PDE for $W$ are in one-to-one correspondence with solutions to the Laplace equation for a metric conformal to the given background $h$.  Using this interpretation together with a description of the spaces $N$ as global orbifold quotients, we show that the relevant solutions must be a linear combination of Green's functions, the constant function, and one exceptional solution in the case $a_- \neq 0$.  Having established the characterization of entire solutions $W$ to the relevant PDE, we finish the proof of Theorem \ref{t:existence}. At this point, all that remains is to use the soliton equation to (1) show that the GK structure, which is only known to be $C^{1,1}$ at the poles of $W$, in fact extends smoothly, and (2) prove that the whole structure extends smoothly over the degeneracy loci.  Finally, relying primarily on the prior classification of the moment map images and possible functions $W$, we finish the proof of Theorem \ref{t:uniqueness}.

\vskip 0.1in

\textbf{Acknowledgements:} The authors are grateful to anonymous referees for useful remarks and suggestions.
We would also like to thank Vestislav Apostolov and Connor Mooney for helpful discussions.  The first author acknowledges support from the NSF via DMS-1454854.
\section{Background}\label{s:background}

The aim of this section is to provide the necessary background for the tools and notions used in this paper. In \S \ref{ss:GK_structures} we give a very brief overview of the biHermitian interpretation of generalized K\"ahler geometry with an emphasis on the $4$-dimensional case. We discuss the relation to Poisson geometry, and give an explicit pointwise description of a \emph{nondegenerate} generalized K\"ahler structure in terms of a preferred vector and \emph{angle function} $p$. We refer the reader to~\cite{ASNDGKCY} and references therein for more details.
In \S \ref{ss:s1_actions} we discuss the topology of $S^1$ actions on 4-dimensional manifolds, and review the notions of Seifert fibrations and orbifolds.

\subsection{Generalized K\"ahler structures}\label{ss:GK_structures}

Let $M$ be a smooth $2n$-dimensional manifold equipped with a Riemannian metric $g$ and two integrable almost complex structures $I$ and $J$ compatible with $g$. Denote by
\[
\gw_I:=g(I\cdot,\cdot)\quad \mathrm{and}\quad \gw_J:=g(J\cdot,\cdot)
\]
the associated fundamental 2-forms, and let
\[
d^c_I\colon \Lambda^k(M; \R)\to \Lambda^{k+1}(M; \R),\quad 
d^c_J\colon \Lambda^k(M; \R)\to \Lambda^{k+1}(M; \R)
\]
be the real differential operators, also known as \textit{twisted differentials}
\[
d^c_I=\sqrt{-1}(\bar\del_{I}-\del_I),\quad d^c_J=\sqrt{-1}(\bar\del_{J}-\del_J).
\]

\begin{defn} \label{def:gk structure}
	The data $(M,g,I,J)$ is called a \emph{generalized K\"ahler} (GK) structure on a manifold $M$, if
	\[
	d_I^c\gw_I=-d_J^c\gw_J
	\]
	and the form $H:=- d_I^c\gw_I$ is closed:
	\[
	dH=0.
	\]
	In particular, the Hermitian manifolds $(M,g,I)$ and $(M,g,J)$ are \emph{pluriclosed} (sometimes also called SKT~--- \emph{strong K\"ahler with torsion}):
	\[
	dd^c_I\gw_I=dd^c_J\gw_J=0.
	\]
\end{defn}
For a four-dimensional GK structure we furthermore define the two associated Lee forms $\theta_I$ and $\theta_J$ by the equations
\begin{align*}
d \gw_I = \theta_I \wedge \gw_I, \qquad d \gw_J = \theta_J \wedge \gw_J.
\end{align*}
Equivalently, in the four dimensional case, $\theta_I=-(*H)$. Thus in this case, when $I$ and $J$ induce the same orientation the generalized K\"ahler condition implies that $\theta_I=-\theta_J$. 
In general there are two natural connections associated to a GK structure, 
\begin{align*}
\N^{I} = D + \tfrac{1}{2} g^{-1} H, \qquad \N^{J} = D - \tfrac{1}{2} g^{-1} H
\end{align*}
where $D$ denotes the Levi-Civita connection.  These are the Bismut connections associated to the Hermitian structures $(g, I)$ and $(g, J)$ \cite{Bismut}.  Call the associated curvature tensors $R_I, R_J$ and then we obtain the associated Bismut-Ricci tensors
\begin{align*}
\rho_I = \tfrac{1}{2} \tr_{\gw_I} R_I, \qquad \rho_J = \tfrac{1}{2} \tr_{\gw_J} R_J.
\end{align*}
In \cite{PCF, GKRF} the first named author and Tian introduced a parabolic flow of generalized K\"ahler structures (generalized K\"ahler-Ricci flow) which can be expressed as
\begin{align*}
\dt \gw_I =&\ - \rho_I^{1,1}, \qquad \dt J = L_{\tfrac{1}{2} \left(\theta_J^{\#} - \theta_I^{\#}\right) } J.
\end{align*}
The equation for $\gw_I$ is the pluriclosed flow, which is defined for pluriclosed structures $(g, I)$.  With generalized K\"ahler initial data $(g, I, J)$ and the evolution equation for $J$ above, the flow preserves the GK conditions.  It turns out that pluriclosed flow is a gradient flow for a Perelman-type energy functional \cite{PCFReg}, and this energy is fixed on steady solitons:
\begin{defn}\label{d:soliton} A pluriclosed structure $(M^{2n}, g, I)$ is a \emph{steady soliton} if there exists a function $f\in C^\infty(M,\R)$ such that
\begin{align*}
\Rc - \tfrac{1}{4} H^2 + \N^2 f =&\ 0,\\
d^* H + i_{\N f} H =&\ 0.
\end{align*}
\end{defn}
It follows from \cite{PCFReg} that these equations are satisfied by a solution to pluriclosed flow evolving purely by diffeomorphisms.

\subsubsection{Poisson geometry} \label{ss:Poisson}

Hitchin/Pontecorvo observed~\cite{HitchinPoisson, PontecorvoCS} that there is a real Poisson tensor associated to any generalized K\"ahler structure:
\begin{equation*}
\gs = \tfrac{1}{2}g^{-1}[I, J].
\end{equation*}
\begin{rmk}
	In order to match the common conventions in hyperK\"ahler geometry, our definition of the Poisson tensor associated to a GK structure differs by a factor of $2$ from the one in~\cite{ASNDGKCY}. This discrepancy will affect some of the general identities from~\cite{ASNDGKCY} which we will use in the present paper.
\end{rmk}
We call the generalized K\"ahler structure \emph{commuting} if $\gs \equiv 0$, and \emph{nondegenerate} if $\gs$ defines a nondegenerate pairing.  In dimension $4$ these are the only two possibilities at a given point.  In general the rank of $\gs$ can vary at different points.  Fix $(M,g,I,J)$ a generalized K\"ahler manifold with nondegenerate Poisson tensor, and define symplectic form
\begin{equation*}
\Omega = \gs^{-1}=2[I,J]^{-1}g,
\end{equation*}
It is elementary to show that $\Omega$ is of holomorphic type $(2,0) + (0,2)$ with respect to both $I$ and $J$, and thus we can define the (respectively $I$- and $J$-) $(2,0)$-type holomorphic symplectic forms
\begin{align*}
\Omega_I:=\Omega-\sqrt{-1}I\Omega,\quad \Omega_J:=\Omega-\sqrt{-1}J\Omega,
\end{align*}
where $(I\Omega)(X,Y):=\Omega(IX,Y)$.

From here forward we will be dealing with 4-dimensional GK manifolds with generically nonzero Poisson tensor $\gs$. In this case $\gs$ is invertible on the complement of the locus
\[
\mathbf{T}=\{x\in M\ |\ I_x=\pm J_x\}.
\]
The set $\mathbf{T}$ is analytic with respect to both $I$ and $J$ and has complex dimension one. By definition the GK structure $(M\backslash\mathbf{T}, g, I, J)$ on the complement of $\mathbf{T}$ is nondegenerate.  Furthermore, in this nondegeneracy locus, one can recover $I$ and $J$ from a pair of compatible complex-valued symplectic forms. Specifically we have the following:
\begin{prop}[{\cite[Thm.\,2]{AGG} and \cite[Lemma\,2.14]{ASNDGKCY}}]\label{p:holo_symplectic-acs}
	Given a 4-dimensional manifold $M$ and two closed complex-valued forms $\Omega_I,\Omega_J\in\Lambda^2(M,\C)$ satisfying
	\[
	\begin{split}
	\Omega_I^2=\Omega_J^2=0,\\
	\Re(\Omega_I)=\Re(\Omega_J),
	\end{split}
	\]
	there are two unique integrable complex structures $I$ and $J$ such that $\Omega_I$ and $\Omega_J$ are holomorphic $(2,0)$-forms with respect to $I$ and $J$ respectively.  Furthermore, if the symmetric pairing $g$ defined by
	\[
	g(X,X):=\Omega(JX,IX)
	\]
	is positive definite, then $(M,g,I,J)$ is a nondegenerate GK structure.
\end{prop}

\begin{ex}[Generalized K\"ahler structures on the standard Hopf surface]\label{ex:Hopf1}
Consider the standard Hopf surface
\begin{align*}
M^4 = \left(\mathbb C^2 - \{0\} \right) / \left< (z_1,z_2) \to (2 z_1, 2 z_2) \right> \cong S^3 \times S^1.
\end{align*}
Let $I$ denote the natural complex structure inherited from $\mathbb C^2$, and consider the \emph{round} metric
\begin{align*}
g=\Re\left(\frac{dz_1\otimes d\bar{z_1}+dz_2\otimes d\bar{z_2}}{|z_1|^2+|z_2|^2}\right).
\end{align*}
This is the natural cylindrical metric on $S^3 \times \mathbb R$, which is Hermitian, and descends to the quotient to give a direct sum of the round metric on $S^3$ with the standard metric on $S^1$.  This is a pluriclosed structure, and we obtain a second complex structure $J$ by pulling it back via an orientation-preserving involution: $J:=j^*I$, where
\begin{equation*}
\begin{split}
j&\colon \C^2\backslash\{0\}\to \C^2\backslash\{0\}\\
j&(z_1,z_2)=\left(\frac{\bar{z_2}}{|z_1|^2+|z_2|^2}, \frac{z_1}{|z_1|^2+|z_2|^2}\right).
\end{split}
\end{equation*}
The corresponding holomorphic symplectic forms are
\begin{align*}
\Omega_I =&\
\frac{dz_1}{z_1} \wedge \frac{dz_2}{z_2},\\
\Omega_J =&\ \left(d\log(\brs{z_1}^2+\brs{z_2}^2)-\frac{d \bar z_2}{\bar z_2}\right)\wedge\left(d\log(\brs{z_1}^2+\brs{z_2}^2)-\frac{d z_1}{z_1}\right)
\end{align*}
and it is straightforward to check that $\Re(\Omega_I)=\Re(\Omega_J)$ and $\Omega(JX,IX)$ is given by metric $g$. Thus by Proposition~\ref{p:holo_symplectic-acs}, $(M,g,I,J)$ is a GK structure. In particular, Poisson tensor is given by a real part of a holomorphic Poisson tensor:
\begin{align*}
\gs = -\Re\left(z_1 \frac{\del}{\del z_1} \wedge z_2\frac{\del}{\del z_2}\right).
\end{align*}
The degeneracy loci are the elliptic curves $\{z_1 = 0\}$, $\{z_2= 0\}$ where $I = - J$ and $I = J$, respectively. A schematic picture of the fundamental domain in $\mathbb C^2\backslash\{0\}$ is given in Figure \ref{fig:Hopf}.
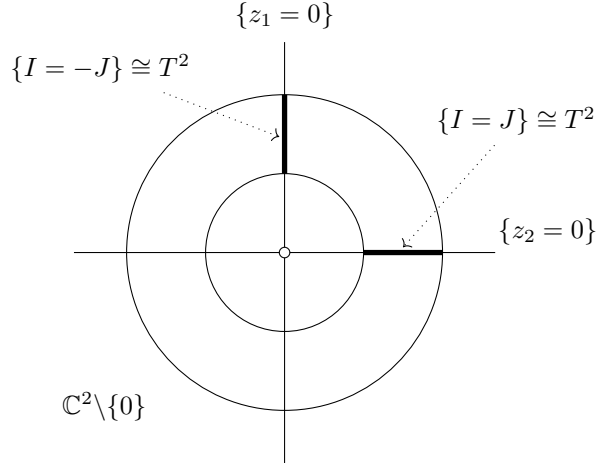
\begin{figure}
\begin{tikzpicture} [scale=0.7]

\draw (0,-4) edge (0,4);
\draw (-4,0) edge (4,0);

\draw[line width=2pt] (3,0) edge (1.5,0);
\draw[line width=2pt] (0,3) edge (0,1.5);

\draw (0,0) circle (1.5);
\draw (0,0) circle (3);

\node at (0,4.5) {\small{$\{z_1 = 0\}$}};
\node at (5,0.4) {\small{$\{z_2 = 0\}$}};

\node (v1) at (-3.5,3.5) {\small{$\{I = -J\} \cong T^2$}};
\draw [dotted,->] (v1) edge (-0.1,2.20);

\node (v2) at (4.4,2.5) {\small{$\{I = J\} \cong T^2$}};
\draw [dotted,->] (v2) edge (2.25,0.1);

\draw[fill=white] (0,0) circle (0.1);

\node at (-1.5,-1,5) {\small{$\C^2\backslash\{0\}$}};

\end{tikzpicture}
\caption{Fundamental domain of standard Hopf surface with GK structure}
\label{fig:Hopf}
\end{figure}

The above GK structure $(M,g,I,J)$ on a standard Hopf surface has a large automorphism group: it admits an effective action of a 3-dimensional torus $(S^1)^3$, where the first two factors act via coordinate-wise complex multiplication, and the last factor $\R/\Z\simeq S^1$ acts by dilations: $t\cdot(z_1,z_2)=(e^tz_1,e^tz_2)$.
\end{ex}

\subsubsection{GK structures with a preferred vector}

Next we show that, given a preferred vector $Z\in T_xM$, it is possible to give an explicit description of a 4-dimensional GK structure in terms of one function.  The key linear algebraic fact special to four dimensions is the pointwise identity~\cite[Lemma 3.2]{StreetsNDGKS}
\begin{equation*}
IJ+JI=-2p\Id,
\end{equation*}
where $p\colon M\to(-1,1)$, $p=-\frac{1}{4}\tr(IJ)$ is the \emph{angle function}. We also have a third natural endomorphism, which we can rewrite using $p$:
\begin{equation}\label{f:K}
K:=\tfrac{1}{2}[I,J]=IJ+p\Id=-JI-p\Id.
\end{equation}
Note that in the hyperK\"ahler setting, the function $p$ is identically zero and $K$ is another integrable complex structure. More generally, if $p\colon M\to(-1,1)$ is constant, then $I$ is a part of a hyperK\"ahler triple $(I,J',K')$ with $J'=\sqrt{1-p^2}I+pJ$, $K'=[I,J']$.

Thus, in the presence of a preferred vector $Z\in T_xM$, we can use the operators $I,J$ and $K$ to obtain a basis of $T_xM$, thus yielding a complete description of all structures in terms of the function~$p$.  Specifically we have the following proposition:

\begin{prop}[Linear description of a GK structure on $M^4$] \label{p:matrices}
	For $(M^4,g,I,J)$ as above, given a unit vector $Z\in TM$, the vectors $Z,IZ,JZ,KZ$ form a basis of $TM$ with $g,I,J,K$ and $\Omega$ given by the matrices
	\begin{equation}\label{f:g_matrix}
	g=\left[\begin{matrix}1 & 0 & 0 & 0\\0 & 1 & p & 0\\0 & p & 1 & 0\\0 & 0 & 0 & 1 - p^{2}\end{matrix}\right],
	\end{equation}
	\begin{equation}\label{f:I_matrix}
	I=\left[\begin{matrix}0 & -1 & - p & 0\\1 & 0 & 0 & p\\0 & 0 & 0 & -1\\0 & 0 & 1 & 0\end{matrix}\right],
	\end{equation}
	\begin{equation}\label{f:J_matrix}
	J=\left[\begin{matrix}0 & - p & -1 & 0\\0 & 0 & 0 & 1\\1 & 0 & 0 & - p\\0 & -1 & 0 & 0\end{matrix}\right],
	\end{equation}
	\begin{equation}\label{f:K_matrix}
	K=\left[\begin{matrix}0 & 0 & 0 & p^{2} - 1\\0 & - p & -1 & 0\\0 & 1 & p & 0\\1 & 0 & 0 & 0\end{matrix}\right],
	\end{equation}
	\begin{equation}\label{f:Omega_matrix}
	\Omega=\left[\begin{matrix}0 & 0 & 0 & -1\\0 & 0 & -1 & 0\\0 & 1 & 0 & 0\\1 & 0 & 0 & 0\end{matrix}\right].
	\end{equation}
\end{prop}
\begin{proof}
	The formulas for matrices of $g,I,J,K$ are straightforward from the compatibility of the metric with $I$ and $J$ and the identity~\eqref{f:K} defining $K$. Finally, for the matrix of $\Omega$ we have
	\[
	\Omega=(K^{-1})^t\cdot g
	\]
	which is given by the last matrix.
\end{proof}

\subsection{Manifolds with \texorpdfstring{$S^1$}{} action}\label{ss:s1_actions}

	In this subsection we review basic facts about circle actions on manifolds, and discuss the related concepts of \emph{Seifert fibrations} and \emph{orbifolds}. We will usually think of the circle $S^1$ as the group of complex numbers of the unit norm.
	
	\begin{defn}
	Let $M$ be an $n$-dimensional manifold with an effective circle action
	\[
	S^1\times M\to M.
	\]
	Throughout this paper we keep the following notation
	\begin{enumerate}
		\item $\pi\colon M\to N$ is the projection onto the orbit space;
		\item $M^{S^1}\subset M$ is the fixed point set;
		\item $N_0:=(M\backslash M^{S^1})/S^1$ is the orbit space of the complement of the fixed point set.
	\end{enumerate}
	The action is called \emph{free} if the stabilizer $G_x\subset S^1$ of any point $x\in M$ is trivial. In this case $M\to N$ is the total space of a principal $S^1$ bundle.
	The action $S^1\times M\to M$ is called \emph{almost free} if $M^{S^1}\subset M$ is empty, or, equivalently, the stabilizer of any point in $M$ is finite. In this case we call $M\to N$ a \emph{Seifert fibration}.  The action is called \emph{effective} if no nontrivial element of $S^1$ acts as the identity on $M$.
	\end{defn}
	
	Given any point $x\in M$ there is a representation of its stabilizer $G_x\subset S^1$ in the normal space to the orbit $\nu_x:=T_xM/T_x(S^1x)$:
	\[
	\rho_x\colon G_x\to GL(\nu_x)
	\]
	If $M$ is oriented then the representation $\rho_x$ preserves the induced orientation on $\nu_x$. We have the following equivariant tubular neighbourhood theorem:
	
	\begin{thm}[{Slice theorem \cite[\S B.2]{gu-gi-ka}}] \label{t:slice}
		Let $M$ be a manifold with an effective $S^1$ action, $x\in M$. Then the representation $\rho_x$ in $\nu_x$ is faithful and there exists an $S^1$-invariant neighbourhood $U_x$ of $x$ and an equivariant diffeomorphism
		\[
		U_x\simeq S^1\times_{G_x} \nu_x:=(S^1\times \nu_x)/G_x.
		\]
		We will call $U_x$ a \emph{canonical neighbourhood} of $x\in M$.
	\end{thm}
	
	From now on we restrict our attention to a special case when $M$ is oriented, $\dim M=4$ and the action of $S^1$ has only isolated fixed points $M^{S^1}\subset M$. Note that every isotropy group $G_x$ is abelian, therefore its linear orientation-preserving representation is a sum of irreducible representations in $\C$ and a trivial representation.  Considering separately the situations when $x\in M$ is fixed and when $x$ has a finite isotropy subgroup, we have the following.
	\begin{cor}\label{c:slice}
		Let $M$ be a 4-dimensional manifold with an effective $S^1$ action with isolated fixed points. Given $x\in M$, exactly one of the following holds:
		\begin{enumerate}
			\item $\nu_x\simeq \C\times\R$,  $G_x=\Z_k\subset S^1$ and $\rho_x$ is given by
			\[
			\rho_x\colon \Z_k\to GL(\C)\simeq \C^*,\quad
			\rho_x\colon \exp\frac{2\pi\sqrt{-1}}{k}\mapsto \exp\frac{2\pi\sqrt{-1}l}{k},\quad \mathrm{gcd}(k,l)=1;
			\]
			\item $\nu_x\simeq \C^2$, $G_x=S^1$ and $\rho_x$ is given by
			\[
			\rho_x\colon S^1\to GL(\C^{2}),\quad \rho_x\colon z\mapsto \mathrm{diag}(z^{w_1},z^{w_{2}})
			\]
			where $\mathbf{w}=(w_1,w_2)\in \Z^{2}$ is the weight vector such that $w_1,w_2\neq 0$ and $\mathrm{gcd}(w_1,w_2)=1$.
		\end{enumerate}
	\end{cor}
	
	\begin{rmk}\label{r:local_s1_quotient}
		If $G_x\simeq \Z_k$ and $\rho_x$ is given by $l\in\Z_k^{\times}$, we will say that the point $x\in M$ is of type $(k,l)$. If we change the orientation of $\C$ in the identification $\nu_x\simeq \C\times \R$, then $l$ is substituted by $k-l$. In particular, if $M$ admits an $S^1$-invariant complex structure, $l$ is well-defined.		
		If $k=1$, then the stabilizer $G_x$ is trivial and we will say that point $x$ is \emph{regular}. Clearly the type is constant along the $S^1$ orbit, so we can also talk about the type of a point $[x]\in M/S^1$.
	\end{rmk}
	
	In general, when the action of $S^1$ on $M$ is not free, the orbit space $N=M/S^1$ is not a manifold. However, in our case when $\dim M=4$ and there are only isolated fixed points,
	thanks to the slice theorem we have a good understanding of the geometry and local structure of $N$:
	\begin{enumerate}
		\item The set of regular points $M^{\mathrm{reg}}\subset M$ is open and dense, and the quotient $M^{\mathrm{reg}}/S^1$ is a smooth manifold, such that $M^{\mathrm{reg}}\subset M$ is the total space of a principal $S^1$-bundle over $N^{\mathrm{reg}}:=M^{\mathrm{reg}}/S^1$;
		\item If $x\in M$ is a point of type $(k,l)$, then a neighbourhood of $[x]\in N$ is naturally isomorphic to $\C/\Z_k\times\R$. The latter can be identified with $\C\times\R$ via the map $\C\to\C$, $z\mapsto z^k$, providing a smooth structure near $[x]\in N$;
		\item Given a fixed point $x\in M$ with weights $\mathbf{w}=(w_1,w_2)$ there are two disjoint subsets of points of types $(w_1,w_2)$ and $(w_2,w_1)$ in a punctured canonical neighbourhood $U_x$ of $x$. These subsets are given by complex coordinate planes in $U_x\simeq \C^2$;
		\item If a fixed point $x\in M$ has weights $\mathbf w=(\pm 1,\pm 1)$, then there is a neighbourhood $U_x$ such that the projection $U_x\to U_x/S^1$ is equivalent to the Hopf fibration projection $\C^2\to \R^3$. In particular there is a natural smooth structure in a neighbourhood of $[x]\in N$;
	\end{enumerate}
	
	Consider the complement of the fixed point set $M_0:=M\backslash M^{S^1}$. The action $S^1\times M_0\to M_0$ is almost free, and $M_0$ is a Seifert fibration. In this case the orbit space $N_0$ comes equipped with an extra structure:
	\begin{enumerate}
		\item The orbit space $N_0$ inherits a natural \emph{cyclic orbifold} structure. That is, given a point $x\in N_0$, its open neighbourhood $V_x$ can be realized as $V_x\simeq \nu_x/G_x\simeq \C/\Z_k\times \R$ and the triples $(\nu_x, G_x, V_x)$ satisfy the usual gluing compatibility properties~\cite[\S 1.3]{ad-le-ru-07}.
		We will call a triple $(\nu_x,G_x,V_x)$ together with an identification $\nu_x/G_x\simeq V_x$ an \emph{orbifold chart} at a point $x\in N_0$. Note that in general an orbifold chart $\nu_x$ is not naturally embedded into the orbifold $N_0$.
		\item Additionally, since in our case the orbifold structure on $N_0$ is very special, the space $N_0$ admits a canonical smoothing: every quotient $\C/\Z_k$ can be identified with $\C$. With this identification, the set of points of any given type $(k,l)$ forms a smooth submanifold $Z\subset N_0$ of codimension~2.
	\end{enumerate}
	
	If the action $S^1\times M\to M$ is free, then we can recover $M$ from the orbit space $N=M/S^1$ and a \emph{characteristic class}
	\[
	e(M\to N)\in H^2(N,\Z)
	\]
	representing the Euler class of the bundle $M\to N$. Specifically, we have the following standard topological characterization of principal bundles:
	\begin{prop}\label{p:s1_bundle_top}
		Given a topological space $N$ and a class $e\in H^2(N,\Z)$ there exists a unique (up to equivalence) principal $S^1$ bundle $M\to N$ such $e(M\to N)=e$.
	\end{prop}
	While the above proposition holds if $N$ is merely a topological space, if $N$ is a manifold there is a differential-geometric refinement.
	\begin{prop}\label{p:s1_bundle_dg}
		Given any manifold $N$ and a closed form $\beta\in \Lambda^2(N,\R)$ such that $\beta/2\pi$ represents a class $e\in H^2(M,\Z)$, let $\pi\colon M\to N$ be the principal $S^1$ bundle corresponding to the class $e$. Then there exists a principal connection $\eta$ in $M$ such that its curvature form is $\beta$. If $\eta'$ is another connection with the same curvature, then
		\[
		\eta-\eta'=\pi^* \alpha
		\]
		where $\alpha\in \Lambda^1(N,\R)$ is a closed one-form.
	\end{prop}
	Any function $f\in C^\infty(N,S^1)\simeq C^\infty(N,\R)\times H^1(N,\Z)$ defines a gauge equivalence on $M$ such that $\eta \sim \eta+\pi^*df$. Hence, modulo gauge equivalence, the connection $\eta$ is defined up to an element in $H^1(N,\R)/H^1(N,\Z)$.
	
	Remarkably, the above classical results can be generalized to the case of almost free actions and Seifert fibrations. To formulate the generalization of Proposition~\ref{p:s1_bundle_top} we need to extend the definition of the Euler class to Seifert fibrations.
	
	\begin{defn}[Euler class]
		Let $M\to N$ be a Seifert fibration with points of types $\{(k_i,l_i)\}$. Define $k:=\mathrm{lcm}(\{k_i\})$. Then there is a free action of $S^1/\Z_k$ on the quotient space $M/\Z_k$ with the orbit space $N$. Then we can define the usual Euler class $e(M/\Z_k\to N)\in H^2(N,\Z)$ of the principal bundle $M/\Z_k\to N$. The \emph{Euler class} of the fibration $M\to N$ is
		\[
		e(M\to N):=\frac{1}{k}e(M/\Z_k\to N)\in H^2(N,\mathbb Q).
		\]
	\end{defn}
	
	If $M\to N$ is a Seifert fibration, then points of its orbit space $N$ can be marked with a type $(k,l)$ according to the local structure provided by the slice theorem (see Remark~\ref{r:local_s1_quotient}). It turns out that, conversely, given a manifold $N$ with marked submanifolds of codimension two and a cohomology class $e\in H^2(N,\Z)$, we can recover the Seifert fibration (uniquely if $H^2(N,\Z)$ has no torsion). Namely, we have the following specialization of {\cite[Prop.30]{ko-05}} to the case $\dim N=3$.
	\begin{prop}\label{p:s1_bundle_seifert_t}
		Consider a manifold $N$, $\dim N=3$ and $\{Z_i\subset N\}$ a collection of cooriented pairwise non-intersecting one-dimensional submanifolds. Let $1\leq l_i<k_i$ be integers, $\mathrm{gcd}(l_i,k_i)=1$ and $e\in H^2(N,\Z)$.
		
		Then there exists a Seifert fiber bundle $\pi\colon M\to N$ such that $\pi^{-1}(Z_i)\subset M$ is the set of points of type $(k_i,l_i)$, all other points are regular, and
		\[
		e(M\to N)=e+\sum_i \frac{l_i}{k_i}[Z_i]\in H^2(N,\mathbb Q).
		\]
		If $H^2(N,\Z)$ has no torsion, $M$ is uniquely defined.
	\end{prop}
	
	\begin{rmk}
		Proposition~\ref{p:s1_bundle_seifert_t} states that the Seifert fibration $M\to N$ can be recovered from the marked orbit space $(N,\{Z_i,(k_i,l_i)\})$ and a cohomology class $e\in H^2(N,\Z)$. If we realize $N$ as the orbit space $M/S^1$, then we can think of $N$ as an orbifold with orbifold charts $V_x$ of the form $\C\times\R\to\C/\Z_{k_i}\times\R\simeq V_x$, where $x\in Z_i$ and $Z_i\cap V_x=\{0\}\times \R$. This orbifold structure depends only on the underlying smooth manifold $N$ and the collection $\{Z_i, k_i\}$.
	\end{rmk}
	
	Our next goal is to formulate and prove an analogue of Proposition~\ref{p:s1_bundle_dg} for Seifert fibrations. To do that, we first need to define smooth tensors on orbifolds and introduce the notions of connections and basic differential forms in Seifert fibrations.  
	
	\begin{defn}[Smooth tensors on orbifolds]
	Let $N$ be a smooth 3-dimensional manifold, and $\{Z_i\subset N\}$ a collection of pairwise non-intersecting one-dimensional manifolds, together with integers $\{k_i\}$, defining an orbifold structure $N_{\mathrm{orb}}$ on $N$.  A \emph{smooth tensor field} $T$ on the orbifold $N_{\mathrm{orb}}$, is a collection of tensors $T_\alpha$ in the orbifold charts $(W_\alpha, G_\alpha, V_\alpha)$ such that $T_\alpha$ is a smooth $G_\alpha$-invariant tensor in $V_\alpha$, compatible with the gluing maps. In particular, $T$ defines a smooth tensor field in the usual sense on the regular open dense part $N^{\mathrm{reg}}=N\backslash \cup_i Z_i$.
	\end{defn}
	
	There is a natural one-to-one correspondence between the set of smooth Riemannian metrics on the orbifold $N_{\mathrm{orb}}$ as above and the set of Riemannian metrics on the underlying manifold $N$ with cone singularities along $\{Z_i\}$ of angles $2\pi/k_i$.
	
	\begin{defn}[Basic differential forms]
		Let $X\in \Gamma(M,TM)$ be the vector field generating an almost free $S^1$ action on $M$. A differential form $\alpha\in \Lambda^q(M,\R)$ is \emph{basic} with respect to the vector field $X$ if
		\[
		i_X\alpha=\mathcal L_X\alpha=0.
		\]
		We will denote the space of basic $q$-forms by $\Lambda^q_b(M,\R)$. The de Rham differential preserves the subcomplex $\Lambda^*_b(M,\R)\subset \Lambda^*(M,\R)$, thus we can define \emph{basic cohomology} $H^*_b(M,\R)$ to be the cohomology of $(\Lambda^*_b(M,\R), d)$.
	\end{defn}
	The basic cohomology ring can be defined for any foliation on a manifold $M$. However, in our case the vector field $X$ generates an almost free $S^1$ action, and there is a natural isomorphism
	\[
	\pi^*\colon \Lambda^q(N_{\mathrm{orb}},\R)\to \Lambda^q_b(M,\R).
	\]
	The de Rham cohomology of $(\Lambda^q(N_{\mathrm{orb}},\R), d)$ is isomorphic to the usual singular cohomology $H^*(N,\R)$. Combining this observation with a general result of Koszul \cite{kosz} we see that map $\pi^*$ induces an isomorphism
	\begin{equation*}
	\pi^*\colon H^*(N,\R)\to H_b^*(M,\R).
	\end{equation*}
	
	\begin{defn}
		A \emph{connection} on a Seifert bundle $\pi\colon M\to N$ is an $S^1$-invariant differential form $\eta\in\Lambda^1(M,\R)$ such that $\eta(X)=1$. The \emph{curvature} of a connection $\eta$ is the basic differential form $\beta:=d\eta$.
	\end{defn}
	Clearly, the difference of two connections is a basic one-form. Furthermore, we can always construct a connection $\eta$ by choosing an $S^1$-invariant Riemannian metric on $M$, and taking the orthogonal projection $TM\to \R\cdot X$ onto the vertical space spanned by $X$. Therefore, the set of connections on $M$ is an affine space modeled on $\Lambda^1_b(N,\R)$. The orbifold version of Chern-Weil theory implies that the closed basic differential form  $\beta/2\pi$ represents the Euler class $e(M\to N)$ in $H^2_b(M,\R)\simeq H^2(N,\R)$.
	
	\begin{prop}\label{p:s1_bundle_seifert_dg}
		Let $N$ be a smooth 3-dimensional manifold, and $\{Z_i\subset N\}$ a collection of cooriented pairwise non-intersecting one-dimensional manifolds. Let $1\leq l_i<k_i$ be integers, $\mathrm{gcd}(k_i, l_i)=1$ and $N_{\mathrm{orb}}$ be the corresponding orbifold. Then given a closed form $\beta\in \Lambda^2(N_{\mathrm{orb}},\R)$ such that
		\[
		[\beta/2\pi]=e+\sum_i \frac{l_i}{k_i}[Z_i]\in H^2(N,\mathbb Q)
		\]
		where $e\in H^2(N,\Z)$ is an integral class, there exists a Seifert fibration $\pi\colon M\to N$ with a connection $\eta$ and curvature $d\eta=\beta$.
		
		The Seifert fibration $\pi\colon M\to N$ is unique if $H^2(N,\Z)$ has no torsion. Modulo gauge equivalence, the connection $\eta$ is unique if $H^1(N,\R)=0$.
	\end{prop}
	\begin{proof}
		First, let us use Proposition~\ref{p:s1_bundle_seifert_t}, to construct a Seifert fibration $\pi\colon M\to N$ associated to the data $(N,\{Z_i,k_i,l_i\},e)$. Let $\eta'$ be any connection in the fibration, and denote by $\beta':=d\eta'$ its curvature form.  The form $\beta'/2\pi$ represents the class $e(M\to N)=e+\sum_{i}l_i/k_i[Z_i]\in H^2(N,\mathbb Q)$. Then the basic forms $\beta$ and $\beta'$ represent the same cohomology class in $H^2_b(M,\R)\simeq H(N,\R)$. Therefore, there exists a basic 1-form $\alpha\in \Lambda^1_b(M,\R)$ such that
		\[
		\beta=\beta'+d\alpha.
		\]
		Then $\eta:=\eta'+\alpha$ is the required connection with curvature $\beta$.
		
		Uniqueness of $N$ follows from Proposition~\ref{p:s1_bundle_seifert_t}. If $\eta_0$ is another connection with curvature $\beta$, then $\eta_0=\eta+\alpha$, where $\alpha$ is a closed basic 1-form. Since $H^1(N,\R)=0$, we have $\alpha=df$ where $f$ is a smooth function on the orbifold $N_{\mathrm{orb}}$. The function $f\in C^\infty(N_{\mathrm{orb}},\R)$ generates a gauge equivalence on $M\to N$, transforming $\eta$ into $\eta_0$.
	\end{proof}

\section{Generalized K\"ahler structures with \texorpdfstring{$S^1$}{} symmetries} \label{s:GKsymm}

In this section we will study generalized K\"ahler structures $(M^4,g,I,J)$ with a generically nonzero Poisson tensor $\gs$, admitting an $S^1$ symmetry generated by a vector field $X$.  First, we focus on the nondegeneracy locus, where the GK structure is recovered by a symplectic triple as in Proposition \ref{p:holo_symplectic-acs}.  Where the action is free, it is locally tri-Hamiltonian, and we define the moment map $\pmb{\mu}$, which is a local submersion onto a domain in $\mathbb R^3_{\pmb \mu}$.  The GK structure defines a metric $h$ on this domain, determined explicitly in terms of the angle function $p$, and furthermore the original GK structure is recovered in terms of $h$ and $W^{-1} = g(X,X)$, yielding Theorem \ref{t:mainGH} (cf. Theorem \ref{t:nondegenerate_gk_description}).

We then turn to addressing the case of an effective $S^1$ action.  We derive the necessary asymptotics for $W$ at a fixed point of the action, then prove a removable singularity result, showing that a real analytic metric in our ansatz, defined on a punctured ball and satisfying the necessary asymptotics, will extend in a $C^{1,1}$ way across the missing point.  Using the explicit form of the torsion we are able to then show that the entire GK structure extends in $C^{1,1}$ sense across the puncture.  This leads to Theorem \ref{t:nondegenerate_gk_description_v2}, a version of Theorem \ref{t:nondegenerate_gk_description} which allows for fixed points of the action.

We finish this section with a discussion of some fundamental examples.  First we recast the GK structures on the standard Hopf surface in this language, then do the same for the examples of GK structures on diagonal Hopf surfaces constructed in \cite{SU}.  Lastly we recover the anti-self-dual metrics of LeBrun on parabolic Inoue surfaces~\cite{LeBrun} in this ansatz.  This requires two key further steps, namely showing that the extensions across fixed points of the action are not just $C^{1,1}$ but smooth, and also extending across the degeneracy loci for the associated Poisson tensor.  These points were treated explicitly by LeBrun, and we briefly indicate how to translate his setup to ours.  We note that we are not able to give a general answer to these two points, namely the smooth extension across fixed points and the gluing in of degeneracy loci.  However, starting in the next section, we will see that in the case of \emph{solitons}, we have further structure which allows for a complete answer to these questions.

\subsection{Nondegenerate GK structures with free \texorpdfstring{$S^1$}{} action}\label{ss:inv_gk_str}

Assume that a nondegenerate GK manifold $(M^4,g,I,J)$ admits a \emph{free} $S^1$-action preserving $g$, $I$, and $J$ and let $X$ be the vector field generating this action. In this case we have a smooth orbit space $N=M/S^{1}$ and $M$ is realized as a total space of a principal $S^1$-bundle:
\begin{equation}\label{f:pi_def}
\pi\colon M\to M/S^1 = N.
\end{equation}
Our first aim in this section is to give a local description for $(M,g,I,J)$ in terms of two scalar functions on the orbit space $N$. To this end we start with the construction of a moment map $\pmb\mu\colon M\to \R^3$ generalizing the moment map of the Gibbons-Hawking ansatz \cite{GibbonsHawking} in hyperK\"ahler geometry. 

The vector field $X$ preserves the symplectic forms $\Omega$, $I\Omega$ and $J\Omega$, hence we have:
\[
d(i_X\Omega)=d(i_XI\Omega)=d(i_XJ\Omega)=0.
\]
If we furthermore assume the vanishing of the corresponding cohomology classes:
\begin{align*}
[i_X\Omega] = [i_XI\Omega] = [i_XJ\Omega] = 0 \in H^1_{dR}(M),
\end{align*}
then the vector field $X$ is \emph{Hamiltonian} with respect to each of three symplectic forms, and we can define the momentum map
\[
\pmb\mu\colon M\to\R^3,\quad \pmb\mu=(\mu_1,\mu_2,\mu_3),
\]
such that
\[
d\mu_1=i_X\Omega, \quad d\mu_2=i_X I\Omega, \quad d\mu_3=i_X J\Omega.
\]
In this case we say that the action $S^1\times M\to M$ is \emph{tri-Hamiltonian}. Since the target of the moment map~--- the space $\R^3$ with distinguished coordinates $(\mu_1,\mu_2,\mu_3)$~--- will appear in this paper many times, we reserve for it the special notation $\R^3_{\pmb\mu}$.

If the action of $X$ is not tri-Hamiltonian, i.e., some of the forms $i_X\Omega, i_XI\Omega$ and $i_XJ\Omega$ are not exact, then one can take an appropriate $\Z^k$-cover $\til M\to M$, such that the pullbacks of these forms are exact and the action of $S^1$ on $M$ lifts to a tri-Hamiltonian action on $\til M$. From now on in this section, we assume that the action of $S^1$ on $M$ is tri-Hamiltonian.

\begin{prop}\label{p:principal_bundle}
	Given a nondegenerate GK structure $(M,g,I,J)$ with a free tri-Hamiltonian $S^1$ action, there exists a local diffeomorphism
	\begin{equation}\label{f:iota_def}
	\iota\colon M/S^1\to \pmb\mu(M)\subset \R^3_{\pmb\mu}
	\end{equation}
	such that 
	\[
	\iota\circ \pi=\pmb\mu.
	\]
\end{prop}
\begin{proof}
	Since $X\in\ker d\pmb\mu$, the moment map is constant on the $S^1$-orbits of $X$, hence the map $\pmb\mu\colon M\to \R^3_{\pmb\mu}$ descends to a map $\iota\colon M/S^1\to \R^3_{\pmb\mu}$.  Since the symplectic forms $\Omega,I\Omega$ and $J\Omega$ are linearly independent, nondegenerate, and $X$ is nonzero, the moment map $\pmb\mu\colon M\to \pmb\mu(M)$ is a submersion with the kernel generated by $X$, therefore $\iota$ must be a local diffeomorphism.
\end{proof}

Consider, in addition to the angle function
\[
p\colon M\to (-1,1),\quad p:=-\tfrac{1}{4}\tr(IJ),
\]
a function $W\colon M\to (0,+\infty)$
\begin{equation}
W:= g(X, X)^{-1}.
\end{equation}
Both functions $p$ and $W$ are invariant under the $S^1$ action and thus descend to the orbit space $M/S^1$. Since the discussion below has a local nature, we will use $\iota\colon M/S^1\to \R^3_{\pmb\mu}$ to identify a neighbourhood of a point in $M/S^1$ with an open subset in $\pmb\mu(M)$. In particular in a neighbourhood of any point $x\in M/S^1$ functions $\{\mu_1,\mu_2,\mu_3\}$ provide local coordinates.

The principal bundle~\eqref{f:pi_def} carries with it a decomposition of the metric and complex structures according to the discussion of \S \ref{s:background}. Given a given point $x\in M$ we have the orthogonal decomposition of the tangent space into the vertical and the horizontal subspaces
\begin{equation}
T_xM=\R X\oplus\mathcal H_x,\quad \mathcal H_x=\langle IX,JX,KX\rangle.
\end{equation}
Thus we can define a connection 1-form $\eta \in \Lambda^{1}(M,\R)$ by setting $\eta(X)=1$ and prescribing its kernel:
\begin{equation}
\eta \colon T_xM\to \R,\quad \Ker\eta=\mathcal{H}_x.
\end{equation}

Now using Proposition \ref{p:matrices} we get explicit descriptions of the metric, almost complex structures and symplectic forms. In particular, using the definition of $W$, by~\eqref{f:g_matrix} the metric in this basis is given by
\begin{equation}\label{f:g_matrix_X}
g=W^{-1}\left[
\begin{matrix}
\w1 & \w0 & \w0 & \w0\\
0 & 1 & p & 0\\
0 & p & 1 & 0\\
0 & 0 & 0 & 1-p^2
\end{matrix}
\right].
\end{equation}
Note that we have two natural coframes
\[
\{X^*,(IX)^*,(JX)^*,(KX)^*\}\quad\mbox{and}
\quad\{\eta,d\mu_1,d\mu_2,d\mu_3\}.
\]
According to the matrix expression~\eqref{f:Omega_matrix} for $\Omega$ we have
\begin{align*}
\Omega=W^{-1}((KX)^*\wedge X^*+(JX)^*\wedge (IX)^*).
\end{align*}
Hence, by the definitions of $d\mu_i$, we find that
\begin{gather} \label{f:coframes}
\begin{split}
\eta&=X^*,\\
d\mu_1&=i_X\Omega=-W^{-1}(KX)^*,\\
d\mu_2&=i_{IX}\Omega=-W^{-1}(JX)^*,\\
d\mu_3&=i_{JX}\Omega=W^{-1}(IX)^*,\\
\end{split}
\end{gather}
Therefore, the metric $g$ can be expressed as
\begin{gather} \label{f:hdef}
\begin{split}
g =&\ W\pi^* h+W^{-1}\eta^2,\\
h=&\ (1-p^2)d\mu_1^2+d\mu_2^2+d\mu_3^2-2pd\mu_2d\mu_3.
\end{split}
\end{gather}

It remains to understand what restrictions on the functions $p$, $W$ and connection 1-form $\eta$ are imposed by the closedness of $\Omega,I\Omega,J\Omega$. We first express these forms in the basis $\{d\mu_1,d\mu_2,d\mu_3,\eta\}$ as
\begin{gather} \label{f:Omegaexplicit}
\begin{split}
\Omega&=W^{-1}\left((KX)^*\wedge X^*+(JX)^*\wedge(IX)^*\right) \\
&=-(d\mu_1\wedge \eta+Wd\mu_2\wedge d\mu_3),\\
\\
I\Omega&=W^{-1}\left(
(JX)^*\wedge X^*+((IX)^*+p(JX)^*)\wedge (KX)^*
\right)\\
&=-d\mu_2\wedge \eta+W(d\mu_3-pd\mu_2)\wedge (-d\mu_1),\\
\\
J\Omega&=W^{-1}\left(
-(IX)^*\wedge X^*+(p(IX^*)+(JX)^*)\wedge (KX)^*
\right)\\
&=-d\mu_3\wedge\eta+W(pd\mu_3-d\mu_2)\wedge (-d\mu_1).
\end{split}
\end{gather}
To determine the necessary condition we first recall that the curvature of the principal $S^1$-connection $\eta$ is a closed \emph{basic} 2-form in $\Lambda_{\mathrm{bas}}^2(M,\R)$, which we identify with a pullback of a closed 2-form $\gb\in\Lambda^2(M/S^{1},\R)$ under projection $\pi\colon M\to M/S^1$. In particular, we have that $\gb$ represents a class in $H^2(M/S^1,2\pi\Z)\subset H^2(M/S^1,\R)$. We omit the pullback in the notation and express
\begin{align} \label{f:betadef}
\beta:=d\eta=\beta_{ij}d\mu_i\wedge d\mu_j.
\end{align}
By taking the exterior derivative of the equations in (\ref{f:Omegaexplicit}), we obtain that in the coframe $d\mu_i$ the components of $\beta$ are given by
\begin{equation}\label{f:star_W}
\begin{split}
\beta_{23}&=W_1,\\
\beta_{31}&=W_2+(pW)_3,\\
\beta_{12}&=W_3+(pW)_2,
\end{split}
\end{equation}
where here and in the sequel $f_i$ means the partial derivative of a function $f$ with respect to $\mu_i$.

Now, if $W$ and $\eta$ satisfy system~\eqref{f:star_W}, then computing $d(d\eta)=0$ we find
\begin{equation}\label{f:W_laplace}
W_{11}+W_{22}+W_{33}+2(pW)_{23}=0.
\end{equation}
Furthermore the form $\gb$ defined by~\eqref{f:star_W} can be expressed as
\begin{equation}\label{f:beta0_def}
\gb=*_hdW+ W \gb_0,\quad \gb_0:=d\mu_1\wedge(p_2d\mu_2-p_3d\mu_3),
\end{equation}
where $*_h$ is computed with respect to the orientation of $\R^3_{\pmb\mu}$ given by $d\mu_1\wedge d\mu_2\wedge d\mu_3$.

The discussion above proves that any nondegenerate GK manifold $(M,g,I,J)$ with a free $S^1$ action has a principal $S^1$-bundle structure $\pi\colon M\to M/S^1$ together with an open map $\iota\colon M/S^1\to \R^3_{\pmb\mu}$ and the metric and complex structure are completely determined by the two scalar functions $W$ and $p$ solving equation~\eqref{f:W_laplace}.  Crucially, the converse is also true: given an open map $\iota\colon N\to \R^3_{\pmb\mu}$, $\dim N=3$, functions $p,W$ on $N$ solving~\eqref{f:W_laplace} and a compatible connection 1-form $\eta$, one can construct a GK structure on the principal $S^1$ bundle $P\to N$ defined by the above data.  Summarizing, we have the following theorem:

\begin{thm}[Generalized-K\"ahler Gibbons-Hawking ansatz]\label{t:nondegenerate_gk_description}
Fix a smooth 3-dimensional manifold $N$ and consider
	\begin{enumerate}
		\item an open map $\iota\colon N\to \R^3_{\pmb\mu}$,
		\item smooth functions
		\[
		p\colon N\to (-1,1),\quad W\colon N\to (0,+\infty)
		\]
		solving the equation
		\begin{equation*}
		W_{11}+W_{22}+W_{33}+2(pW)_{23}=0,
		\end{equation*}
		such that the closed differential form $\beta\in\Lambda^2(N,\R)$
		\[
		\gb=(W_3+(pW)_2)d\mu_1\wedge d\mu_2-(W_2+(pW)_3)d\mu_1\wedge d\mu_3+W_1d\mu_2\wedge d\mu_3
		\]
		represents a class in $H^2(N,2\pi\Z)$,
		\item a connection form $\eta$ with curvature $\beta$ in the principal $S^1$-bundle $\pi\colon M\to N$ determined by $[\beta]$.
	\end{enumerate}
	Then the total space of the principal $S^1$-bundle $M$
	admits a nondegenerate GK structure
	\[
	(M,g,I,J)
	\]
	with
	\[
	g=Wh+W^{-1}\eta^2,\quad h=(1-p^2)d\mu_1^2+d\mu_2^2+d\mu_3^2-2p\,d\mu_2d\mu_3,
	\]
	and $I$, $J$ the unique almost complex structures such that the complex-valued 2-forms
	\[
	\begin{split}
	\Omega_{I}:=&
	(-d\mu_1+\sqrt{-1}d\mu_2)\wedge(\eta+\sqrt{-1}W(d\mu_3-pd\mu_2)),\\
	\Omega_{J}:=&
	(-d\mu_1+\sqrt{-1}d\mu_3)\wedge(\eta+\sqrt{-1}W(-d\mu_2+pd\mu_3)),
	\end{split}
	\]
	are holomorphic with respect to $I$ and $J$ respectively.
	
	Conversely any nondegenerate GK manifold $(M,g,I,J)$ with a free isometric tri-Hamiltonian $S^1$ action arises via this construction.
\end{thm}

\begin{proof}
The forms $\Omega_{I}$ and $\Omega_{J}$ satisfy the assumptions of Proposition~\ref{p:holo_symplectic-acs} defining the corresponding integrable complex structures $I$ and $J$. Furthermore, the tensors $\Omega=\Re(\Omega_I)$, $I$, $J$ and $g$ are chosen to satisfy the presentation in Proposition~\ref{p:matrices} in the basis $\{X,IX,JX,KX\}$. In particular, 
\[
\Omega^{-1}=\frac{1}{2}g^{-1}[I,J]
\]
does not vanish identically, since $I$ and $J$ do not commute.  Therefore $(M, g, I, J)$ is a nondegenerate GK structure with a free $S^1$ symmetry.

Conversely, given any nondegenerate GK structure $(M,g,I,J)$ with a free $S^1$ symmetry, we can identify $M$ with the total space of an $S^1$-bundle and express $g$, $\Omega_I$ and $\Omega_J$ through functions $W$ and $p$ as in equations~\eqref{f:hdef} and \eqref{f:Omegaexplicit}.
\end{proof}

\begin{rmk}
	Given any smooth function $p\colon \R^3_{\pmb\mu}\to (-1,1)$ one can always solve the corresponding elliptic equation for $W$ at least locally. Therefore the germ of any function $p\colon \R^3_{\pmb\mu}\to (-1,1)$ can be realized on some nondegenerate GK manifold with $S^1$ symmetry.
\end{rmk}

\begin{rmk}[Relation to T-duality]
	A surprising feature underlying both the Gibbons-Hawking ansatz and Theorem~\ref{t:nondegenerate_gk_description} is that a solution to a linear PDE yields a non-linear geometric structure on $M^4$. We would like to thank an anonymous referee for kindly bringing to our attention that this is not a coincidence, and can be explained through the notion of \textit{T-duality}. Specifically, as was observed by Cavalcanti and Gualtieri in~\cite{ca-gu-10}, given a non-degenerate GK structure with a free $S^1$-symmetry
	\[
	M^4\to M^4/S^1=N^3
	\]
	there is a natural T-dual $S^1$-bundle $\hat M^4\to N^3$ endowed with a \textit{commuting} (odd) GK structure $(\hat M^4, \hat g, \hat I,\hat J)$, $\hat I\hat J=\hat J\hat I$, and by a result of Apostolov and Gualtieri~\cite{ap-gu-07} any such structure determines a local holomorphic orthogonal splitting
	\[
	\hat M^4\simeq_{\mathrm{loc}}\C\times \C
	\]
	rendering a \textit{linear} description of the underlying GK structure, which can be translated back to a linear description of a nondegenerate GK structure on $M^4\to N^3$ via the inverse T-duality. We expect that one could push this idea even further and use the T-duality to derive a local ansatz for \textit{GK solitons} (see section \S\ref{s:GKRS} for the precise definitions) similar to an alternative description of the Gibbons-Hawking anstaz presented in~\cite[Ex.\,5.2]{ca-gu-10}. We, however, take a different approach in this paper and derive the soliton equations directly on a nondegenerate GK manifold $(M^4,g,I,J)$, see Proposition~\ref{p:invGKsoliton}.
\end{rmk}

To simplify further computations, we observe a useful change of coordinates in $\R^3_{\pmb\mu}$ suggested by the structure of $h$. \[
\mu_+:=\tfrac{1}{2}(\mu_2+\mu_3),\quad \mu_-:=\tfrac{1}{2}(\mu_2-\mu_3).
\]
In these coordinates the metric $h$ of (\ref{f:hdef}) diagonalizes and takes the form
\begin{equation}\label{f:h_diagonal}
h=(1-p^2)d\mu_1^2+2(1-p)d\mu_+^2+2(1+p)d\mu_-^2.
\end{equation}
Furthermore, the equation ~\eqref{f:W_laplace} for $W$ takes form
\begin{equation}\label{f:W_laplace2}
W_{11}+\tfrac{1}{2}((1+p)W)_{++}+\tfrac{1}{2}((1-p)W)_{--}=0.
\end{equation}
In the rest of this paper, we will use $(\mu_1,\mu_+,\mu_-)$ as the coordinates on $\R^3_{\pmb\mu}$.

\subsection{Nondegenerate GK structures with effective \texorpdfstring{$S^1$}{} action} \label{ss:NGKeff}

Theorem~\ref{t:nondegenerate_gk_description} provides a classification of nondegenerate GK structures with \emph{free} $S^1$ action. Our next step is to relax this assumption and allow the $S^1$ action to have nontrivial isotropy subgroups. The following lemma shows that in the nondegenerate case the only new possibility is the presence of fixed points.

\begin{lemma}\label{l:free_nondegenerate}
	Let $(M,g,I,J)$ be a nondegenerate GK manifold, $\dim_\R M=4$. Then an effective action of $S^1$ on $(M,g,I,J)$ is free outside of a discrete set of points $M^{S^1}\subset M$. If $y\in M$ is a fixed point, then the isotropy representation of $S^1$ in $T_yM$ has weights $(\pm 1, \pm 1)$.
\end{lemma}
\begin{proof}
	Since the action of $S^1$ preserves $I$ and $J$, the fixed point set $M^{S^1}\subset M$ must be $I$ and $J$-holomorphic. 
	Note the complex structures $I$ and $J$ are compatible with $g$, induce the same orientation and $I\neq \pm J$ on $M$. Since $\dim_\R M=4$, this implies that $\Ker(I\pm J)=0$. Thus we conclude that $M^{S^1}$ must be zero-dimensional.	 To prove that the action of $S^1$ is free outside the fixed point set $M^{S^1}$, assume on the contrary that for some $x\in M\backslash M^{S^1}$ we have $\mathrm{Stab}(x)=\Z_k\subset S^1$. Then $Z:=M^{\Z_k}$ must be $I$ and $J$-holomorphic. On the other hand $Z$ is at least one-dimensional, since it contains $S^1\cdot x$. Therefore $Z$ must coincide with the whole $M$ and the action of $S^1$ on $M$ is not effective.
	
	Now let $y\in M$ be a fixed point, and let $\mathbf w=(w_1,w_2)$ be the weights of the representation of $S^1$ in $T_yM$:
	\begin{equation}\label{f:Li_decompoistion}
	T_yM= L_1\oplus L_2,\quad L_i\simeq \C
	\end{equation}
	where the action of $S^1\simeq U(1)$ on $L_i$ is given by $t\cdot z=t^{w_i}z$. Since the action of $S^1$ on the complement of the fixed point set is free, we necessarily have $w_i=\pm 1$, as claimed. 
\end{proof}

We next analyze the local behavior of $W$ near a fixed point.  Fix $y\in M^{S^1}$. By Lemma~\ref{l:free_nondegenerate},  a small $S^1$-invariant open ball $U_y$ around $y$ is equivariantly diffeomorphic to a unit ball in $\C^2$ with the standard diagonal action of $S^1$ via coordinate-wise complex multiplication:
\[
e^{\sqrt{-1}t}\cdot (z_1,z_2)=(e^{\sqrt{-1}t}z_1,e^{\sqrt{-1}t}z_2).
\]
Therefore, there exist spherical coordinates $(s,\rho)\in S^3\times [0,1)_\rho$ on $U_y$ such that the projection onto the orbit space $U_y\to U_y/S^1$ is given by
\[
\mathrm{pr}\colon S^3\times [0,1)_\rho\to S^2\times [0,1)_r,\quad \mathrm {pr}(s,\rho)=(\chi(s),\rho^2/2)
\]
where $S^2\times [0,1)_r\simeq U_y/S^1$ and the map $\chi\colon S^3\to S^2$ is given by the Hopf fibration. Using this local model of the projection $M\to M/S^1$, we conclude that if the $S^1$ action on a 4-dimensional manifold is free outside the fixed point set, then the orbit space $M/S^1$ admits the structure of a smooth 3-dimensional manifold, such that the projection $\pi\colon M\to M/S^1$ is a smooth map with $d\pi$ surjective outside of the fixed point set.

The local model of the projection $M\to M/S^1$ near a fixed point $y\in M$ implies that the curvature form $\beta$ of the principal bundle
\[
U_y\backslash\{y\}\to (U_y\backslash\{y\})/S^1
\]
must pair to $-2\pi$ with the homology class of a 2 sphere enclosing the center of $(U_y\backslash \{y\})/S^1\simeq B_1(0;\R^3)$:
\begin{equation}\label{f:beta_pairing_s2}
\int_{S^2}\beta=-2\pi.
\end{equation}
Moreover, since the vector field $X$ has a simple zero at $y$, we have $W(x)=g(X,X)^{-1}\sim \rho^{-2}$ near $y\in U_y$, so that $W(x)$ descended to $U_y/S^1$ must blow up as $r^{-1}$ for $r\to 0$. Similarly $|dW|_{h}\sim r^{-2}$ as $r\to 0$.

Now we prove a partial converse of the above local observation about the behavior of $W$ and $\beta$ near the image of the fixed point. 

\begin{prop}[Removable singularity result for the metric]\label{p:removable_singularity}
	Let $B^3$ be the unit ball in $\R^3$ centered at the origin. Assume that metric $h$ and 2-form $\beta_0$ in $B^3$ are real analytic. Let $W$ be a positive solution to
	\begin{equation}\label{f:removable_singualrity_W_eq}
	d\beta=0,\quad \beta := *_hdW+W\beta_0
	\end{equation}
	in $B^3\backslash \{0\}$ such that $W(x)/r\to 1/2$ and $|dW|_h(x)=o(r^{-3})$ as $r:=d_h(0,x)\to 0$. Then 
	\begin{itemize}
		\item $\beta$ pairs to $-2\pi$ with any 2-sphere enclosing the origin oriented by the outward normal vector;
		\item the total space $P$ of $S^1$-bundle with the connection $\eta$ such that $d\eta=\beta$ over $B^3\backslash\{0\}$ is diffeomorphic to $B^4\backslash\{0\}$, where $B^4$ is a unit ball in $\R^4$;
		\item the metric completion of $(P, g)$, $g=Wh+W^{-1}\eta^2$ is diffeomorphic to $B^4\supset B^4\backslash 0$, and the metric $g$ extends to a $C^{1,1}$ metric across $0\in B^4$.
	\end{itemize}
\end{prop}

\begin{proof}
	With our assumptions on the coefficients of equation $d\beta=0$ and on the solution $W$, we can apply a general result about fundamental solutions to elliptic equations with analytic coefficients~\cite{Fritz}. Specifically, using~\cite[\S 6]{Fritz} (see also p.\,275), we conclude that, in the notations of the cited paper, $W(x)=K(x,0)+A(x)$, where $A(x)$ is analytic in $B^3$ and $K(x,z)$ is the fundamental solution for the Dirichlet problem $d\beta=0$. Further, using the decomposition~\cite[eq.\,5.14]{Fritz} and the fact that the leading term in the decomposition is proportional to $1/r$ (see discussion on p.290), we find that $W$ admits a decomposition
	\begin{equation}\label{eq:W_series}
	W=\frac{1}{2r}+\sum_{k=0}^\infty \alpha_k(\zeta)r^k,
	\end{equation}
	where $(r,\zeta)\in (0,1)\times S^2$ are spherical coordinates, the functions $\alpha_k(\zeta)$ are analytic and the series is absolutely convergent with all its derivatives.
	
	To prove the first claim, we note that since $\beta$ is closed, the pairing is independent of the choice of sphere. If we take $S_\epsilon:=\{r=\epsilon\}$, then using the decomposition~\eqref{eq:W_series}
	\[
	\int_{S_\epsilon^2}\beta=\lim_{\epsilon\to 0}\int_{S_\epsilon^2}\beta=\lim_{\epsilon\to 0}\int_{S_\epsilon^2}*_hdW=\lim_{\epsilon\to 0}\int_{S_\epsilon^2} \frac{\del W}{\del r}d\gs_\epsilon=-2\pi.
	\]
	proving the first claim.  The second claim follows from a simple topological observation. The principal bundle $P$ is uniquely determined by a class $[\beta]\in H^2(B^3\backslash\{0\},2\pi\Z)\simeq 2\pi\Z$, and the bundle corresponding to $-2\pi$ is diffeomorphic to the total space of the standard Hopf fibration $B^4\backslash \{0\}\to B^3\backslash\{0\}$, which extends to a smooth map $B^4\to B^3$.
	
	Our next goal is to prove the $C^{1,1}$ extension of $g$ in this chart, following the argument of LeBrun~\cite{LeBrun}. In the exponential coordinates $(r,\zeta)\in (0,1)\times S^2$, the metric $h$ can be written as
	\[
	h=dr^2+r^2h^{S}_r,
	\]
	where $h^S_r$ is an analytic family of metrics on $S^2$ such that $h^S_r\to h^S_0$, where $h^S_0$ is the standard round metric on the unit sphere. Decompose $h$ as
	\[
	h=h'+h'',\quad h':=dr^2+r^2 h^S_0,
	\]
	i.e., $h'$ is the flat metric in exponential coordinates, and $h''$ is the remainder. Since $h''$ is analytic, and $h'$ matches $h_1$ up to order 2 near the origin, we have
	\[
	h''=\sum_{k=2}^{\infty} r^k h''_k(\zeta),
	\]
	where $h''_k$ is a family of symmetric two-tensors on $\R^3$ parametrized by $\zeta\in S^2$.
	
	Using the above two series expansions and inspecting the closedness of $\beta=*_hdW+W\beta_0$ in this decomposition, we find that
	\[
	\beta=*_{h'}d\left(\frac{1}{2r}\right)+d\gg, \qquad \gg=dr\sum_{k=0}^\infty r^k\gg'_k(\zeta)+\sum_{k=1}^\infty r^k \gg_k''(\zeta),
	\]
	where $\gg'_k(\zeta)$ and $\gg''_k(\zeta)$ are analytic functions and 1-forms on the sphere $S^2$. In particular, after a gauge transform $\eta=\eta_0+\gamma$, where $\eta_0$ is the connection form of the standard Hopf fibration $B^4\to B^3$ given by the flat metric on $B^4$.
	
	Let, as above, $\mathrm{pr}\colon S^3\times [0,1)_\rho\to S^2\times [0,1)$ be the projection given by $\mathrm{pr}(x,\rho)=(\chi(s),\rho^2/2)$, where $\chi\colon S^3\to S^2$ is the Hopf fibration on a 3-sphere. In particular, $\mathrm{pr}^*dr=\rho d\rho$.  By slight abuse of notation, we omit $\mathrm{pr}^*$ in front of $h,h',h''$ and $W$, below. As is well known from the geometry of the Hopf fibration $B^4\to B^3$, the metric
	\[
	g':= \frac{1}{\rho^2}h'+\rho^2 \eta_0^2=d\rho^2+\rho^2 (h_0^S+\eta_0^2)
	\]
	on $B^4\backslash\{0\}$ is flat and extends smoothly across $\{0\}\in B^4$. Now we observe that on $B^4\backslash\{0\}$
	\begin{equation}
	\begin{split}
	g-g'&=\left(W(h'+h'')+W^{-1}(\eta_0+\gg)^2\right)-g'\\&=
	Wh''+(W-\rho^{-2})h'+(W^{-1}-\rho^2)\eta_0^2+W^{-1}(2\gg\otimes \eta_0+\gg^2).
	\end{split}
	\end{equation}
	Using the series expansion above we find that
	\[
	g-g'=\sum_{k=2}^\infty \rho^k g''_k(\xi),
	\]
	where $g''_k(\xi)$ is an analytic family of symmetric 2-tensors on $\R^4$ depending on $\xi\in S^3$. This difference extends to a $C^{1,1}$ symmetric tensor on $B^4$.

\end{proof}
\begin{rmk}
	If $h$ and $\beta_0$ are additionally $SO(3)$-rotationally symmetric around $0$, then the $C^{1,1}$-regularity of $g$ can be upgraded to $C^\infty$ (analytic), since the coefficients in all the expansions are constant. This is the case with the original Gibbons-Hawking ansatz and LeBrun's construction of self-dual metrics~\cite{LeBrun}.
\end{rmk}

Proposition~\ref{p:removable_singularity} ensures that a metric $g$ of a certain form on the total space of a principal bundle $B^4\backslash\{0\}\to B^3\backslash\{0\}$ has a $C^{1,1}$-extension over the origin. If the principal bundle $B^4\backslash\{0\}\to B^3\backslash\{0\}\subset \R^3_{\pmb\mu}$ and the metric $g$ are given by Theorem~\ref{t:nondegenerate_gk_description} then we have a nondegenerate GK structure $(B^4\backslash\{0\}, g,I,J)$. The following proposition ensures that not only the metric, but the whole GK structure extends to a $C^{1,1}$ structure on $B^4$.

\begin{prop} \label{p:removable_singularity_GK} The following hold:
\begin{enumerate}
\item Suppose $(B^4\backslash\{0\},g,I,J)$ is a smooth GK structure such that $g$ has a $C^{1,1}$ extension to $B^4$, and $H$ has a $C^{0,1}$ extension to $B^4$.  Then there exists a $C^{1,1}$ extension of $(g,I,J)$ to $B^4$.
\item If $B^4\backslash \{0\}\to B^3\backslash \{0\}$ is a Hopf fibration, and there is a GK structure $(g,I,J)$ on $B^4\backslash \{0\}$ given by Theorem~\ref{t:nondegenerate_gk_description}, then the torsion $H$ has a $C^{0,1}$ extension to $B^4$.  In particular, the GK structure $(g,I,J)$ extends to a $C^{1,1}$ GK structure on $B^4$.
\end{enumerate}
\end{prop}
\begin{proof}
	To show item (1), first note that by the hypotheses on the extensions of $g$ and $H$, the Christoffel symbols of the Bismut connections $\N^I$ and $\N^J$ have a $C^{0,1}$ extension to $B^4$. Thus we can define $I$ (resp.\,$J$) at the origin by the parallel $\N^I$ (resp.\,$\N^J$) transport. The $C^{0,1}$ regularity of $\N^I$ implies that the resulting almost complex structure does not depend on the choice of a path and is $C^{1,1}$.
	
	Now we turn to item (2) and prove that the torsion 3-form $H$ of the GK structure provided by Theorem~\ref{t:nondegenerate_gk_description} in a principal $S^1$ bundle $B^4\backslash\{0\}\to B^3\backslash\{0\}$ has a $C^{0,1}$ extension. On the nondegenerate part of a GK manifold we have an identity (see~\cite[Lemma 3.8]{ASNDGKCY}) for $\theta_I^{\#}=-g^{-1}(*H)$,
	\[
	i_{\theta_I^{\#}}\Omega=-\frac{dp}{1-p^2}.
	\]
	In our case $\Omega$ is given by
	\[
	\Omega=-d\mu_1\wedge\eta+2Wd\mu_+\wedge d\mu_-
	\]
	hence
	\[
	\theta_I^\#=
	\frac{1}{1-p^2}\left(
	-p_1X+
	\frac{1}{2}W^{-1}p_+\frac{\del}{\del \mu_-}-
	\frac{1}{2}W^{-1}p_- \frac{\del}{\del \mu_+}
	\right)
	\]
	or equivalently, since $g=W\left((1-p^2)d\mu_1^2+2(1-p)d\mu_+^2+2(1+p)d\mu_-^2\right)+W^{-1}\eta^2$, we find
	\begin{equation}\label{f:theta_through_p}
	\theta_I=-W^{-1}\frac{p_1}{1-p^2}\eta+\frac{p_+}{1-p}d\mu_--\frac{p_-}{1+p}d\mu_+.
	\end{equation}
	
	The construction of Proposition~\ref{p:removable_singularity} provides a smooth map $\mathrm{pr}\colon B^4\to B^3$ such that $X$ is a smooth vector field on $B^4$ vanishing at the origin and the metric $g$ has a $C^{1,1}$ extension. Therefore the pull back of a $C^{k,\alpha}$ function on $B^3$ is also $C^{k,\alpha}$ on $B^4$. This implies that
	\[
	\frac{p_1}{1-p^2}X,\quad
	\frac{p_+}{1-p}d\mu_-,\quad \frac{p_-}{1+p} d \mu_+
	\]
	extend to $C^\infty$ vector field and differential forms on $B^4$. Therefore, since $g$ is $C^{1,1}$, tensors $\theta_I^\#, \theta_I$ and $H$ all have a $C^{1,1}$ extension over $B^4$.
\end{proof}

We can summarize the above observations as follows (compare with Theorem~\ref{t:nondegenerate_gk_description}).

\begin{thm}[Generalized-K\"ahler Gibbons-Hawking ansatz II]\label{t:nondegenerate_gk_description_v2}
	Consider
	\begin{enumerate}
		\item an open map $\iota\colon N\to \R^3_{\pmb\mu}$,
		\item a discrete subset $\{x_i\}\subset N$,
		\item smooth functions
		\[
		p\colon N\to (-1,1),\quad W\colon N\backslash\{x_i\}\to (0,+\infty)
		\]
		solving equation
		\begin{equation*}
		W_{11}+W_{22}+W_{33}+2(pW)_{23}=0,
		\end{equation*}
		such that the closed differential form $\beta\in\Lambda^2(N\backslash\{x_i\},\R)$
		\[
		\beta=-(W_3+(pW)_2)d\mu_1\wedge d\mu_2+(W_2+(pW)_3)d\mu_1\wedge d\mu_3-W_1d\mu_2\wedge d\mu_3
		\]
		represents a class in $H^2(N\backslash\{x_i\},2\pi\Z)$, and for any $x_i\in\{x_i\}$ we have 
		$W(x)/r\to 1/2$ and $|dW|_h(x)=o(r^{-3})$ as $r:=d_h(x_i,x)\to 0$,
		\item connection form $\eta$ with curvature $\beta$ in the principal $S^1$-bundle $\pi\colon P\to N\backslash\{x_i\}$ determined by $[\beta]$.
	\end{enumerate}
	Then there exists a unique GK structure $(M,g,I,J)$ with a tri-Hamiltonian $S^1$-action with fixed points $\{y_i\}$ such that
	\[
	(M\backslash\{y_i\},g,I,J)
	\]
	is the GK structure provided by Theorem~\ref{t:nondegenerate_gk_description}. The GK structure on $M$ is $C^\infty$ on $M\backslash\{y_i\}$ and is $C^{1,1}$ across $\{y_i\}$.

	Conversely any smooth nondegenerate GK manifold $(M,g,I,J)$ with a tri-Hamiltonian $S^1$ action is isomorphic to a GK manifold as above.
\end{thm}

\subsection{Examples}

	Before we have assumed that the Poisson tensor
	\[
	\gs=\tfrac{1}{2}g^{-1}[I,J]
	\]
	is everywhere nondegenerate, in particular, there was a well defined symplectic form $\Omega=\gs^{-1}$. Now we drop this assumption, and allow $\gs$ to vanish along a \emph{proper} subset $\mathbf{T}\subset M$. As we mentioned in Section~\ref{s:background}, in this situation
	\[
	\mathbf{T}=\{x\in M\ |\  I_x=\pm J_x \}
	\]
	is the union of (possibly singular) complex one-dimensional jointly $I$ and $J$-holomorphic curves. Since the $S^1$ action on $(M,g,I,J)$ preserves all the structure, we see that $\mathbf{T}$ is $S^1$-invariant. In particular, $(M\backslash \mathbf{T},g,I,J)$ is a \emph{nondegenerate} GK manifold with an effective $S^1$ action. Therefore, we can invoke Theorem~\ref{t:nondegenerate_gk_description_v2} and conclude that an appropriate cover $\til{M\backslash\mathbf{T}}$ is given by the construction in the theorem. It remains to understand how and when one can glue back a degeneracy locus into a quotient of such a manifold.
	
	We will not be able to give a complete answer to this question.  Instead, in this section, we review several known examples of GK structures with $S^1$ symmetry and demonstrate how they arise via the ansatz of Theorem~\ref{t:nondegenerate_gk_description_v2}.	
	Later in the paper we will show how to extend GK structure across the degeneracy locus for GK solitons.
	
	\begin{ex}[Generalized K\"ahler structure on the standard Hopf surface revisited]\label{ex:Hopf2}
		Consider the standard Hopf surface $(M,g,I,J)$ with the GK structure of Example~\ref{ex:Hopf1}.
		There is an $S^1$ action preserving the GK structure:
		\[
		e^{\sqrt{-1}t}\cdot(z_1,z_2):=(e^{\sqrt{-1}t}z_1, e^{\sqrt{-1}t}z_2).
		\]

		Let $\C^2\to(\C^*)^2/\Z\subset \C^2\backslash\{0\}/\Z=M$ be the universal cover of the nondegenerate part of the Hopf surface. Pick $w_1,w_2\in \C$ to be the coordinates in $\C^2$, $w_i=x_i+\sqrt{-1}y_i$ so that the holomorphic symplectic forms $\Omega_I$ and $\Omega_J$ are given by
		\[
		\begin{split}
		&\Omega_I=(dx_1+\sqrt{-1}dy_1)\wedge (dx_2+\sqrt{-1}dy_2),\\
		&\Omega_J = d\left(\log(e^{2x_1}+e^{2x_2})-x_2+\sqrt{-1}y_2\right)\wedge
		d\left(\log(e^{2x_1}+e^{2x_2})-x_1-\sqrt{-1}y_1\right).
		\end{split}
		\]
		In these coordinates, the vector field $X$ generating the $S^1$ action is given by $X=\del_{y_1}+\del_{y_2}$, and the moment map can be recovered from $i_X\Omega_I=d\mu_1-\sqrt{-1}d\mu_2$ and $i_X\Omega_J=d\mu_1-\sqrt{-1}d\mu_3$:
		\[
		\begin{split}
		\mu_1&=y_1-y_2=\arg z_1-\arg z_2,\\
		\mu_2&=x_1-x_2=\log\frac{|z_1|}{|z_2|},\\
		\mu_3&=x_1+x_2-2\log(e^{2x_1}+e^{2x_2})=\log\frac{|z_1 z_2|}{(|z_1|^2+|z_2|^2)^2}.
		\end{split}
		\]
		The functions $p$ and $W$ in this case are
		\[
		p=\frac{e^{2x_2}-e^{2x_1}}{e^{2x_2}+e^{2x_1}}=\frac{|z_2|^2-|z_1|^2}{|z_1|^2+|z_2|^2}, \qquad W^{-1}=g(X,X)=1.
		\]
	\end{ex}

	\begin{ex}[Generalized K\"ahler structure on the diagonal Hopf surfaces]\label{ex:hopf_diagonal_gk}
		Let $M$ be a diagonal Hopf surface with parameters $\alpha,\beta\in \{z\in\C\ |\ |z|>1\}$:
		\[
		M=(\C^2_{z_1,z_2}\backslash\{0\})/\Gamma,\quad \Gamma\colon (z_1,z_2)\mapsto (\alpha z_1, \beta z_2).
		\]
		In~\cite{SU} we have described a family of GK structures on $M$ determined by one scalar function. Let us review this construction and reinterpret it through the lens of Theorem~\ref{t:nondegenerate_gk_description}.
		
		Let $\C^2\to(\C^*)^2/\Gamma\subset \C^2\backslash\{0\}/\Gamma=M$ be the universal cover of the nondegenerate part of the Hopf surface. Choose $w_i=\log z_i$ to be the coordinates in $\C^2$, $w_i=x_i+\sqrt{-1}y_i$ and denote $a=\log|\alpha|$, $b=\log|\beta|$.  Given a function $p\colon \R\to (-1,1)$ of a real argument $2(\frac{b}{a} x_1-x_2)$, we can define a pair of complex structures $I$ and $J$ on $\C^2$ via a pair of complex-valued symplectic forms (in~\cite{SU} we used a function $k=(p+1)/2$ instead of $p$):
		\[
		\begin{split}
		\Omega_I&=dw_1\wedge dw_2,\\
		\Omega_J&=\left(dw_1-\frac{a}{b}d\bar w_2\right)\wedge \left(\frac{b(1+p)}{2a}d\bar{w_1}+\frac{1-p}{2}dw_2\right).
		\end{split}
		\]
		By a direct computation, $\Omega_I,\Omega_J$ satisfy the assumptions of Proposition~\ref{p:holo_symplectic-acs}, thus we have a GK structure on $\C^2_{w_1,w_2}$. This structure descends to $(\C^*)^2_{z_1,z_2}$, and under certain assumptions on the asymptotic of $p(x)$ as $x\to\pm\infty$, extends to a smooth GK structure $(M,g,I,J)$ on our Hopf surface. The function $p$ provides the angle function of this GK structure and $g$ is given by
		\[
		g=\Re\left(\frac{b(1+p)}{2a}dw_1\otimes d\bar{w_1}+\frac{a(1-p)}{2b}dw_2\otimes d\bar{w_2}\right).
		\]

		To relate this construction to Theorem~\ref{t:nondegenerate_gk_description}, let us choose coprime $m,n\in \Z$ and introduce an $S^1$ action on $M$ preserving the entire GK structure:
		\[
		u\cdot (z_1,z_2)=(u^mz_1,u^nz_2).
		\]
		The vector field generating this action is given in coordinates $(w_1,w_2)$ by
		\[
		X=m\del_{y_1}+n\del_{y_2},
		\]
		so that
		\[
		W^{-1}=g(X,X)=\frac{m^2b^2}{2ab}(1+p)+\frac{n^2a^2}{2ab}(1-p).
		\]
		As in the previous example, we recover the moment map by computing $i_X\Omega_I$ and $i_X\Omega_J$:
		\begin{equation}\label{f:diagonal_hopf_mu}
		\begin{split}
		\mu_1&=ny_1-my_2,\\
		\mu_2&=nx_1-mx_2,\\
		\mu_3&=-n\left(\frac{b}{a} x_1+\frac{1}{2}\chi\right)-m\left(\frac{a}{b}x_2+\frac{a}{2b}\chi\right),
		\end{split}
		\end{equation}
		where $\chi\colon \R\to \R$ is an antiderivative of $p$ evaluated at $2(\frac{b}{a}x_1-x_2)$. We want to stress that this moment map $\pmb\mu$ is defined only on the universal cover of the nondegenerate part of $M$.
	\end{ex}
\begin{rmk}
	The key new feature of the above circle action on diagonal Hopf surfaces is the existence of nontrivial isotropy subgroups. Specifically, we have stabilizers
	\[
			G_{(0,z_2)}\simeq \Z_{n}\quad G_{(z_1,0)}\simeq \Z_{m}.
	\]
	Because of the nontrivial isotropy groups, the orbit space $M/S^1$ is not a manifold anymore, but rather an \emph{orbifold}. Geometrically, $M/S^1$ is the product of $S^1$ and a spindle 2-sphere $S^2(m,n)$ with two cone singularities of angles $2\pi/n$ and $2\pi/m$. In this case, $\pmb\mu$ identifies the universal cover of the smooth part of $S^1\times S^2(m,n)$ with $\R^3_{\pmb\mu}$.
	
	The presence of points with nontrivial finite stabilizers is a new phenomenon compared to $S^1$ actions on hyperK\"ahler 4-dimensional manifolds. By Lemma~\ref{l:free_nondegenerate} only points on the degeneracy locus $\mathbf{T}$ might potentially have nontrivial stabilizers.
\end{rmk}

\begin{ex}[LeBrun's GK structure on parabolic Inoue surface]
	In~\cite{LeBrun} LeBrun has constructed an explicit family of anti-self-dual metrics on blown-up Hopf surfaces and their deformations~--- parabolic Inoue surfaces. It was later observed by Pontecorvo that this Hermitian structure on Inoue surfaces naturally fits into a GK structure. Let us review this construction.
	
	Let $\mathbb H^3$ denote hyperbolic 3-space modeled on the upper half-space
	\[
	\mathbb H^3=\{(x,y,z)\in\R^3\ |\ z>0\}, \quad \til h:=\frac{dx^2+dy^2+dz^2}{z^2}.
	\]
	It is well-known that given any point $p\in\mathbb H$ there exists a positive Green's function $G_p$~--- a solution to 
	\[
	\Delta_{\til h}G_p=-2\pi \delta_p.
	\]
		
	Let $\{p_j\}_{j\in\Z}$ be a sequence of points $p_j=(0,0,\gl^j)\in\mathbb H$, $\gl>1$. Using the explicit form of $G_p$, one can check that the sum $V=1+\sum_j G_{p_j}$ is absolutely convergent and the limit solves Laplace equation $\Delta_{\til h}V=-2\pi\sum_j \delta_{p_j}$. Then we can define a 2-form $\beta=*_{\til h}dV$ representing a class in $H^2(\mathbb H\backslash\{p_j\},2\pi \Z)$ and construct a principal $S^1$ bundle $M_0$ over $\mathbb H\backslash\{p_j\}$ with connection $\eta$ and curvature form $\beta$. One can then define a complex structure $I$ as the unique complex structure such that complex-valued symplectic form
	\[
	\Omega_I:=\frac{dw}{w}\wedge\left(V\frac{dz}{z}+\sqrt{-1}\eta\right),\quad w=x+\sqrt{-1}y
	\]
	is of type $(2,0)$. This complex structure is compatible with the metric
	\[
	g=\frac{z^2}{x^2+y^2+z^2}\left( V\til h+V^{-1}\eta^2\right).
	\]
	LeBrun proved that while the metric $g$ on $M_0$ is not complete, its completion is a complex manifold obtained from $M_0$ by gluing in points $\{p_j\}$ and a punctured plane $\C^*$ along $z=0$. This complete manifold admits a free isometric $\Z$ action generated by scalings $(x,y,z)\mapsto (\gl^k x,\gl^k y,\gl^k z)$, and the quotient space $M$ is a compact complex surface, isomorphic to a \emph{parabolic Inoue surface}.
	
	It was observed by Pontecorvo~\cite[\S 2.3]{PontecorvoCS} that inversion in the hemisphere $\{x^2+y^2+z^2=1\}$ composed with the reflection in $\{x=0\}$ plane yields an orientation preserving, isometric involution $j\colon M_0\to M_0$, and denoting by $J$ the pullback of $I$ under this involution, we obtain a GK structure $(M,g,I,J)$.  The degeneracy locus of this GK structure consists of the disjoint union of an elliptic curve (the quotient of the \emph{glued in} copy of $\C^*$ along $z=0$) and a cycle of $k$ rational curves over $\{x=y=0\}\subset\mathbb H$.
	
	Now, we can relate this construction to the Gibbons-Hawking ansatz by evaluating $i_X\Omega_I$ and $i_X\Omega_J=i_X j^*(\Omega_I)$ on the universal cover of the nondegenerate part of $M$:
	\[
	\begin{split}
	\mu_1&=\arg w,\\
	\mu_2&=\log|w|=\frac{1}{2}\log(x^2+y^2),\\
	\mu_3&=\log(x^2+y^2+z^2)-\frac{1}{2}\log(x^2+y^2).
	\end{split}
	\]
	Or equivalently $\mu_+=\frac{1}{2}\log(x^2+y^2+z^2)$ and $\mu_-=\frac{1}{2}(\log(x^2+y^2)-\log(x^2+y^2+z^2))$.
	The range of the moment map is $\{\mu_-<0\}\subset \R^3_{\pmb\mu}$.
	
	We can rewrite the metric $g$ as
	\[
	g=W\underbrace{\left(\frac{z^2}{x^2+y^2+z^2}\right)^2\til h}_{:=h}+  W^{-1}\eta^2,\quad W:=\frac{x^2+y^2+z^2}{z^2}V.
	\]
	With this change of perspective, the equation $\Delta_{\til h}V=0$ translates into the equivalent form ~\eqref{f:W_laplace} for $W$ and the curvature form $\beta=*_{\til h}dV$ matches the expression in ~\eqref{f:beta0_def}.
	As before, the moment map is defined only on the universal cover of the nondgenerate set, and the entire parabolic Inoue surface is the completion of a quotient of a local model provided by Theorem~\ref{t:nondegenerate_gk_description}.
\end{ex}

\section{Generalized K\"ahler-Ricci solitons} \label{s:GKRS}

In this section we turn to analyzing generalized K\"ahler-Ricci solitons.  We first establish the fundamental point that for a pluriclosed soliton, the vector field $V = \tfrac{1}{2} (\theta^{\#} - \N f)$ is holomorphic, and $IV$ is Killing.  We thus have potentially \emph{two} Killing fields in the case of generalized K\"ahler-Ricci solitons.  By further analysis we prove in Proposition \ref{p:IV_I_J_holomorphic} that either these vector fields both vanish, yielding hyperK\"ahler structure, one is a quotient of the standard Hopf surface, or one of these vector fields is biholomorphic.  The last case of course is our main interest, yielding a structure in the setting of \S \ref{s:GKsymm}.  A key point in this setting is to obtain a scalar reduction for the soliton equation, and in particular in Proposition \ref{p:invGKsoliton} we obtain an explicit local form for the angle function $p$ in $\mathbb R^3_{\pmb \mu}$ up to a choice of two real parameters.  We finish by rederiving the solitons constructed in \cite{SU} in this ansatz, and show how it leads to a certain explicit solution $\til{W}$ to the elliptic equation \eqref{f:W_laplace2} depending on the real parameters determining $p$.

\subsection{Fundamental structure}
	Recall that a GK structure $(M,g,I,J)$ is a (steady, gradient) soliton, if there exists a function $f\colon M\to \R$ such that
	\begin{equation}\label{f:soliton}
	\left\{
	\begin{split}
		\Rc - \tfrac{1}{4} H^2 + \N^2 f =&\ 0\\
		d^* H + i_{\N f} H =&\ 0
	\end{split}
	\right.
	\end{equation}
	where $\Rc=\Rc^{M,g}$ is the Ricci curvature of $g$ and $H:=-d^c_I\gw_I=d^c_J\gw_J$.  
	
    

We recall the well-known fact that the K\"ahler-Ricci soliton equations imply that the soliton vector field $\N f$ is holomorphic, with $I \N f$ a Killing field.  It turns out that a similar phenomenon holds here for the modified vector field $\theta^{\#} - \N f$ (previously established in \cite{st-19-soliton} for the 4-dimensional case). Specifically, we have the following equivalent formulation of the soliton system. In this statement, $(M,g,I)$ is not necessarily a part of GK structure, but merely a pluriclosed Hermitian manifold.

\begin{prop} \label{p:solitonVF} Let $(M^{2n}, g, I)$ be a pluriclosed Hermitian manifold (i.e. $dd^c_I\gw_I=0$). Then $g$, $H=d^c_I\gw_I$ and $f\colon M\to \R$ solve the soliton system~\eqref{f:soliton} if and only if the vector field
\begin{align*}
V:=&\ \tfrac{1}{2} \left(\theta_I^{\#} - \N f \right).
\end{align*}
satisfies
\begin{equation}\label{f:solitonB}
\begin{cases}
\mc L_{V}I=0\\
\mc L_{IV}g=0\\
\mc L_{V}\gw_I=\rho_I^{1,1}
\end{cases}
\end{equation}
In particular $IV$ is holomorphic and Killing.
\begin{proof}
	We claim that after splitting the soliton equations~\eqref{f:soliton} by holomorphic types we have the following equivalences:
	\begin{equation}\label{f:soliton_equivalences}
	\begin{split}
	(\Rc - \tfrac{1}{4} H^2 + \N^2 f)^{2,0} &=0\
	\Longleftrightarrow\ (\mc L_V g)^{2,0}=0\\
	(d^* H + i_{\N f} H)^{2,0}&=0\ \Longleftrightarrow\ (\mc L_V\gw_I)^{2,0}=0\\
	(\Rc - \tfrac{1}{4} H^2 + \N^2 f)^{1,1} &=0\
	\Longleftrightarrow\ (\rho_I(I\cdot,\cdot)+\mc L_V g)^{1,1}=0\\
	(d^* H + i_{\N f} H)^{1,1}&=0\ \Longleftrightarrow\ (\mc L_{IV}\gw_I)^{1,1}=0.
	\end{split}
	\end{equation}
	Indeed, denote by $\N^C$ the Chern connection on $(M,g,I)$ and by $T_{ij}^k$ the components of its torsion tensor. Let $\theta_I=\theta_i dz^i+\theta_{\bar j} d\bar{z^j}$, $\theta_i=T_{ik}^k$ and $\theta_I^\#=g^{i\bar j}(\theta_i \del_{\bar j}+\theta_{\bar j} \del_i)$ be the local coordinate expressions for the Lee form and the Lee vector field respectively. 
	By the computations of~\cite[Prop.\,6.3]{PCFReg} we know that
		\begin{equation}\label{f:ricci_bismut_identity}
		\Rc-\tfrac{1}{4}H^2+\mc L_{\tfrac{1}{2}\theta_I^\#} g=-\rho_I^{1,1}(I\cdot,\cdot),
		\end{equation}
		Therefore,
		\begin{equation*}
		(\Rc-\tfrac{1}{4}H^2+\N^2 f)^{2,0}=\tfrac{1}{2}(-\mc L_{\theta_I^\#}g+\mc L_{\N f} g)^{2,0}=-(\mc L_V g)^{2,0}.
		\end{equation*}
		This establishes the first equivalence of~\eqref{f:soliton_equivalences}.
		
		Furthermore, by an explicit computation we have
		\begin{equation*}
		\begin{split}
		(d^*H)^{2,0}&=
		-\sqrt{-1}\,\bar{\del}^*\del \gw=
		\sqrt{-1}[\N^C_k T_{ij}^k+\theta_k T_{ij}^k]dz^i\wedge dz^j =
		\sqrt{-1}(\mc L_{\theta_I^\#}\gw)^{2,0}.
		\end{split}
		\end{equation*}
		Next we have
		\[
		(i_{\N f}H)^{2,0}=
		-\sqrt{-1}(i_{\N f}d\gw)^{2,0}=
		-\sqrt{-1}(i_{\N f}d\gw+dd^cf)^{2,0}=
		-\sqrt{-1}(\mc L_{\N f}\gw)^{2,0}.
		\]
		Collecting the last two identities together we find:
		\begin{equation*}
		(d^*H+i_{\N f} H)^{2,0}=2\sqrt{-1}(\mc L_V\gw)^{2,0}.
		\end{equation*}
		This establishes the second equivalence of~\eqref{f:soliton_equivalences}.
		
		The two quantities $(\mc L_V\gw)^{2,0}$ and $(\mc L_V g)^{2,0}$ represent the $g$-symmetric and $g$-antisymmetric parts of the operator $\mc L_V I$, thus the first two identities of~\eqref{f:soliton_equivalences} are equivalent to
		\[
		\mc L_V I=0.
		\]
		Using again identity~\eqref{f:ricci_bismut_identity} we conclude that
		\[
		(\Rc - \tfrac{1}{4} H^2 + \N^2 f)^{1,1}=-(\rho_I(I\cdot,\cdot)+\mc L_V g)^{1,1}.
		\]
		This is the third identity of~\eqref{f:soliton_equivalences}. If we also assume that $V$ is holomorphic, we can equivalently rewrite it as
		\[
		\rho_I^{1,1}=\mc L_V\gw_I.
		\]

		For the last identity we observe that
		\begin{equation*}
		\begin{split}
		(d^*H)^{1,1}=&
		\sqrt{-1}(\bar{\del^*}\bar{\del}-\del^*\del)\gw=
		g^{m\bar n}(
		\N^C_m T_{\bar n\bar j i}
		+\theta_m T_{\bar n \bar j i}
		-\N^C_{\bar n}T_{mi\bar j}
		-\theta_{\bar n} T_{mi\bar j}
		)
		dz^i\wedge d\bar{z^j}\\=&
		g^{m\bar n}(
		\N^C_{\bar j}T_{im\bar n}
		-\N^C_i T_{\bar j\bar n m}
		+\theta_m T_{\bar n\bar j i}
		-\theta_{\bar n}T_{mi\bar j}
		)
		=(\mc L_{I\theta^\#}\gw)^{1,1},
		\end{split}
		\end{equation*}
		where we have used the following corollary of the differential Bianchi identity:
		\[
		g^{m\bar n}(
		\N^C_m T_{\bar n \bar j i}-\N^C_{\bar n} T_{m i \bar j})
		=
		g^{m\bar n}(
		\N^C_{\bar j}T_{im\bar n}-\N^C_i T_{\bar j\bar n m}
		).
		\]
		At the same time
		\[
		(i_{\N f}H)^{1,1}=-(i_{I \N f}d\gw+\underbrace{di_{I\N f}\gw}_{d^2f=0})^{1,1}=-(\mc L_{I\N f}\gw)^{1,1}.
		\]
		Combining the above two equations gives
		\[
		(d^*H+i_{\N f}H)^{1,1}=(\mc L_{IV}\gw)^{1,1},
		\]
		as claimed. This finishes the proof of~\eqref{f:soliton_equivalences}. It remains to notice that the right hand side equations of~\eqref{f:soliton_equivalences} are immediately equivalent to~\eqref{f:solitonB}.
\end{proof}
\end{prop}


The interpretation of the soliton equations \eqref{f:solitonB} becomes particularly useful on a non-degenerate GK manifold, since the Bismut-Ricci tensors $\rho_I$ and $\rho_J$ can be expressed through the associated \emph{Ricci potential}, as	
derived in \cite{ASNDGKCY}.  We recall the basic elements, referring to \cite{ASNDGKCY} for further detail.  Given a nondegenerate generalized K\"ahler structure $(M^{2n}, g, I, J)$, let
\begin{align*}
\Phi = \frac{1}{2}\log \frac{\det (I - J)}{\det (I + J)}.
\end{align*}
In dimension $4$ this is determined by $p$, in particular
\begin{align*}
\Phi = \log \frac{1-p}{1 + p}.
\end{align*}
The function $\Phi$ is in a sense analogous to the usual Ricci potential in K\"ahler geometry.  In particular, one has
\begin{align*}
\rho_I = - \tfrac{1}{2} d J d \Phi, \qquad \rho_J = - \tfrac{1}{2} d I d \Phi.
\end{align*}
Furthermore, (\cite{ASNDGKCY} Lemma 3.8), the difference of Lee vector fields is $\Omega$-Hamiltonian with potential function $\Phi$, i.e.
\begin{equation}\label{f:thetaPhi}
\begin{split}
\left(\theta_I^{\#} - \theta_J^{\#}\right) = \gs d \Phi.
\end{split}
\end{equation}

\begin{defn} \label{d:GKsoliton} If a generalized K\"ahler manifold $(M^{2n}, g, I, J)$ is a soliton, then both pluriclosed structures $(M^{2n},g,I)$ and $(M^{2n},g,J)$  must satisfy equations~\eqref{f:solitonB}. Thus we obtain vector fields
	\begin{align*}
	V_I = \tfrac{1}{2} \left( \theta_I^{\#} - \N f \right), \qquad V_J = \tfrac{1}{2} \left( \theta_J^{\#} - \N f \right),
	\end{align*}
	which are holomorphic with respect to the corresponding complex structures, as well as Killing fields $I V_I$ and $J V_J$.  We say the \emph{rank} of the soliton is
	\begin{align*}
	\sup_{p \in M} \dim \spn \{I V_I, J V_J \} \subset T_p M.
	\end{align*}
\end{defn}

\begin{ex}[GK solitons the standard Hopf surface and its universal cover]\label{ex:hopf_round_soliton}
	Let $(M,g,I,J)$ be the standard Hopf surface with the round metric and GK structure of Example~\ref{ex:Hopf1}. This structure solves~\eqref{f:soliton} with $f=0$. Since $g$ is conformally K\"ahler with a factor $(|z_1|^2+|z_2|^2)^{-1}$, we find that $\theta=-d\log(|z_1|^2+|z_2|^2)$. Hence $V_I$ and $V_J$ are nowhere zero and are given by $\pm \tfrac{1}{2}\log(|z_1|^2+|z_2|^2)$. In particular $IV_I$ and $JV_J$ are not proportional to each other, so $(M,g,I,J)$ with $f=0$ is a GK soliton of rank two.
	In~\cite{GauduchonIvanov} Gauduchon and Ivanov proved that this is the only non-K\"ahler compact complex surface admitting a GK soliton with $f=0$.
	
	Let $\til M$ be the universal cover of $M$ with the GK structure inherited from $M$. Clearly, we still can think of $(\til M,g,I,J)$ and a function $f_0=0$ as a GK soliton of rank two. Consider now a new function $f_1:=-\tfrac{1}{2}\log(|z_1|^2+|z_2|^2)$. The gradient flow of this function preserves the entire GK structure, therefore $f_1$ (and any of its multiples) is also a solution to the soliton system. For this choice of the soliton function, we will have $V_J=0$ and $V_I=2\N f_1$. In particular, $(\til M,g,I,J)$ with a function $f_1$ is now a rank \emph{one} soliton. However, since $f_1$ does not descend to the quotient $M$, this solution makes sense only on the universal cover of $M$.
	
	The observed phenomenon~--- different soliton functions compatible with the same GK structure~--- can potentially occur only on non-compact manifolds with a gradient vector field preserving the entire GK structure. In Example~\ref{ex:hopf_diagonal_soliton} below we review our construction from~\cite{SU} and show that the universal cover of a diagonal Hopf surface admits a rank one GK soliton.
\end{ex}

Now we observe further fundamental properties of the vector fields $V_I, V_J$ on the generalized K\"ahler Ricci solitons.

\begin{lemma} \label{l:jointholo} Let $(M^{2n}, g, I, J)$ be a generalized K\"ahler-Ricci soliton.  Then either
\begin{enumerate}
\item $g$ is compatible with a one-parameter family of distinct complex structures,
\item $I V_I$ is $J$-holomorphic and $J V_J$ is $I$-holomorphic.  In this case $[I V_I, J V_J] = 0$.
\end{enumerate}
\begin{proof} Suppose $I V_I$ is not $J$-holomorphic.  Since $I V_I$ is an $I$-holomorphic Killing field, we can pull back the GK structure $(g, I, J)$ by the one-parameter family of diffeomorphisms $\phi_t$ generated by $I V_I$ to see that $(\phi_t^* g, \phi_t^* I, \phi_t^* J) = (g, I, \phi_t^* J)$ is GK.  In particular, $g$ is compatible with the one-parameter family of complex structures $\phi_t^* J$, which are distinct since $I V_I$ is not $J$-holomorphic.  An identical argument holds if $J V_J$ is not $I$-holomorphic.

For the second claim, we prove that if $IV_I$ is $J$-holomorphic, then $[I V_I, J V_J] = 0$. Indeed, since $IV_I$ is Killing and $I$-holomorphic, we find $\mc L_{IV_I}\theta_I=\mc L_{IV} \left(*d^c_I\gw_I\right)=0$, thus
\begin{align*}
\mc L_{I V_I} df=\mc L_{IV_I} (-\theta_I+df)=-2g([IV_I,V_I],\cdot)=0.
\end{align*}
Since $I V_I$ is Killing it also follows that $\mc L_{IV_I}(g^{-1}df)=[I V_I, \N f] = 0$.  Using that $I V_I$ is both $I$ and $J$-holomorphic we then obtain
 \begin{align*}
 [I V_I, J V_J] = J [I V_I, V_J] = 
- J [I V_I, V_I] = 0,
\end{align*}
where we have used the identity $V_J=-V_I-\N f$.
\end{proof}
\end{lemma}

Using structures special to four dimensions, we can show in this case that without loss of generality, $IV_I$ is $J$-holomorphic, unless $(M,g,I,J)$ has a very special geometry.

\begin{prop}\label{p:IV_I_J_holomorphic} Let $(M^4, g, I, J)$ be a generalized K\"ahler-Ricci soliton.  Up to exchanging the roles of $I$ and $J$, at least one of the following holds:
\begin{enumerate}
\item Both $V_I = 0$ and $V_J = 0$, and $(M^4, g, I)$ is hyperK\"ahler.
\item $I V_I$ is nonzero and $J$-holomorphic.
\item $I V_I$ is nonzero, and is not $J$-holomorphic, and $J V_J$ is nonzero, and is not $I$-holomorphic.  In this case $(M^4, g, I)$ and $(M^4, g, J)$ are isometric to a quotient of the universal cover of the standard Hopf surface.
\end{enumerate}

\begin{proof} To address case (1), assume that both $\theta_I^{\#}-\N f$ and $\theta_J^{\#}-\N f$ vanish. Since on a $4$-dimensional GK manifold one has $\theta_I=-\theta_J$, we have $\theta_I^{\#} =\N f=0$, i.e., $(M,g,I)$ and $(M,g,J)$ are K\"ahler, Ricci flat and in fact $(g, I)$ is part of a hyperK\"ahler structure as discussed in \S \ref{s:background}.

Now assume without loss of generality that $V_I$ is nonzero.  We first suppose that $V_J$ is zero.  Using that $\theta_I = - \theta_J$ and $\theta_J^{\#} = \N f$ it follows that 
\begin{align*}
\sigma d \Phi = (\theta_I - \theta_J)^{\#} = -2 (\theta_J)^{\#} = -2 \N f.
\end{align*}
Noting that now $V_I = -\N f$, we have that $I \N f$ is $I$-holomorphic and Killing.  Also, since $V_J = 0$, by \eqref{f:solitonB} we have $\rho_J^{1,1} = 0$. We also know that $\rho_J^{2,0}=0$, since Bismut-Ricci curvature is conformally invariant, and $(M,g,J)$ is conformally equivalent to a K\"ahler metric $(M,e^{-f}g,J)$. Thus
\begin{align*}
-2\mc L_{I \N f} \Omega = d i_{I \sigma d \Phi} \Omega = 2d I d \Phi = -2 \rho_J = 0.
\end{align*}
Also, since $\sigma$ is type $(2,0) + (0,2)$ with respect to $I$,
\begin{align*}
\mc L_{I \N f} \Phi = I d \Phi (\N f) = \tfrac{1}{2} \sigma( d \Phi, I d \Phi) = 0.
\end{align*}
To summarize, we have shown that the vector field $I V_I = - I \N f$ preserves $g$, $\sigma$, and $\Phi = \log \frac{1-p}{1+p}$.  By rearranging formula \eqref{f:K}, it is possible to express $J$ in terms of $\sigma, g$, and $p$, and thus
\begin{align*}
\mc L_{I V_I} J = \mc L_{\N f} J = 0,
\end{align*}
yielding case (2).

Now we address the remaining case, where $I V_I$ is nonzero and not $J$-holomorphic and $J V_J$ is nonzero and not $I$-holomorphic.  Using first the hypothesis that $I V_I$ is not $J$-holomorphic, by Lemma \ref{l:jointholo} it follows that $g$ is compatible with a one-parameter family of complex structures, given by $J_t = \phi_t^* J$ where $\phi_t$ is generated by $I V_I$.  It follows from (\cite{PontecorvoCS} Corollary 1.6) that $(M^4, g)$ is antiself-dual, so $W_+ = 0$.  As is well-known, in the K\"ahler setting $W_+$ is determined by the scalar curvature.  In the Hermitian setting $W_+$ is determined by the scalar curvature and $d \theta$, and in particular by (\cite{BoyerConformalDuality} Lemma 1), the vanishing of $W_+$ implies that $d \theta_I$ is antiself-dual (noting that the proof of the equivalence of items 1-3 in that statement is local and does not need the compactness hypothesis and hence applies in this setting).  Taking the Hodge dual of the second soliton equation, we get an equivalent identity
\begin{align*}
d (e^{-f} \theta_I) = e^{-f} \left( d \theta_I - df \wedge \theta_I \right)=0.
\end{align*}
It follows that $df \wedge \theta_I$ is also antiself-dual, but as a wedge product of one-forms this can only be antiself-dual if it vanishes, hence we conclude $d \theta_I = 0$.  By passing to the universal cover we obtain a smooth function $\phi$ such that $\theta_I = d \phi$.  Now note that we have a one-parameter family of generalized K\"ahler structures $(g, I, J_t)$, all of which have generically nondegenerate associated Poisson tensor.  It follows that
\begin{align*}
\theta_{J_t} = - \theta_I = - d \phi.
\end{align*}
Hence we obtain a one-parameter family of (possibly incomplete) K\"ahler structures $(g_0 = e^{- \phi} g, J_t)$ with fixed Riemannian metric. Thus the reduced holonomy group of $(M,g_0)$ commutes with $J_t$, so must be a subgroup of $SU(2)$, and $g_0$ is hyperK\"ahler with complex structures $(I', J, K)$ (note that the original complex structure $I$ is \emph{not} part of this hyperK\"ahler structure, otherwise the original structure is already hyperK\"ahler).  Note that, by the conformal invariance of the Bismut-Ricci tensor in dimension $4$, it then follows that $\rho_J = 0$, and thus $V_J$ is a Killing field.  

Now we use that $J V_J$ is nonzero, and also not $I$-holomorphic.  The same line of arguing above holds with the roles of $I$ and $J$ reversed, leading to the conclusion that $\rho_I = 0$ and $V_I$ is a Killing field.  It follows that $\theta_I = V_I - V_J$ is a nonzero Killing field, and thus also is parallel.  The equation $(\rho_I)^{1,1} = 0$ now takes the simplified form
\begin{gather}
\Rc - \tfrac{1}{4} H^2 = 0.
\end{gather}
The universal cover of $(M^4, g, I, J)$ now splits according to the leaves of $\theta_I$, and by the last equation and the fact that $H = * \theta_I$ it follows that the transverse metric has constant positive sectional curvature, and is thus a quotient of $S^3$.  It follows that the universal cover of $(M^4, g)$ is a standard cylinder $S^3 \times \mathbb R$, and the complex structure $I$ and $J$ must be that of the lift of a standard Hopf surface by \cite{GauduchonWeyl}.
\end{proof}
\end{prop}

\begin{rmk}\label{r:rank1_s1}
	Proposition~\ref{p:solitonVF} implies that the dimension of the isometry group $\Isom(g)$ of rank one soliton $(M,g,I,J)$ is at least 1.  Going further, it is possible to reduce the study of the `least symmetric' rank one solitons to the study of the rank one solitons among GK manifolds with $S^1$ symmetry as in Section~\ref{s:GKsymm}.
\end{rmk}

\begin{prop}\label{p:one-diml_isometry}
	Let $(M,g,I,J)$ be a complete GK soliton of rank one. Assume that the isometry group $\Isom(g)$ is one-dimensional. Then $\Isom(g)$ preserves GK structure and we have two possibilities
	\begin{enumerate}
		\item $G\simeq \R$ acts freely and properly on $M$. In particular, there is a free $S^1\simeq\R/\Z$ action on a GK soliton $(M/\Z,g,I,J)$;
		\item $G\simeq S^1$.
	\end{enumerate}
\end{prop}
\begin{proof}
	First we note that we can assume $IV_I$ is both $I$ and $J$ holomorphic, since otherwise by Proposition~\ref{p:IV_I_J_holomorphic} $(M,g)$ is isometric to a quotient of the standard Hopf surface, and in particular $\dim\Isom(g) > 1$.
	
	Without loss of generality we assume that $IV_I$ is a nonzero vector field. Then $IV_I$ generates a one-dimensional subgroup $G$ of $\Isom(g)$. Since the isometry group is one-dimensional itself, $\Isom(g)=G$. Thus either $G\simeq \R$ or $G\simeq \R$. In either case the action of $G$ must be proper~\cite[Prop.\,1]{YauIsometries}, and if $G\simeq \R$, then the action must be also free, since the isotropy subgroups of an isometric action are compact.
\end{proof}

\subsection{Invariant solitons}

Let $(M,g,I,J)$ be a GK soliton of rank one. Then there is a generically nonzero vector field $X$ preserving the structure, and we can use the construction of \S\ref{s:GKsymm} to give a local description of $(M,g,I,J)$ in terms of $p$ and $W$. In this subsection we show that on a rank one soliton $p$ can be completely described locally in the nondegenerate locus. The key ingredients of this description are the simple expression of the Bismut-Ricci curvature and the fact that the difference of the Lee vector fields is a Hamiltonian vector field as discussed at the start of this section.

\begin{prop} \label{p:invGKsoliton} Suppose $(M^{2n}, g, I, J)$ is a nondegenerate generalized K\"ahler-Ricci soliton.  Then the soliton vector fields $V_I$ and $V_J$ preserve $\Omega$, and
\begin{align*}
\Phi =&\ 2\psi_{V_I} - 2\psi_{V_J}+\mbox{const},
\end{align*}
where $\psi_{V_I}, \psi_{V_J}$ denote the associated local Hamiltonian potentials.

In particular, in the case $n=2$, if one has a rank one soliton with $V_I = {a_I} IX$, $V_J = {a_J} JX$ then locally $\Phi$ satisfies
\begin{equation}\label{f:Phi_soliton}
\Phi = a_+ \mu_+ + a_- \mu_-+ \mbox{const},\quad a_+=-2a_I+2a_J,\ a_-:=-2a_I-2a_J
\end{equation}
with the associated soliton function satisfying
\begin{align*}
df &= \frac{1}{2}\left(p(a_+d\mu_++a_-d\mu_-)- a_+d\mu_+ + a_-d\mu_-\right).
\end{align*}
Conversely, if $\Phi$ is given by~\eqref{f:Phi_soliton}, then $ \frac{1}{2}\left(p(a_+d\mu_++a_-d\mu_-)- a_+d\mu_+ + a_-d\mu_-\right)$ is closed, and if it is exact with potential function $f$, then $(M,g,I,J)$ is a rank one soliton with soliton potential function $f$.
\begin{proof} 
Consider the solution to GKRF in the $I$-fixed gauge with initial data $(g, I, J)$.  By (\cite{ASNDGKCY} Lemma 5.2) we know that $\dt \Omega = 0$.  On the other hand all of the data evolves by pullback by the family of diffeomorphisms generated by $V_I$.  Thus
\begin{align*}
0 =&\ \dt \Omega = \mc L_{V_I} \Omega = d \left( i_{V_I} \Omega \right).
\end{align*}
The same argument using GKRF in the $J$-fixed gauge implies that vector field $V_J$ also preserves symplectic structure $\Omega$. Let $\psi_{V_I}$ and $\psi_{V_J}$ be local Hamiltonian potentials for vector fields $V_I$ and~$V_J$:
\[
i_{V_I}\Omega=d\psi_{V_I},\quad i_{V_J}\Omega=d\psi_{V_J}.
\]
At the same time, by definition of $V_I$ and $V_J$  we have $V_I = \tfrac{1}{2} \left(\theta_I^{\#} - \N f \right)$, $V_J = \tfrac{1}{2} \left( \theta_J^{\#} - \N f \right)$.  Thus by \eqref{f:thetaPhi} we see that $2(V_I - V_J)$ is $\Omega$-Hamiltonian with potential function $\Phi$, therefore:
\begin{align*}
\Phi =&\ 2\psi_{V_I} - 2\psi_{V_J}+\mbox{const}.
\end{align*}

If $n=2$ and $(M,g,I,J)$ is a rank one soliton, then there exists a vector field $X$ preserving the GK structure and constants $a_I,a_J$ such that $V_I=a_IIX$ and $V_J=a_JJX$. By the definition of $\mu_2$ and $\mu_3$, we have
\[
i_{V_I}\Omega=a_Id\mu_2,\quad i_{V_J}\Omega=a_Jd\mu_3.
\]
Hence, by the first part of the proposition, $\Phi =2{a_I} \mu_2 - 2{a_J} \mu_3+\mbox{const}$. Introducing new constants $a_+:=2a_I-2a_J$, $a_-:=2a_I+2a_J$  we can rewrite $\Phi$ as
\begin{align*}
\Phi = {a_+} \mu_+ + {a_-} \mu_-+\mbox{const}.
\end{align*}
To find the soliton function $f$, we observe that $- \N f = V_I + V_J$, therefore
\begin{align*}
i_{\N f} \Omega = -{a_I} d \mu_2 -{a_J} d \mu_3=-\tfrac{1}{2}(a_-d\mu_++a_+d\mu_-),
\end{align*}
where the second identity is the definition of constants $a_-$ and $a_+$. Referring to the explicit form of $\Omega$ in (\ref{f:Omegaexplicit}) then gives the formula for $\N f$. Since $g=Wh+W^{-1}\eta^2$, the latter yields the claimed expression for $df$.

Conversely, assume that in the Gibbons-Hawking ansatz of Theorem~\ref{t:nondegenerate_gk_description_v2} we have locally
\[
\Phi=a_+\mu_++a_-\mu_-+\mbox{const}
\]
for some constants $a_+$ and $a_-$. Then $p$ is also a function of $a_+\mu_++a_-\mu_-$, forcing the 1-form $-a_+d\mu_++a_-d\mu_-+p(a_+d\mu_++a_-d\mu_-)$ to be closed. Assume further that this form is exact and equals $df$ for some function $f$. We claim that the corresponding GK structure $(M,g,I,J)$ is a GK soliton solving~\eqref{f:solitonB} with the soliton function $f$.

First we note that starting with the exact expressions for $\Phi$ and $df$ and following the above computations backwards, one concludes that $V_I=\tfrac{1}{2}(\theta_I-\N f)$ and $V_J=\tfrac{1}{2}(\theta_J-\N f)$ are given by
\[
V_I=a_IIX,\quad  V_J=a_J JX.
\]
Since $IV_I$ is Killing and holomorphic, it remains to prove that
\[
\rho_I^{1,1}=\mc L_{V_I}\gw_I.
\]
Using  the relation between $\gw_I=g(I\cdot,\cdot)$ and $\Omega=2[I,J]^{-1}g$ we find:
\[
\gw_I=\tfrac{1}{2}I[I,J]\Omega=-(J\Omega)^{1,1}.
\]
Together with the identity $\rho_I=-\tfrac{1}{2}dJd\Phi$ this gives
\[
\begin{split}
\rho_I^{1,1}-\mc L_{V_I}\gw_I&=
-\left(\tfrac{1}{2}dJd\Phi+\mc L_{V_I}(J\Omega)\right)^{1,1}=
\left(-dJi_{V_I-V_J}\Omega+d i_{V_I}J\Omega\right)^{1,1}\\&=
\left(-dJi_{V_I}\Omega+di_{V_I}J\Omega \right)^{1,1}=0
\end{split}
\]
where in the third identity we used the fact that $Ji_{V_J}\Omega=-a_Ji_X\Omega$ is closed. This proves the claim that $(M,g,I,J)$ with $\Phi$ locally given by $a_+\mu_++a_-\mu_-+\mbox{const}$ is a rank one GK soliton, as long as the closed form $-a_+d\mu_++a_-d\mu_-+p(a_+d\mu_++a_-d\mu_-)$ is exact.
\end{proof}
\end{prop}

\begin{ex}[Rank one solitons on diagonal Hopf surfaces]\label{ex:hopf_diagonal_soliton}
	Let $(M,g,I,J)$ be a GK structure on a diagonal Hopf surface as in Example~\ref{ex:hopf_diagonal_gk} defined by a single scalar function $p$. Assume additionally that the parameters $a,b\in\R$ satisfy $a/b=m^2/n^2$, where $m,n$ are coprime integers, and consider the $S^1$ action on $M$ induced by the action on $\C^2$:
	\[
	u\cdot (z_1,z_2)=(u^mz_1,u^nz_2).
	\]
	Then using the expression~\eqref{f:diagonal_hopf_mu} for the moment map $\mu$, we find
	\begin{equation}
	\begin{split}
	\frac{2}{n}\mu_++\frac{2}{m}\mu_-=\left(1+\frac{a}{b}\right)\left(\frac{b}{a}x_1-x_2\right)+\left(1-\frac{a}{b}\right)\frac{\chi}{2},
	\end{split}
	\end{equation}
	where, as in Example~\ref{ex:hopf_diagonal_gk}, $\chi$ is the antiderivative of $p$ evaluated at $2(\tfrac{b}{a}x_1-x_2)$.
	
	The above calculation makes sense for any function $p$ defining the GK structure on $M$. However, as we proved in~\cite{SU}, $(M,g,I,J)$ is a soliton if and only if $p$ solves an ODE,
	\[
	\left(\log\frac{1-p}{1+p}\right)'=\frac{1}{2}\left(1-\frac{a}{b}\right)p+\frac{1}{2}\left(1+\frac{a}{b}\right)
	\]
	which is equivalent to the identity:
	\[
	\Phi = \frac{2}{n}\mu_++\frac{2}{m}\mu_-+\mbox{const},
	\]
	in accordance with Proposition~\ref{p:invGKsoliton}. In particular, the solitons constructed in~\cite{SU} can be viewed as rank one solitons. (Strictly speaking we have established this claim only in the case $a/b=m^2/n^2$, with $n,m\in\Z$. In general, if $a/b$ is arbitrary, we still can run the same argument, keeping in mind that $X=m\del_{y_1}+n\del_{y_2}$ only generates an $\R$ action).
\end{ex}

The above example provides a function $W$ solving equation~\eqref{f:W_laplace2} with $a_+=2/n$ and $a_-=2/m$. Motivated by this example, it is easy to construct an explicit solution $\til{W}$ to the equation~\eqref{f:W_laplace2} with any given constants $a_-,a_+$ determining function $p$. In what follows, we will use this \emph{baseline} solution to modify the GK structure on a given soliton.  A key point is that $\til W$ is bounded below by a positive constant and on the set $\{|p|<1-\delta\}$ we have bounds on $\til W$ and $|d\til W|_h$.

\begin{lemma}\label{l:W0_solution} Given constants $a_+,a_-\in\R$ and setting $\Phi=\log\frac{1-p}{1+p} = a_+ \mu_+ + a_- \mu_-$, the function
	\begin{equation}\label{f:W0_def}
	\til W=(a_+^2(1+p)+a_-^2(1-p))^{-1}
	\end{equation}
	satisfies equation~\eqref{f:W_laplace2}.
\end{lemma}
\begin{proof}
	As both $p$ and $\til W$ are functions of the same linear combination of $\mu_+$ and $\mu_-$, we obtain
	\begin{align*}
	\til W_{11} + \tfrac{1}{2} \left( (1+p) \til W \right)_{++} + \tfrac{1}{2} \left( (1-p) \til W \right)_{--} =&\ \tfrac{1}{2} \left[ a_+^2 ((1 + p) \til W)'' + a_-^2 ((1-p) \til W)'' \right]\\
	=&\ \tfrac{1}{2} \left[ \left( a_+^2 (1 + p) + a_-^2 (1 - p) \right) \til W \right]''=0,
	\end{align*}
	as claimed.
\end{proof}

\section{Analysis of the completion} \label{s:completion}

Let $(M,g,I,J)$ be a GK soliton of rank one. Following Remark~\ref{r:rank1_s1} we assume that one of the vector fields $IV_I$ or $JV_J$ generates an $S^1$-action on $M$ preserving the GK structure. We impose a lower bound on the Ricci curvature $\Rc^{M,g}$, and finiteness of the set of isolated fixed points. The assumptions imposed on $(M,g,I,J)$ in the rest of the paper are summarized in the following definition.

\begin{defn}[Complete regular rank one solitons]\label{d:regular_rank_one}
	Let $(M^4,g,I,J)$ be a generalized K\"ahler manifold such that
	\begin{enumerate}
		\item $(M,g)$ is complete;
		\item The Poisson tensor $\gs$ of $(M,g,I,J)$ is not identically zero;
		\item $(M,g,I,J)$ is a \emph{rank one} soliton, i.e. it satisfies the equations of Definition~\ref{d:GKsoliton} and $\mathrm{span}\{IV_I,JV_J\}$ is generically one-dimensional;
		\item Either $IV_I$ or $JV_J$ generates an $S^1$ action preserving the GK structure $(M,g,I,J)$ with finitely many fixed points. The latter is automatically satisfied if $\dim H^*(M,\R)<\infty$;
		\item On the nondegeneracy locus of $\gs$, 
		\begin{align*}
		\Phi=\log\frac{1-p}{1+p}=a_+\mu_++a_-\mu_-+\mbox{const},
		\end{align*} where $\mu_+$ and $\mu_-$ are defined locally up to an additive constant, and without loss of generality we assume that $a_+\neq 0$;
		\item There exists a constant $k$ such that $\Rc^{M,g}>-k^2$;
	\end{enumerate}
	Then we call $(M,g,I,J)$ a \emph{complete regular rank one soliton}.
\end{defn}
Our ultimate goal is to give an exhaustive classification of such GK manifolds. So far we have obtained an explicit local description in the nondegeneracy locus of $\gs$ in terms of the functions $W$ and $p$, where now, thanks to Proposition \ref{p:invGKsoliton}, $p$ is determined by the choice of constants ${a_+}, {a_-}$. To address the global structure of complete rank one solitons we have to understand how these local models can be patched together, and to determine their behavior and possible extension to the degeneracy locus.  In the end we will obtain a global definition of the moment map together with a complete classification of the possible images.

\subsection{Local structure near the locus \texorpdfstring{$\{I=\pm J\}$}{}}

In this subsection we determine the local structure of a GK soliton near the degeneracy locus for the Poisson tensor $\gs$. Since $\gs$ does not vanish identically, $I$ and $J$ induce the same orientation, and $\gs$ will vanish precisely when $I = \pm J$.  Define
\[
\mathbf T_+:=\{x\in M\ |\ I_x= J_x \},\quad \mathbf T_-:=\{x\in M\ |\ I_x=- J_x \}.
\]
The sets $\mathbf T_{\pm}$ are one-dimensional complex analytic subsets of $M$ with respect to either complex structure.  Our aim is to understand the local structure of the generator of the $S^1$ action $X$ and the symplectic form $\Omega$ in the image of $\pmb{\mu}$, near points in $\mathbf{T}=\mathbf{T}_+\cup \mathbf{T}_-$. 
Note that the complex-valued form $\Omega_I:=\Omega-\sqrt{-1}I\Omega$, initially defined on $M \backslash\ \mathbf T$, extends to an $I$-meromorphic $(2,0)$ form on $M$, and similarly for $\Omega_J$.  In this section we let
\begin{align*}
\mathbf X = \tfrac{1}{2} \left(X - \i I X \right)
\end{align*}
denote the complex holomorphic vector field associated to $X$.
\begin{prop} \label{p:Omegapoles} Given $(M^4, g, I, J)$ a complete regular rank one soliton, the forms $\Omega_{I}, \Omega_J$ have poles of order $1$ at $\mathbf{T}$, and $\mathbf{X}$ does not vanish identically along $\mathbf{T}$.
\end{prop}

The proof will consist of several lemmas.  Our analysis is local, and we can choose local $I$-holomorphic coordinates $z,w$ on a neighborhood $U$ of $p \in \mathbf{T}_+$ such that $g_{i\bj}(p) = \gd_i^j$ and 
\[
\Omega_I=w^{-k}dz\wedge dw + \mbox{higher order terms},\quad k\geq 1,
\]
and $\mathbf T_+\cap U=\{w=0\}$.

\begin{lemma} \label{l:deglocuslemma1}
	If $X\neq 0$ at $x_0\in\mathbf T_+$, then $k=1$.
\end{lemma}
\begin{proof}
	For two real-valued functions $f_1,f_2$ we will write $f_1 \sim f_2$ if there is a positive constant $C>0$ such that $C^{-1}f_1<f_2<Cf_1$.
	Let $d$ denote the distance from the divisor $\mathbf T_+$:
	\[
	d(x):=\mathrm{dist}_g(x,\mathbf T_+).
	\]
	In the $I$-holomorphic coordinates $(z,w)$ as above centered at $x_0$ we have $d(x) \sim \brs{w}$, thus in a neighbourhood of $\mathbf T_+$ we will have
	\[
	\frac{\Omega_I \wedge \bar{\Omega}_I}{dV_g}\sim d^{-2k},
	\]
	where $dV_g$ is the volume form associated with $g$.
	
	Fix an $S^1$-invariant neighborhood $U$ of $x_0$.  By~\eqref{f:Omega_matrix} and \eqref{f:coframes} at $x\in U\backslash\mathbf T_+$ we have
	\[
	\Omega=W^{-1}\left(X^*\wedge (KX)^*+(IX)^*\wedge(JX)^*\right),
	\]
	so that in the local coordinates $\{\mu_i\}$,
	\[
	\Omega_I \wedge \bar{\Omega}_I = \Omega\wedge\Omega=C W\eta\wedge d\mu_1\wedge d\mu_+\wedge d\mu_-.
	\]
	Now, since the differentials
	\[
	W^{-1/2}\eta, \quad (1-p^2)^{1/2}W^{1/2}d\mu_1, \quad (1-p)^{1/2}W^{1/2}d\mu_+, \quad (1+p)^{1/2}W^{1/2}d\mu_-
	\]
	form an orthonormal basis of $T^*_xM$, we have
	\[
	\Omega\wedge\Omega=C(1-p^2)^{-1}dV_g.
	\]
	
	In a neighbourhood of $\mathbf{T_+}$ function $1+p$ is bounded from above and away from zero, therefore, it remains to estimate $d(x)$ in terms of $1-p$. Since $X\neq 0$ at $x_0$, we can shrink $U$ to ensure that $X \neq 0$ in $U$, and obtain a constant $C>0$ such that $W:=g(X,X)^{-1}$ satisfies
	\[
	C^{-1}<W<C
	\]
	in $U$. Then
	metric $g$ is uniformly equivalent in $U$ to the metric
	\[
	\til g=h+\eta^2.
	\]
	Let $\gamma\subset U$ be a horizontal geodesic with respect to $\til g$ connecting $x_0\in \mathbf{T}_+$ to a point $x_1\in U$ which realizes the shortest distance between $\mathbf{T}_+$ and $x_1$. In particular, $d(\gamma(t))\sim t$. By Theorem~\ref{t:nondegenerate_gk_description} the universal cover of ${U\backslash\mathbf{T}_+}$ admits a well-defined moment map
	\[
	\pmb\mu\colon \til{U\backslash\mathbf{T}_+}\to\R^3_{\pmb\mu}.
	\]
	Under this map the geodesic $\gamma$ is isometrically mapped to an $h$-geodesic $\gg_0:=\pmb\mu(\gg)$. Since $\mathbf T_{+}=\{p=1\}$, we have $p(\gg_0(t))\to 1$ as $t\to 0$. By definition of $\Phi$ we have
	\[
	p=\frac{1-e^\Phi}{1+e^\Phi},
	\]
	thus it follows from Proposition \ref{p:invGKsoliton} that along $\gg_0(t)$ we have $\Phi=a_+\mu_++a_-\mu_-\to -\infty$ as $t\to 0$. Given the explicit form of the metric~\eqref{f:h_diagonal}, we conclude that $a_+\neq 0$, since otherwise $\gg_0(t)$ would have an infinite length. So without loss of generality we assume $a_+>0$.
	Then the length of $\gg_0$ with respect to $h$ is bounded from above by the length of the curve $\gamma_1$ with $\mu_1=\mbox{const}$, $\mu_-=\mbox{const}$ and $\mu_+\to -\infty$. Thus we estimate
	\[
	d(x_1)\sim L_{h}\gg_0\leq L_h\gg_1=
	\int_{\gg_1} (1-p)^{1/2}d\mu_+	
	\sim
	\int_{-\infty}^{\mu_+}\left(\frac{e^{a_+s}}{1+e^{a_+s}}\right)^{1/2}ds\sim e^{a_+\mu_+/2}\sim(1-p)^{1/2}.
	\]
	Since $\til g$ and $g$ are equivalent, we have $d(x) \leq C (1-p)^{1/2}$, so that along $\gamma$ we have
	\[
	d(x)^{-2k} \sim \frac{\Omega\wedge\Omega}{dV_g}=C(1-p^2)^{-1}\leq C' d(x)^{-2}.
	\]
Thus $k = 1$, as claimed.
\end{proof}

\begin{lemma} \label{l:deglocuslemma2}
	If $X$ vanishes identically along $\mathbf T_+$, then $k=1$.
\end{lemma}
\begin{proof}
	Given that $X$ vanishes on $\mathbf T_+$, then at any point $x\in \mathbf T_+$ there is a local chart $U\simeq \C^2$ provided by the slice theorem such that the action of $S^1$ on $\C^2$ has weights $\mathbf{w}=(\alpha,0)$.  We claim that $\alpha\neq 0$. Indeed, if $\alpha=0$ then the vector field $X$ would act trivially in $U\simeq \C^2$, and being Killing, $X$ must preserve all geodesics emanating from $x$. Thus $X$ will be identically zero entirely on $M$, which contradicts our basic assumptions on $X$.
	
	Given the weights of $X$ at $x$ and knowing that $X$ vanishes on $\mathbf{T}_+$, we can pick local $I$-holomorphic coordinates $(z,w)$ near $x$ such that $\mathbf{T_+}=\{w=0\}$, $x=(0,0)$ and in $T_xM$,
	\[
	[X,\del_z]=0,\quad [X,\del_w]=\sqrt{-1}\alpha \del_w.
	\]	
	Then locally we can write the holomorphic vector field $\mathbf{X}=X-\sqrt{-1}IX$ as
	\[
	\mathbf X=wf_1(z)\del_w+(\mbox{higher order terms in $w$}),\quad f_1(0)=\sqrt{-1}\alpha\neq 0.
	\]
	In these coordinates we still have
	\[
	\Omega_I = w^{-k} dz \wedge dw+(\mbox{higher order terms in }w),
	\]
	and thus
	\[
	i_X\Omega_I=\frac{1}{w^{k-1}}(f_1(z)dz+f_2(z)dw)+(\mbox{higher order terms in }w)
	\]
	Since $f_1(0)\neq 0$, this form cannot be closed, unless $k = 1$.
	\end{proof}

\begin{lemma}\label{l:X_nonvanish} $\mathbf{X}$ does not vanish identically along $\mathbf{T}_+$.
	\begin{proof} It follows from Lemmas \ref{l:deglocuslemma1} and \ref{l:deglocuslemma2} that $\Omega_I$ has a pole of order one.  As above we choose local $I$-holomorphic coordinates so that
	\[
	\Omega_I = w^{-1} dz \wedge dw + (\mbox{higher order terms in $w$}),
	\]
	If $\mathbf{X}$ vanishes identically along $\mathbf{T}_+$ then we can express
	\[
	\mathbf X= w \left( f_1 \del_z + f_2 \del_w \right), \qquad f_i \in C^{\infty}.
	\]
	It follows that
	\begin{align*}
	i_{\mathbf{X}} \Omega_I = f_1(z) dw - f_2(z) dz + (\mbox{higher order terms in $w$}) 
	\end{align*}
	is smooth across $\mathbf{T}_+$. Therefore $\int_\gamma i_{\mathbf X}\Omega_I$ is finite, where $\gamma$ is any path connecting $x_0\in\mathbf{T}_+$ and $x_1\in M\backslash\mathbf{T}_+$. Using a similar argument for $\Omega_J$, we conclude that the integrals
	\[
	\int_\gamma i_XI\Omega,\quad \int_{\gamma} i_XJ\Omega
	\]
	are also finite. On the other hand, the pullbacks of $i_XI\Omega, i_XJ\Omega$ to $\til{M\backslash \mathbf{T}_+}$ are $d \mu_2$ and $d\mu_3$. Since $p$ is bounded away from $1$ on any compact subset of $\R^3_{\pmb\mu}$, the integral of $a_Id\mu_2+a_Jd\mu_3$ over any curve $\gamma_0\subset \R^3_{\pmb\mu}$ escaping to the locus $\{p=1\}$ is infinite, giving a contradiction.
\end{proof}
\end{lemma}

\begin{proof}[Proof of Proposition \ref{p:Omegapoles}] Lemmas \ref{l:deglocuslemma1} and \ref{l:deglocuslemma2} gives that $\Omega_I$ has a pole of order $1$, and then Lemma \ref{l:X_nonvanish} gives that $\mathbf{X}$ does not vanish identically.
\end{proof}

\subsection{Completeness properties of the quotient} \label{ss:compquot}

Now we prove a general statement about Riemannian 4-manifolds with $S^1$ action, which might be of independent interest.

\begin{prop}\label{p:s1_bundle_complete}
	Let $(M,g)$ be a connected smooth complete Riemannian 4-dimensional manifold with boundary admitting
	an isometric $S^1$ action with isolated fixed points $\{y_i\}$.
	Let 
	\[\pi\colon M\to N\] be the projection onto the orbit space, and set $z_i=\pi(y_i)$. Assume that the orbifold with boundary $N_0=N\backslash\{z_i\}$ is
	equipped with a Riemannian metric $h$ and a smooth 2-form $\beta_0$ such that
	\begin{enumerate}
		\item the curvature 2-form of the $S^1$-bundle $M\backslash\{y_i\}\to N_0$ is given by a closed 2-form
		\[
		\beta=*_hdW+W\beta_0,
		\]
		where $W$ is a smooth function on $N_0$;
		\item there are constants $c_W,C_{\beta_0} >0$ such that $W>c_W$ and $|\beta_0|_h+|d\beta_0|_h<C_\beta$;
		\item $g$ on $M\backslash\{y_i\}$ is given by
		\[
		g=Wh+W^{-1}\eta^2
		\]
		where $\eta$ is the connection 1-form with curvature $\beta$;
		\item $\{z_i\}\subset N$ is complete with respect to the distance function $d_h$ induced by $h$;
		\item $|\Rm^{N_0,h}|<b^2$ for some constant $b$;
		\item $\Rc^{M,g}>-k^2$ for some constant $k$.
	\end{enumerate}
	Then any closed subset $K\subset N$ such that $d_h(K,\del N_0)>0$ is complete with respect to $d_h$.%
\end{prop}

\begin{rmk}
	The assumptions of the above proposition seem rather restrictive. However, as we will show later, aside from the lower Ricci curvature bound, they are automatically satisfied on complete regular rank one solitons.
\end{rmk}
\begin{proof}

	Let $\del N_0$ be the usual boundary of an orbifold $N_0$, and for a subset $K\subset N$ denote by $\del_h K$ its $d_h$-completion boundary, i.e.,
	\[
	\del_h K:=\bar{(K,d_h)}\backslash K,
	\]
	where $\bar{(K,d_h)}$ denotes the metric completion of $(K,d_h)$.
	
	We will prove the completeness by contradiction. The proof is rather technical, so we start with a brief overview. The first step is to get a gradient estimate for $W$. The function $W$ satisfies a second order elliptic PDE on a complete manifold $(M,g)$ which allows for the application of the standard method of Cheng-Yau \cite{ChengYau} using the lower bound on $\Rc^{M,g}$. With the gradient estimates at hand, we exploit completeness of $(M,g)$ and use an argument of Schoen-Yau~\cite{SchoenYau} to get a lower bound on the blow up rate of $W$ near $\del_h N$. Finally, relying on the control over the geometry of $(N_0,h)$ we prove that averages of $W$ over geodesic spheres in $(N_0,h)$ blow up as the boundary of the sphere approaches $\del_h N$, and obtain a contradiction with the mean value inequality. This implies that $K$ must be complete.
	
	\medskip
	\noindent\textbf{Step 1} (Gradient estimate for $W$)
	\noindent
	Assume that some closed $K\subset N$ with $d_h(K,\del N_0)>0$ is not complete so that $\del_h K$ is not empty. Since $d_h(K,\del N_0)>0$, there is a point in $x_\infty\in \del_h K$ such that $d_h(x_\infty,\del N_0)>0$, and as $K\subset N$ is closed, $x_\infty\in \del_h N$.
	
	The set $\{z_i\}$ is complete, therefore we can find $\epsilon>0$ such that
	\[
	d_h(x_\infty,\{z_i\}\cup \del N_0)>\epsilon.
	\]
	Pick a point $x_0\in N_0$ such that $d_h(x_0,x_\infty)<\epsilon/3$. Then $d_h(x_0,\{z_i\}\cup \del N_0)>2\epsilon/3$ and $d_h(x_0,\del_h N)<\epsilon/3$. Since $d_h(x_0, \{z_i\}\cup \del N_0)>2\epsilon/3$ and $h$ is a Riemannian metric on orbifold $N_0$, we can choose the largest $h$-geodesic ball around $x_0$ not containing points in $\{z_i\}\cup \del N_0\cup \del_h N$, and call it $B(x_0)$. Since $d_h(x_0,\del_h N)<\epsilon/3$ and $d_h(x_0, \{z_i\}\cup \del N_0)>2\epsilon/3$, there exists a point $x^*\in \del_h N$ such that $x^*\in \bar{B(x_0)}$. Thus, as $d_h(x_\infty, x^*)<2\epsilon/3$, we have that
	\[
	d_h(x^*,\{z_i\}\cup \del N_0)>\epsilon/3.
	\]
	
	Let $r=\epsilon/6$. We are going to prove a gradient estimate for $W$ in $U_r(x^*)=\{x\in N\ |\ d_h(x,x^*)\leq r\}$. Our plan is to apply the Cheng-Yau local gradient estimate~\cite{ChengYau} to $W$ defined on  $\pi^{-1}(U_r(x^*))\subset M$, where
	\[
	\pi\colon M\to N
	\]
	is the natural projection onto the orbit space. Recall that the metric $g$ on $M$ is given by $Wh+W^{-1}\eta^2$ with $W>c_W>0$. Furthermore, $d_h(x^*,\{z_i\}\cup\del N_0)\geq \epsilon/3=2r$. Therefore for any $y\in \pi^{-1}(U_r(x^*))$, the $g$-geodesic ball of radius $c_Wr$ centered at $y$ does not intersect the boundary $\del M$ and the fixed point set $\{y_i\}$.

	To apply the local Cheng-Yau estimate to $W$ we first note that for any function $f\colon M\to \R$ invariant under the $S^1$-action we have
	\[
	\Delta_{g}f=W^{-1}\Delta_{h}f,
	\]
	therefore $W$ solves
	\[
	\Delta_{g}W=W^{-1}\Delta_h W=-W^{-1}(\langle dW, *\beta_0\rangle_h+(*_hd\beta_0)W).
	\]
	We know that $W$ is bounded from below by $c_W>0$ and $|\beta_0|_h+|d\beta_0|_h<C_{\beta_0}$. It is straightforward now to check that there exists $C'>0$ such that $W$ satisfies the differential inequalities
	\begin{equation}
	\begin{split}
	\Delta_{g}W&\leq C'(|dW|_{g}+W),\\
	|\nabla_{g}(\Delta_{g}W)|_{g}&\leq C'
	\left(|dW|_{g}+W^{-1}|d W|^2_{g}+|\nabla^2_{g}W|_{g}+W\right),
	\end{split}
	\end{equation}
	which allow to run the proof of~\cite[Theorem 6]{ChengYau}. Therefore there exists a constant $C>0$ such that
	\[
	|d W|_{g}<CW
	\]
	in $\pi^{-1}(U_r(x^*))$. Translating this inequality back to $h$ we find that in $U_r(x^*)$
	\begin{equation}\label{f:W_gradient_estimate}
	|dW|_h<CW^{3/2}.
	\end{equation}
	
	\medskip
	\noindent\textbf{Step 2} (Lower bound on the growth rate of $W(x)$ as $x\to x^*$)
	Next we exploit the upper bound of ~\eqref{f:W_gradient_estimate} together with the completeness of $(M,g)$ to obtain a lower bound on the growth of $W$ as its argument approaches $x^*$. The manifold $(M,g)$ is complete, therefore the distance function $d_{Wh}$ on $N$ induced by the metric $Wh$ on $N\backslash\{z_i\}$ is complete.  Let $\gg\colon [0,1)\to N$ be any smooth curve of finite $h$-length, avoiding $\{z_i\}$, such that $\lim_{t\to 1}\gg(t)=x^*$. Since $(N,d_{Wh})$ is complete, and $x^*\not\in N$, we know that
	\[
	\ell_{Wh}(\gg)=\int_{0}^{1} W^{1/2}|\gg'(t)|_{h}dt=\infty.
	\]
	This implies that $W\to +\infty$ subsequentially along any curve approaching $x^*$. We now use the gradient estimate for $W$ to upgrade this subsequential blow up to a uniform lower bound on the growth rate.
	
	Consider any curve $\gg$ in ${U_r}(x^*)$ approaching $x^*$. Pick two points $x,x'$ on $\gg$. Then
	\[
	W^{-\frac{1}{2}}(x)-W^{-\frac{1}{2}}(x')=\int_x^{x'} h\left(\nabla W^{-\frac{1}{2}}, \gg'(t)\right) dt\leq \int_x^{x'}\frac{1}{2}|\nabla W|_h W^{-\frac{3}{2}}dt.
	\]
	Using the gradient estimate~\eqref{f:W_gradient_estimate}, we conclude
	\[
	W^{-\frac{1}{2}}(x)-W^{-\frac{1}{2}}(x')\leq C'' \ell_h(\gg).
	\]
	Picking $x'$ along a sequence $x_i\in\gg$ such that $W(x_i)\to \infty$, we conclude that
	\[
	W(x)\geq \frac{C}{\ell^2_h(\gg)}.
	\]
	As the curve $\gg\subset U_r(x^*)$ above is arbitrary, we conclude that for any $x\in U_r(x^*)$
	\begin{equation}\label{f:W_blowup_bound}
	W(x)\geq \frac{C}{d_{h}^2(x,x^*)}.
	\end{equation}
	
	\medskip
	\noindent\textbf{Step 3} (Contradiction with the mean value theorem)
	
	Finally, we show that the function $W$ satisfying the elliptic equation $d(*_hdW+W\beta_0)=0$ can not have a boundary blow up as in~\eqref{f:W_blowup_bound} using the bound $|\Rm^{N,h}|<b^2$.  Recall that we have a ball $B(x_0)\subset N$ such that $x^*\in\del B(x_0)$. Let $r_0:=d_h(x_0,x^*)$ be the radius of $B(x_0)$. Pick a new center $x_1$ along the $h$-geodesic segment $[x_0, x^*]$ such that
	\[
	r_1:=d_h(x_1,x^*)<\min(r_0, \pi/4b).
	\]
	Denote by $B_{r_1}(x_1)\subset N$ the $h$-geodesic ball of radius $r_1$ centered at $x_1$ and by $B^T_{r_1}(x_1)\subset T_{x_1}N$ the ball of radius $r_1$ in the tangent space $(T_{x_1}N, h_{x_1})$. By Rauch's comparison theorem, the exponential map
	\[
	\exp_{x_1}^h\colon B^T_{r_1}(x_1) \to B_{r_1}(x_1)
	\]
	is an immersion and there is a constant $\delta_b>0$ such that the norm of the differential of $\exp_{x_1}^h$ satisfies bounds:
	\[
	C_0^{-1} < |D(\exp_{x_1}^h)| < C_0,\quad C_0=C_0(b)>0.
	\]
	We use $\exp_{x_1}^h$ to pull back $W$, $\beta_0$ and $h$ to $B^T_{r_1}(x_1)\subset T_{x_1}N$. The estimate for $D(\exp_{x_1}^h)$ implies that the metric $h_{x_1}$ in $B^T_{r_1}(x_1)\subset T_{x_1}N$ is uniformly equivalent to the pull back of $h$. The estimates to follow are conducted inside $B^T_{r_1}(x_1) \subset T_{x_1} N$ with respect to the metric $(\exp_{x_1}^h)^* h$ which we still denote $h$.
	
	
	Let $\rho$ be the $h$-distance function from $x_1$, and let $G=G(\rho)$ be a function to be chosen later. Recall that $W$ solves
	\[
	L(W):=*_hd(*_hdW+ W\beta_0)=0.
	\]
	Then by Green's identity, for any $R<r_1$,
	\begin{equation}
	\begin{split}
	\int_{B_R(x_1)} (L(W)G-WL^*(G))\,d\mu_h=
	\int_{\del B_R(x_1)} \left((*_hdW+W\beta_0)G-W\frac{\del G}{\del\nu}\right)\,d\gs_h,
	\end{split}
	\end{equation}
	where
	\[
	L^*(G):=*_h(d*_hdG-dG\wedge\beta_0)
	\]
	is the dual operator. Using that $L(W)=0$ and
	\[
	G\Big|_{\del B_R(x_1)}=G(R),\quad \frac{\del G}{\del \nu}\Big|_{\del B_R(x_1)}=G'(R),
	\]
	we find 
	\[
	\int_{\del B_R(x_1)} (*_h dW+W\beta_0)Gd\gs_h=
	G(R)\int_{\del B_R(x_1)} (*_h dW+W\beta_0)d\gs_h=0.
	\]
	Therefore
	\[
	G'(R)\int_{\del B_R(x_1)}W\,d\gs_h=\int_{B_R(x_1)} W L^*(G)\,d\mu_h.
	\]
	In particular for any $0<R_1<R_2<r_1$ we have
	\begin{equation}\label{f:mean_value_ineq}
	G'(R_1)\int_{\del B_{R_2}(x_1)}W\,d\gs_h-G'(R_1)\int_{\del B_{R_2}(x_1)}W\,d\gs_h=\int_{B_{R_2}(x_1)-B_{R_1}(x_1)} W L^*(G)\,d\mu_h.
	\end{equation}
	Now consider a function
	\[
	G(\rho)=1-e^{-\gl \rho}.
	\]
	By the Laplacian comparison theorem, $\Delta \rho $ is bounded from above for $\rho\in (r_1/2, r_1)$. Then given a bound on $|\beta_0|_h$ it is easy to check that for a large constant $\gl>0$,
	\[
	L^*(G)<0
	\]
	in the annulus $B_{r_1}(x_1)-B_{r_1/2}(x_1)$. Therefore identity~\eqref{f:mean_value_ineq} implies that for any $R\in(r_1/2,r_1)$ we have a mean value inequality
	\begin{equation}\label{f:mean_value_ineq_2}
	\int_{\del B_{R}(x_1)}Wd\gs_h<C \int_{\del B_{r_1/2}(x_1)}W\,d\gs_h
	\end{equation}
	where $C$ depends only on $r_1$ and operator $L$. In particular the integral on the left hand side is bounded from above as $R\to r_1$.
	
	Fix $R<r_1$ and let $x_N\in \del B_{R}(x_1)$ be the point which is $h$-closest to $x^*$. Pick spherical coordinates $(\phi,\psi)\in [0;2\pi)\times (-\pi/2;\pi/2)$ such that $\phi$ is a longitude, $\psi$ is an latitude $x_N$ is the northern pole. Then there exists a constant $C_1$ independent of $R\in(r_1/2,r_1)$ such that for a point $p\in \del B_{R}(x_1)$ with coordinates $(\phi,\psi)$ we have
	\[
	d_h(p,x^*)<C_1\sqrt{(r_1-R)^2+(\pi/2-\psi)^2}.
	\]
	Now we estimate
	\begin{equation}
	\begin{split}
	\int\limits_{\del B_{R}(x_1)}Wd\gs_h&\geq C\int\limits_{\del B_{R}(x_1)}Wd\gs_{h_1}=CR^2 \int_{-\pi/2}^{\pi/2}\left(\int_{0}^{2\pi} W \,d\phi\right) \cos\psi\,d\psi\\ &\geq
	C'\int_{-\pi/2}^{\pi/2} \frac{\pi/2-\psi}{(r_1-R)^2+(\pi/2-\psi)^2}d\psi\geq -C''\log(r_1-R),
	\end{split}
	\end{equation}
	where in the second inequality we used the blow up estimate~\eqref{f:W_blowup_bound} and uniform equivalence of $h$ and $h_{x_1}$ in $B^T_{r_1}(x_1)$. The final term is unbounded as $R\to r_1$, contradicting the mean value inequality~\eqref{f:mean_value_ineq_2}. Thus we proved that $K\subset N$ is complete for any closed $K$ such that $d_h(K,\del N_0)>0$.
\end{proof}

\subsection{Global definition of the moment map on the non-degeneracy locus of \texorpdfstring{$M$}{}}

	Suppose $(M^4,g,I,J)$ is a complete regular rank one GK soliton.  Considering the nondegenerate domain $M\backslash \mathbf{T}$, $\mathbf{T}=\mathbf T_+\cup \mathbf T_-$, the construction of Theorem~\ref{t:nondegenerate_gk_description_v2} defines a moment map
	\[
	\pmb\mu\colon \til{M\backslash \mathbf{T}}\to \R^3_{\pmb\mu}.
	\]
	on an appropriate covering space $\Z^k\to\til{M\backslash \mathbf{T}}\to M\backslash \mathbf{T}$.  Our goal in this section is to describe the structure of the covering space $\til{M\backslash \mathbf{T}}$, and show that the moment map $\pmb{\mu}$ descends to the $\Z^k$-quotient space.

	\begin{prop} \label{p:deck_transform_mu} Given the setup above, the map $\pmb{\mu}$ descends to a map
	\begin{equation}\label{f:proj_nondegenerate}
	M\backslash \mathbf{T}\to (M\backslash \mathbf{T})/S^1\simeq \R^3_{\pmb\mu}/\Gamma,
	\end{equation}
	where $\Gamma\simeq \Z^k$ acts on $(\R^3_{\pmb\mu},h)$ by isometric translations freely, and properly discontinuously. There are two possibilities for the action of $\Gamma$ on $\R^3_{\pmb\mu}$:
	\begin{enumerate}
		\item $\Gamma_1\simeq \Z$ is generated by $g_1$ such that
		\[
		g_1(\mu_1,\mu_+,\mu_-)=(\mu_1+c_1,\mu_+,\mu_-),\quad c_1\neq 0.
		\]
		\item $\Gamma_2\simeq \Gamma_1\times \Z\simeq \Z^2$ and $a_-=0$, where $\Gamma_2$ is generated by $g_1$ and $g_2$ such that
		\begin{align*}
		g_1&(\mu_1,\mu_+,\mu_-)=(\mu_1+c_1,\mu_+,\mu_-),\quad c_1\neq 0\\
		g_2&(\mu_1,\mu_+,\mu_-)=(\mu_1+c_1',\mu_+,\mu_-+c),\quad c\neq 0.
		\end{align*}
	\end{enumerate}
	\begin{proof}
	First, let us describe the covering space
	\[
	\til{M\backslash \mathbf{T}}\to M\backslash \mathbf{T}.
	\]
	On $M\backslash \mathbf{T}$ we have three closed forms
	\[
	\alpha_1=-i_X\Omega, \quad \alpha_2=-i_X I\Omega, \quad \alpha_3=-i_X J\Omega.
	\]
	Let $K:=\{a\in H_1(M\backslash \mathbf{T}; \Z)\ |\ \langle\alpha_i,a\rangle=0, i=1,2,3\}\subset H_1(M\backslash \mathbf{T}; \Z)$. Then we have a map
	\[
	\pi_1(M\backslash \mathbf{T})\to \Gamma\to 1,
	\]
	where $\Gamma=H_1(M\backslash \mathbf{T}; \Z)/K\simeq \Z^k$.  The action $S^1\times M\to M$ is tri-Hamiltonian if and only if $\Gamma=0$. In general, $\Gamma$ is nontrivial, and in order to construct the moment map of Section~\ref{s:GKsymm} we first need to `kill' $\Gamma$ by taking a cover. Specifically, the kernel of the map $\pi_1(M\backslash \mathbf{T})\to \Gamma$ corresponds to a covering space
	\[
	\mathrm{pr}\colon \til{M\backslash \mathbf{T}}\xrightarrow{\Gamma} M\backslash \mathbf{T}
	\]
	with the deck transformation group $\Gamma$. On $\til{M\backslash \mathbf{T}}$ the forms $\mathrm{pr}^*\alpha_i$ represent the zero class in de Rham cohomology, so we can write $\mathrm{pr}^*\alpha_i=d\mu_i$ and recover the moment map $\pmb \mu$ in the construction of Section~\ref{s:GKsymm}: $\pmb\mu\colon \til{M\backslash \mathbf{T}}\to\R^3_{\pmb\mu}$. We also note that the $S^1$ action on $M\backslash \mathbf{T}$ lifts to an $S^1$ action on the covering space $\til{M\backslash \mathbf{T}}$. Indeed, for any free loop $\gamma\subset M\backslash \mathbf{T}$ represented by an orbit of $S^1$, we have $[\gamma]\in K$, since $\alpha_i(X)=0$. Therefore, the lift of $\gamma$ to $\til{M\backslash \mathbf{T}}$ is still a closed loop. It follows that the lift of the vector field $X$ to $\til{M\backslash \mathbf{T}}$ has periodic orbits.
	
	Next we claim that the image $\pmb\mu(\til{M\backslash\mathbf{T}})$ is the whole $\R^3_{\pmb\mu}$. We want to apply Proposition~\ref{p:s1_bundle_complete} to $\til{M\backslash \mathbf{T}}$. However, we can not do it directly, since, first, this manifold is not necessarily complete, and, second, $W$ might not be bounded from below by a positive constant. Thus we first modify $\til{M\backslash \mathbf{T}}$.
	
	Recall that $\mathbf{T}=\{p= \pm1\}$. For $\gd>0$ denote $\mathbf{T}_\gd=\{|p|>1-\gd\}$ and consider the complement $M\backslash \mathbf{T}_\gd\subset M\backslash \mathbf{T}$. For $\gd>0$ small enough $M\backslash \mathbf{T}_\gd$ is a connected manifold with boundary. Let $M_\gd\subset \til{M\backslash\mathbf{T}}$ be the inverse image of $M\backslash\mathbf{T}_\gd$ under the covering map $\til{M\backslash\mathbf{T}}\to M\backslash\mathbf{T}$.  By Proposition~\ref{p:invGKsoliton} characterizing rank one GK solitons, locally $\Phi=a_+\mu_++a_-\mu_-+\mbox{const}$. Since $\Phi$ and $\mu_i$ are globally defined on a connected manifold $M_\gd$, the same identity holds globally. We can absorb the constant term by redefining $\mu_i$ and assume that $\Phi=a_+\mu_++a_-\mu_-$. In particular, $p$, $h$ and $\beta_0$ all descend to $\R^3_{\pmb\mu}$.

		Next we need to modify the function $W$, and as a consequence the geometry of $M$, to ensure that $W$ is bounded from below by a positive constant.  Let $\til W$ be the baseline solution provided by Lemma \ref{l:W0_solution}. The equation for $\til W$ implies that the 2-form $\til \beta=*_h d \til W+\til W\beta_0\in \Lambda^2(M/S^1)$ is closed, and exact since it is defined on $\mathbb R^3_{\pmb \mu}$, with an antiderivative $\alpha\in \Lambda^1(N)$.  Thus we can choose a new GK structure on the manifold $M$ by setting
	\[
	W_t:=W+t\til W,\quad \eta_t:=\eta+t\alpha,\quad  t\in\R.
	\]
	\begin{lemma}\label{l:ricci_bound}
		Let $g_t=W_th+W_t^{-1}\eta_t^2$. If $(M,g)$ is complete and has $\Rc^{M,g}>-k^2$, then $(M,g_t)$ is also complete. If additionally $|\beta_0|_h, |\nabla_h\beta_0|_h, \til W, |d\til W|_h$ are bounded on $M/S^1$, then for $t>0$ large enough we have $\Rc^{M,g_t}>-k^2$.
	\end{lemma}
	\begin{proof}
		Since the fibers of the map $M\to M/S^1$ are compact, $(M,g_t)$ is complete if and only if $M/S^1$ with the distance function induced by the metric $W_th$ is complete. We know that metric $Wh$ induces a complete distance function on $M/S^1$ and since $W_t>W$ the same must be true for $W_t h$, implying that $(M,g_t)$ is complete.
		
		Now we claim that for some $t>0$ we still have
		\begin{equation}\label{eq:pf_ricci_Mt}
		\Rc^{M,g_t}>-k^2g_t.
		\end{equation}		
		It is enough to prove that bound on the subset of $M$ where $S^1$ acts freely. Let $\{E^1,E^2,E^3\}$ be an $h$-orthonormal frame of $T_xN$, $\{E_t^1,E_t^2,E_t^3\}$ are the horizontal lifts of $E_i$ normalized to the $g_t$-unit length, and $E_t^4$ is the unit vertical vector field. In order to prove the bound~\eqref{eq:pf_ricci_Mt}, we use the exact formula for $\Rc^{M,g_t}$ in terms of $\Rc^{N,h}$, $W_t$ and $\beta_t$ (see \cite{Gilkey-98}) and conclude that for any $1\leq i, j\leq 4$,
		\begin{equation}
		\Rc^{M,g_t}(E_t^i,E_t^j)=W_t^{-2} \langle dW_t,\Psi_1 \rangle_h+W_t^{-1}\Psi_2,
		\end{equation}
		where $\Psi_1$ and $\Psi_2$ are a covector and a function independent of $t$ with the $h$-norms bounded by $C(|\beta_0|^2_h+|\nabla_h\beta_0|_h)$, where $C>0$ is a universal constant.  Notably, there are no second order terms in $W$ here~--- a special artifact of this ansatz recovering that in the case $h$ is Euclidean, the metric on the total space $M$ is Ricci flat for an arbitrary harmonic function $W$.  We know that
		\begin{equation}\label{f:proof_ricci_bound}
		W^2\Rc^{M,g}(E_0^i,E_0^j)=\langle dW,\Psi_1\rangle_h+\Psi_2 W>-k^2W^2.
		\end{equation}
		If we fix
		\[
		t>\max\left(0, \sup_{\mathbf{\R^3_{\pmb\mu}}} \left\{ -\til W^{-2}\langle d\til W,\Psi_1\rangle_h-\Psi_2\til W^{-1} \right\} \right),
		\]
		then
		\[
		\langle d(t\til W),\Psi_1\rangle_h+\Psi_2 (t\til W)>-2k^2tW\til W-k^2t^2\til W^2,
		\]
		and using~\eqref{f:proof_ricci_bound} we find
		\[
		\langle dW_t,\Psi_1\rangle+\Psi_2W_t>-k^2W_t^2,
		\]
		which is equivalent to the required Ricci bound on  $(M_t,g_t)$.
	\end{proof}
	
	Furthermore, by a direct computation, on $\{|p|\leq 1-\delta\}$ all the norms $|\beta_0|_h, |\nabla_h\beta_0|_h, \til W, |d\til W|_h$ are bounded.
	Hence we can apply Lemma~\ref{l:ricci_bound} and modify the metric on $(M_\delta, g)$ obtaining $(M_\delta,g_t)$ with $\Rc^{M_\delta,g_t}>-k^2$ and $W_t>t\inf_{\R^3_{\pmb\mu}} \til W>0$. We claim that $(M_\gd, g_t)$ satisfies the assumptions of Proposition~\ref{p:s1_bundle_complete}:
	\begin{enumerate}
		\item $(M\backslash\mathbf{T}_\gd, g)$ is a closed subset of a complete manifold, therefore it is itself complete. $M_\gd$ is a covering space over $M\backslash\mathbf{T}_\gd$, therefore it is also complete with respect to $g$. By Lemma~\ref{l:ricci_bound} $M_\delta$ is also complete with respect to $g_t$
		\item $S^1$ has only isolated fixed points by Lemma~\ref{l:free_nondegenerate}.
		\item By Theorem~\ref{t:nondegenerate_gk_description_v2} and our construction of $(M_\delta,g_t)$, metric $g_t$ is given by $g=W_th+W_t^{-1}\eta^2$ and the curvature form is given by $\beta=*_hdW_t+W_t\beta_0$.
			
		\item $M$ has finitely many fixed points. Since $M_\gd\to M\backslash \mathbf{T}_\gd\supset M$ is a covering map, the set of fixed points on $M_\gd$ is complete with respect to the pull back of any distance function on $M$, in particular with respect to the distance function induced by the metric $h+\eta^2$ on $M_\gd\backslash M^{S^1}$.
		\item $\Rm^{\R^3_{\pmb\mu},h}$ has a two-sided bound and $\beta_0$, $\nabla_h\beta_0$ are both bounded with respect to $h$ on $d_h$-bounded subsets. Indeed, as we observe below (see Remark~\ref{r:Norbifold1}), the metric completion of an isometric quotient $(\R^3_{\pmb\mu}/\Z, h)$ is a smooth orbifold, such that $h$ and $\beta_0$ extend to smooth tensors. Therefore all geometric quantities naturally attached to $\beta_0$ and $h$ have a bounded norm on $d_h$-bounded subsets. 
		\item $\Rc^{M,g_t}$ is bounded from below by our standing assumption on $M$ and by Lemma~\ref{l:ricci_bound}.
	\end{enumerate}
	Given $\delta'>0$, Proposition~\ref{p:s1_bundle_complete} implies that for any $\delta>\delta'$ a closed subset
	$M_{\delta}/S^1\subset M_{\delta'}/S^1$ is complete with respect to the metric $h$. Therefore the immersion $\iota$ of Theorem~\ref{t:nondegenerate_gk_description_v2}
	\[
	\iota\colon M_\delta/S^1\to \{x\in\R^3_{\pmb\mu}\ |\ |p(x)| \leq 1-\gd\}
	\]
	is a local diffeomorphism of $h$-complete manifolds. It follows that map $\iota$ satisfies the \textit{curve lifting property}: for any smooth curve $\til{\gg}\subset \{x\in\R^3_{\pmb\mu}\ |\ |p(x)| \leq 1-\gd\}$ starting at $\til\gg(0)=\iota(x_0)$ there exists its lift $\gg\subset M/S^1$, such that $\gg(0)=x_0$. A general result~\cite[Prop.\,6, \S5-6]{Carmo} implies that $\iota$ is a covering map.
	
	Now, since $\{x\in\R^3_{\pmb\mu}\ |\ |p(x)| \leq 1-\gd\}$ is simply connected, $\iota$ must be a diffeomorphism. Letting $\gd\to 0$, we conclude that there is a diffeomorphism
	\[
	(\til{M\backslash \mathbf{T}})/S^1\simeq \R^3_{\pmb\mu}
	\]
	so that $\pmb\mu \colon \til{M\backslash \mathbf{T}}\to \R^3_{\pmb\mu}$ is just the projection on the orbit space.
	\begin{lemma}\label{l:T_nonempty}
		The set $\mathbf{T}_+=\{p=1\}$ is nonempty as long as $a_+\neq 0$. Similarly $\mathbf{T}_-=\{p=-1\}$ is nonempty as long as $a_-\neq 0$.
	\end{lemma}
	\begin{proof}
		Assume that $a_+\neq 0$ yet $\mathbf{T}_+$ is empty. Then $\mathbf T_{+,\epsilon}:=\{y\in M\ |\ p(y)\geq 1-\epsilon\}$ is complete. By the same argument as above invoking Proposition~\ref{p:s1_bundle_complete}, we have an isomorphism of $h$-complete manifolds with boundary
		\[ 
		\mathbf{T}_{+,\epsilon}/S^1\simeq\{x\in\mathbf \R^3_{\pmb\mu}\ |\ p(x)\geq 1-\epsilon\}.
		\]
		Now consider a curve $\gamma\subset \R^3_{\pmb\mu}$ such that $\mu_1$ and $\mu_-$ are constant and $a_+\mu_+\to -\infty$. Along $\gamma$ we have $p\to 1$ and by a direct computation $\gamma$ has a finite $h$-length, giving a contradiction with the $h$-completeness of $\mathbf{T}_{+,\epsilon}/S^1$.
	\end{proof}

	The action of the deck transformation group $\Gamma$ on $\til{M\backslash \mathbf{T}}$ commutes with the flow of $X$ and therefore descends to the action on $\pmb\mu(\til{M\backslash \mathbf{T}})=\R^3_{\pmb\mu}$.  By construction, the action of $\Gamma$ on $\til{M\backslash \mathbf{T}}$ must preserve $(\Omega,I,J,X)$, therefore $\Gamma$ preserves the 1-forms $d\mu_i$. Hence, the descended action of $\Gamma$ on $\R^3_{\pmb\mu}$ must be via translations in $\mu_i$ coordinates.
	
	Now we prove that the action of $\Gamma$ on $\R^3_{\pmb\mu}$ is free. Assume on the contrary that the deck transformation, corresponding to a loop $\gg\subset M\backslash \mathbf{T}$ fixes a point $x_0\in \R^3_{\pmb\mu}$. Take any $y_0\in\pmb\mu^{-1}(x_0)$, and let $\til{\gg}$ be a path between $y_0$ and some $y_1$, which corresponds to the lift of $\gg$ to $\til{M\backslash \mathbf{T}}$. Since the element corresponding to $\gg$ fixes $x_0$, we have that $\pmb\mu(y_0)=\pmb\mu(y_1)=x_0$. The fibers of $\pmb\mu$ are connected, so we can ``close up'' $\til{\gg}$ by joining $y_0$ to $y_1$ within $\pmb\mu^{-1}(x)$. The loop $\hat{\gg}$ obtained this way satisfies:
	\[
	0=\int_{\hat{\gg}}d\mu_i=\int_{\til{\gg}}d\mu_i=\int_{\gg} \alpha_i,
	\]
	where the first identity is just Stokes' Theorem and the second identity uses the fact that $d\mu_i$ vanishes on $\pmb\mu^{-1}(x)$. We see that $\langle\alpha_i,[\gg]\rangle=0$, so $[\gg]\in K\subset H_1(M\backslash \mathbf{T};\Z)$ and $\gg$ corresponds to the trivial element of $\Gamma$.  Thus the action is indeed free. Next, since the action of $\Gamma$ on $\til{M\backslash \mathbf{T}}$ is properly discontinuous (this is true for any covering space), and the fibers of $\pmb\mu$ are compact, the action of $\Gamma$ on $\R^3_{\pmb\mu}$ also must be properly discontinuous.  With this in place we can take the fiberwise quotient of the moment map $\til{M\backslash \mathbf{T}}\to \R^3_{\pmb\mu}$ to obtain a map
	\[
	\pmb\mu\colon M\backslash \mathbf{T}\to  \R^3_{\pmb\mu}/\Gamma,
	\]
	which by abuse of notation we still denote $\pmb\mu$.
	
	Next, the action of $\Gamma$ on $\til{M\backslash\mathbf{T}}$ preserves GK structure, hence the action of $\Gamma$ descends to an \emph{isometric} action on $(\R^3_{\pmb\mu}, h)$. Given the precise form of the metric $h$, and using the fact that $p$ is a function of $a_+\mu_++a_-\mu_-$, we conclude that $\Gamma\simeq \Z^k$ has rank $k\leq 2$ and is a subgroup of the group $\R^2\subset \Isom(h)$ generated by the translations in $\mu_i$-coordinates, preserving the linear function $a_+\mu_++a_-\mu_-$.
	
	Our final goal is to determine which subgroups $\Gamma\subset\Isom(h)$ could occur for a given complete GK soliton $(M,g,I,J)$. By Lemma~\ref{l:T_nonempty} we know that $\mathbf{T}$ is nonempty. For an irreducible component $T_i\subset \mathbf{T}$ of the degeneracy divisor take a small neighbourhood $U_i=U(y_i)$ of a smooth point $y_i\in T_i$ and choose a loop $\gg_i\subset U_i$ with winding number $1$ with respect to $T_i$. By Proposition~\ref{p:Omegapoles}, $i_X\Omega_I$ is a meromorphic one-form with a pole along $T_i$, therefore $\int_{\gamma_i} i_X\Omega_I\neq 0$, and the corresponding deck transformation $g_i\in\Gamma$ is nontrivial.
	
	Since we can choose $\gg_i$ in its free-loop homotopy class to be arbitrarily short with respect to the metric $g$, element $g_i\in\Gamma$ acting on $\til{M\backslash \mathbf{T}}$ satisfies the following property:
	\[
	\inf_{y\in \til{M\backslash \mathbf{T}}} d_g(y,g_i y)=0.
	\]
	By Lemma~\ref{l:X_nonvanish} we can assume that $X$ does not vanish at $y_i$ so $W$ is bounded from above and away from zero in $U_i$. Then metric $Wh$ on $\pmb\mu(U_i)\subset\R^3_{\pmb\mu}$ is uniformly equivalent to $h$. In particular, for the induced action of $g_i\in\Gamma$ on $\R^3_{\pmb\mu}$ we also must have
	\[
	\inf_{x\in \R^3_{\pmb\mu}} d_h(x,g_i x)=0.
	\]
	The only such nontrivial translation preserving $a_+\mu_++a_-\mu_-$ is
	\[
	g_i(\mu_1,\mu_+,\mu_-)=(\mu_1+c,\mu_+,\mu_-),\quad  c\neq 0.
	\]
	We have proved that such a translation always belongs to $\Gamma$. Now, since $\Gamma\simeq \Z^k$ is a proper discrete subgroup of a two-dimensional translation group, either $\Gamma\simeq \Z$ generated by a translation in $\mu_1$ coordinate, or $\Gamma\simeq \Z^2$ generated by translations preserving the linear form $a_+\mu_++a_-\mu_-$. In the latter case $\Gamma$ can be generated by translations 
	\begin{equation}\label{f:deck_proof}
	\begin{split}
	g_1&(\mu_1,\mu_+,\mu_-)=(\mu_1+c_1,\mu_+,\mu_-),\quad c_1\neq 0\\
	g_2&(\mu_1,\mu_+,\mu_-)=(\mu_1+c_1',\mu_++ca_-,\mu_--ca_+),\quad c\neq 0. 
	\end{split}
	\end{equation}
	At the same time, by Proposition~\ref{p:invGKsoliton}, the closed one-form
	\[
	\alpha=-a_+d\mu_++a_-d\mu_-+p(a_+d\mu_++a_-d\mu_-)
	\]
	must be exact on $M$ and therefore on $\R^3_{\pmb\mu}/\Gamma$. Since $g_2$ preserves $a_+\mu_++a_-\mu_-$, for any $x\in \R^3_{\pmb\mu}$ we have
	\[
	\int_{x}^{g_2x}\alpha=\int_x^{g_2x}(-a_+d\mu_++a_-d\mu_-).
	\]
	Assuming $a_+\neq 0$, the latter integral vanishes if and only if $a_-=0$. In this case, redefining the constant $c$ in~\eqref{f:deck_proof} we arrive at the second possibility in the statement of the proposition.
	\end{proof}
	\end{prop}
	\begin{rmk}[Scaling convention]
		After scaling the metric $g$, we can assume that the action of $\gG$ in Proposition~\ref{p:deck_transform_mu} is such that there is a primitive subgroup $\Z\subset \Gamma$ acting on $\R^3_{\pmb\mu}$ via
		\[
		\mu_1\mapsto \mu_1+2\pi m,\quad m\in \Z.
		\]
		This will be our default scaling through the rest of the paper.
	\end{rmk}

	\subsection{The metric completion of \texorpdfstring{$(\R^3_{\pmb\mu}/\Gamma ,h)$}{}} \label{ss:metcomp}

	On the complement of the fixed point set $M\backslash\{y_i\}$ we have a well-defined horizontal metric $h$, which defines a Riemannian metric and a distance function $d_h$ on the orbit space $(M\backslash\{y_i\})/S^1$. The local asymptotic of $W$ near each fixed point $y_i$ implies that $d_h$ extends to a distance function on the whole $M/S^1$. In particular, there is a natural $d_h$-isometric map 
	\begin{equation}\label{f:proj_orbit_extension}
	M/S^1\to \bar{(\R^3_{\pmb\mu}/\Gamma,h)}
	\end{equation}
	to the metric completion of $(\R^3_{\pmb\mu}/\Gamma,h)$, which extends the isomorphism~\eqref{f:proj_nondegenerate} of Proposition~\ref{p:deck_transform_mu}:
	\[
	(M\backslash\mathbf{T})/S^1 \to \R^3_{\pmb\mu}/\Gamma.
	\]
	The goal of this subsection is to prove that $\bar{(\R^3_{\pmb\mu}/\Gamma,h)}$ has a natural orbifold structure, and then in the next subsection we will show the map~\eqref{f:proj_orbit_extension} is a diffeomorphism of orbifolds.
	
	We just have proved that the orbit space of the nondegenerate part $M\backslash\mathbf{T}$ of $(M,g,I,J)$ is isomorphic to one of the
	\[
	\R^3_{\pmb\mu}/\Gamma_1,\quad \R^3_{\pmb\mu}/\Gamma_2,
	\]
	where $\Gamma_1\simeq \Z$, $\Gamma_2\simeq \Z^2$ are groups described in Proposition~\ref{p:deck_transform_mu}. Our next goal is to describe the metric completions of $(\R^3_{\pmb\mu}/\Gamma_i, h)$. Since $\R^3_{\pmb\mu}/\Gamma_2$ is an isometric quotient of  $\R^3_{\pmb\mu}/\Gamma_1$ it suffices to determine the completion of $(\R^3_{\pmb\mu}/\Gamma_1,h)$.
	
	The metric $h$ on $\R^3_{\pmb\mu}/\Gamma_1$ is characterized by two constants $a_+,a_-\in\R$ which do not vanish simultaneously (if $a_+=a_-=0$ then the underlying GK soliton is hyperK\"ahler). In what follows, we assume that $a_+\neq 0$. It turns out that the completions $(\R^3_{\pmb\mu}/\Gamma_1,h)$ are essentially different in the cases $a_-=0$ and $a_-\neq 0$ which we describe separately. To highlight the dependence of $h$ on the real parameters $a_+$ and $a_-$, we will write $h=h_{a_+,a_-}$.
	
	\subsubsection{Case $a_-=0$}
	
	Introduce new coordinates on $\R^3_{\pmb\mu}/\Gamma_1$:
	\[
	(\mu_1,\rho,\mu_-)\in S^1\times \R_{>0}\times \R,
	\]
	where
	\begin{equation*}
	\begin{split}
	\rho=\exp(a_+\mu_+/2).
	\end{split}
	\end{equation*}
	In these coordinates, $h_{a_+,0}=(1-p^2)d\mu_1^2+2(1-p)d\mu_+^2+2(1+p)d\mu_-^2$ takes form
	\begin{equation}\label{f:h_a0}
	h_{a_+,0}=\frac{4}{1+\rho^2}
	\left(
	\frac{\rho^2}{1+\rho^2}d\mu_1^2+\frac{4}{a_+^2}d\rho^2+4d\mu_-^2.
	\right)
	\end{equation}
	\begin{prop}\label{p:completion_a0}
		The metric completion of $(\R^3_{\pmb\mu}/\Gamma_1, h_{a_+,0})$ is isomorphic to
		\[
		S^1\times \R_{>0}\times \R\simeq \R^3
		\]
		with the metric~\eqref{f:h_a0} extending to a metric with the cone singularity of angle $\pi |a_+|$ along the codimension two subset $\{\rho=0\}$.
	\end{prop}
	\begin{proof}
		We will construct the metric completion of $(\R^3_{\pmb\mu}/\Gamma_1, h_{a_+,0})$ explicitly by hand.  Pick a positive number $k\in\R$ and consider $\R^2$ and $\R^2\backslash\{0\}\simeq \R\times S^1$ equipped respectively with a cone metric and a regular cylinder metric given in the polar coordinates by
		\[
		\begin{split}
		g_1&=k^2\, d\rho^2+\rho^2\, d\phi_1^2,\\
		g_2&=4d\mu_-^2+d\phi_2^2.
		\end{split}
		\]
		Consider $M=\R^2\times \R\times S^1$ with a metric $g=\cfrac{g_1\oplus g_2}{1+\rho^2}$. The metric space $(M,g)$ is evidently complete, and admits a free isometric $S^1$ action induced by the vector field $\del_{\phi_1}+\del_{\phi_2}$.  The orbit space $M/S^1$ is naturally isomorphic to $\R^2\times \R$ and the horizontal component of metric $g$ descends to a complete metric $h_0$ with a cone singularity along $\{0\}\times \R$. 
		
		The metric $h_0$ is regular on the open part $(\R^2\backslash\{0\})\times \R$, and to express it, we choose coordinates $(\phi_1-\phi_2,\rho,d\mu_-)$ on this regular locus. Computing $|d(\phi_1-\phi_2)|^2_g$, $|d\rho|^2_g$ and $|d\mu_-|^2_g$, we find that:
		\begin{equation}\label{f:h_a0_pf}
		h_0=\frac{1}{1+\rho^2}\left(
		\frac{\rho^2}{1+\rho^2}{d(\phi_1-\phi_2)^2+k_1^2d\rho^2+4d\mu_-^2}
		\right).
		\end{equation}
		
		If we set $k=2/|a_+|$ and denote $\mu_1=\phi_1-\phi_2$, then the metric in~\eqref{f:h_a0} will coincide with ~\eqref{f:h_a0_pf} up to a constant factor $4$. Therefore, $(M/S^1, 4h_0)$ is the metric completion of $(\R^3_{\pmb\mu}/\Gamma_1, h_{a_+,0})$ as claimed.
	\end{proof}

	\subsubsection{Case $a_-\neq 0$}
	The argument in this case is very similar to the case $a_-=0$. Introduce new coordinates on $\R^3_{\pmb\mu}/\Gamma_1$:
	\[
	(\mu_1,\rho_1,\rho_2)\in S^1\times \R_{>0}\times \R_{>0}
	\]
	where
	\begin{equation*}
	\begin{split}
	\rho_1=\frac{e^{-a_-\mu_-/2}}{a_-^2e^{a_+\mu_+}+a_+^2e^{-a_-\mu_-}},\\
	\rho_2=\frac{e^{a_+\mu_+/2}}{a_-^2e^{a_+\mu_+}+a_+^2e^{-a_-\mu_-}}.
	\end{split}
	\end{equation*}
	In these coordinates, the metric $h_{a_+, a_-}=(1-p^2)d\mu_1^2+2(1-p)d\mu_+^2+2(1+p)d\mu_-^2$ takes the form
	\begin{equation}\label{f:h_a_neq0}
	h_{a_+, a_-}=\frac{4}{\rho_1^2+\rho_2^2}
	\left(
	\frac{\rho_1^2\rho_2^2}{\rho_1^2+\rho_2^2}d\mu_1^2+\frac{4}{a_-^2}d\rho_1^2+\frac{4}{a_+^2}d\rho_2^2
	\right).
	\end{equation}
	\begin{prop}\label{p:completion_aneq0}
		The metric completion of $(\R^3_{\pmb\mu}/\Gamma_1, h_{a_+, a_-})$ is isomorphic to
		\[
		(S^1\times \R_{\geq 0}\times \R_{\geq 0})\backslash (S^1\times \{0\}\times \{0\})\simeq S^2\times \R
		\]
		with the metric~\eqref{f:h_a_neq0} extending to a metric with cone singularities of angles $\pi |a_-|$ and $\pi |a_+|$ along the  codimension two subsets $\{\rho_1=0\}$ and $\{\rho_2=0\}$ respectively.
	\end{prop}
	\begin{rmk}
		The completion in the above proposition is isomorphic to a direct product $S^2\times \R$, where $S^2$ has a `spindle' metric with two cone singularities of angles $\pi a_-$ and $\pi a_+$.
	\end{rmk}
	\begin{proof}
		We will construct the metric completion of $(\R^3_{\pmb\mu}/\Gamma_1, h)$ explicitly by hand.  Pick two positive numbers $k_1,k_2\in\R$ and consider two copies of $\R^2$ equipped with cone metrics given in polar coordinates by
		\[
		\begin{split}
		g_1&=k_1^2\, d\rho_1^2+\rho_1^2\, d\phi_1^2,\\
		g_2&=k_2^2\, d\rho_2^2+\rho_2^2\, d\phi_2^2.
		\end{split}
		\]
		Each copy $(\R^2,g_i)$ has a cone angle $2\pi/k_i$ at the origin.	Consider $M=(\R^2\times \R^2)\backslash\{(0,0)\}\simeq S^3\times \R$ with a metric $g=\cfrac{g_1\oplus g_2}{\rho_1^2+ \rho_2^2}$. The metric space $(M,g)$ is evidently complete, and admits a free isometric $S^1$ action induced by the vector field $\del_{\phi_1}+\del_{\phi_2}$. The orbit $M/S^1$ space is naturally isomorphic to $S^2\times \R$ and the horizontal component of metric $g$ descends to a complete metric $h_0$ with cone singularities along $((\R^2\backslash\{0\})\times \{0\})/S^1$ and $(\{0\}\times (\R^2\backslash\{0\}))/S^1$. 
		
		The metric $h_0$ is regular on the open part $((\R^2\backslash\{0\})\times (\R^2\backslash\{0\}))/S^1$, and to express it, we choose coordinates $(\phi_1-\phi_2,\rho_1,\rho_2)$ on this regular locus. Computing $|d(\phi_1-\phi_2)|^2_g$, $|d\rho_1|^2_g$ and $|d\rho_2|^2_g$, we find that:
		\begin{equation}\label{f:h_a_neq0_pf}
		h_0=\frac{1}{\rho_1^2+\rho_2^2}\left(
		\frac{\rho_1^2\rho_2^2}{\rho_1^2+\rho_2^2}{d(\phi_1-\phi_2)^2+k_1^2d\rho_1^2+k_2^2\rho_2^2}
		\right).
		\end{equation}
		
		If we set $k_1=2/|a_-|$, $k_2=2/|a_+|$ and denote $\mu_1=\phi_1-\phi_2$, then the metric in~\eqref{f:h_a_neq0} will coincide with~\eqref{f:h_a_neq0_pf} up to a constant factor $4$. Therefore, $(M/S^1, 4h_0)$ is the metric completion of $(\R^3_{\pmb\mu}/\Gamma_1, h_{a_+,a_-})$ as claimed.
	\end{proof}
	
	We have just described the metric completions of $(\R^3_{\pmb\mu}/\Gamma_1,h)$ in the cases $a_-=0$ and $a_-\neq 0$, where group $\Gamma_1\simeq \Z$ acts on $\R^3_{\pmb\mu}$ via translations in $\mu_1$ direction. By Proposition~\ref{p:deck_transform_mu}, in the case $a_-=0$, we also have to deal with the completion $(\R^3_{\pmb\mu}/\Gamma_2,h)$, where $\Gamma=\Gamma_2\simeq \Gamma_1\times \Z$, with $\Z$ acting on $\R^3_{\pmb\mu}$ via
	\begin{equation}\label{f:Z_action}
	(\mu_1,\mu_+,\mu_-)\mapsto (\mu_1+c_1',\mu_+,\mu_-+c),\quad c\neq 0
	\end{equation}
	translations preserving the linear form $a_+\mu_++a_-\mu_-$. In the latter case, $\R^3_{\pmb\mu}/\Gamma_2$ is an isometric $\Z$-quotient of $\R^3_{\pmb\mu}/\Gamma_1$, thus the same is true for the completions:
	\[
	\bar{(\R^3_{\pmb\mu}/\Gamma_2,h)}\simeq \bar{(\R^3_{\pmb\mu}/\Gamma_1,h)}/\Z,
	\]
	where the action of the factor $\Z$ of $\Gamma_2$ is given by Proposition~\ref{p:deck_transform_mu}. In either case, the completion $\bar{(\R^3_{\pmb\mu}/\Gamma,h)}$ has a natural smooth structure, and the metric $h$ extends to a metric with cone singularities along codimension two submanifolds.

\subsection{The global definition of the moment map on \texorpdfstring{$M$}{}} \label{ss:globalmoment}

	\begin{defn}\label{d:N(a+,a-)}
		Define $N(a_+,a_-)$ to be the metric completion $\bar{(\R^3_{\pmb\mu}/\Gamma_1,h_{a_+,a_-})}$. The metric $h_{a_+,a_-}$ induces a smooth metric with cone singularities on $N(a_+,a_-)$.
	\end{defn}
	
		Using the explicit description of $N(a_+,a_-)$ above, we can rewrite $p$ as
		\[
		p=\frac{1-\rho^2}{1+\rho^2},\quad p=\frac{\rho_1^2-\rho_2^2}{\rho_1^2+\rho_2^2}
		\]
		in the cases $a_-=0$ and $a_-\neq 0$ respectively. It follows that function $p$ defined initially on the open dense part has a continuous extension to the whole space, and the completion locus $N(a_+,a_-)\backslash (\R^3_{\pmb\mu}/\Gamma_1)$ consists of disjoint subsets $\{p=1\}$ and $\{p=-1\}$ (the latter possibly empty if $a_-=0$). The metric $h_{a_+,a_-}$ extends to a metric with cone singularities along $\{p=1\}$ and $\{p=-1\}$.  Note that since $(\R^3_{\pmb\mu}/\Gamma_2,h_{a_+,a_-})$ is an isometric $\Z$-quotient of $(\R^3_{\pmb\mu}/\Gamma_1,h_{a_+,a_-})$, the same is true for the completion.  Furthermore, with the explicit description of the completion, it is easy to see that the induced action of $\Z$ on $N(a_+,a_-)$ is free proper and isometric.
	
	Let $(M,g,I,J)$ be a regular compete rank one soliton. By our assumption $a_+\neq 0$ and by Lemma~\ref{l:T_nonempty} the degeneracy set $\mathbf{T_+}$ is nonempty. Consider any point $x\in \mathbf{T_+}$ such that $X\neq 0$ at $x$. Let $G_x\simeq \Z_{k_+}$ be the stabilizer of $x$. Then the orbit space $\pi\colon M\to N$ has an orbifold structure in a neighourhood $U$ of $[x]\in N$ and the horizontal metric $Wh$ as well as $h$ have a cone singularity of angle $2\pi/{k_+}$ along the codimension two subset $(\pi^{-1}(U)\cap \mathbf{T}_+)/S^1$. On the other hand, there is an isometric embedding $N\to N(a_+,a_-)$ (or $N\to N(a_+,a_-)/\Z$) (see~\eqref{f:proj_orbit_extension}), and by the above propositions $N(a_+,a_-)$ has cone singularity of angle $\pi |a_+|$ along $\{p=1\}$. Therefore $2/k_+=|a_+|$. Arguing similarly in the case $a_-\neq 0$, we obtain the following \emph{quantization} result for $a_+,a_-\in \R$.
	\begin{prop}[Quantization of parameters $a_+$ and $a_-$]\label{p:quantization}
		If $a_+\neq 0$, then $\mathbf{T_+}$ is nonempty, and 
		\[
		a_+=2/k_+,\ k_+\in \Z,
		\] where $|k_+|$ is the order of the stabilizer of any point $x\in\mathbf{T_+}\backslash M^{S^1}$. 
		If $a_-\neq 0$, then $\mathbf{T_-}$ is nonempty, and \[
		a_-=2/k_-,\ k_-\in \Z\]
		where $|k_-|$ is the order of the stabilizer of any point $x\in\mathbf{T_-}\backslash M^{S^1}$.		
	\end{prop}
\begin{rmk} \label{r:Norbifold1}
		$N(a_+,a_-)$ has a natural manifold structure and a metric with cone angles of value $2\pi/k, k\in \Z$ along its codimension two submanifolds. Therefore we can think of $N(a_+,a_-)$ as an orbifold, and of $h$ as an orbifold metric. It is also straightforward to check that $\beta_0=d\mu_1\wedge(p_+ d\mu_-+p_-d\mu_+)$ defined on $\R^3_{\pmb\mu}/\Gamma$ extends to a smooth 2-form on the orbifold $N(a_+,a_-)$.
\end{rmk}

	\begin{rmk}[$N(a_+,a_-)$ as a base space of a Seifert fibration]\label{r:N_as_quotient}
		So far, in Propositions~\ref{p:completion_a0} and~\ref{p:completion_aneq0} we have obtained $(N(a_+,a_-), h)$ as a global quotient of a smooth manifold equipped with a cone metric. With the quantization of parameters $a_+$ and $a_-$ established above, there is an alternative description of $(N(a_+,a_-), h)$ as a global quotient of a smooth manifold equipped with a smooth metric:
		
		\begin{enumerate}
			\item Case $a_+=2/k_+, a_-=0$. Consider $\hat M=\C\times \C^*$ with coordinates $z,w\in \C$ and define an almost free circle action $S^1\times \hat M\to \hat M$ as $u\cdot (z,w)=(uz,u^{k_+}w)$. Then there is an isomorphism between the smooth parts of $\hat M/S^1$ and $N(2/k_+,0)$ given by the coordinate change
			\begin{equation*}
			\begin{split}
			k_+&\arg z-\arg w=\mu_1,\\
			|z|&=\exp(\mu_+/k_+),\\
			|w|&=\exp (2\mu_-).
			\end{split}
			\end{equation*}
			As in the proof of Proposition~\ref{p:completion_a0}, this isomorphism transforms the horizontal component $\hat g_{\mathrm{hor}}$ of the metric
			\[
			\hat g=\frac{4}{1+|z|^2}\Re\left(
			k_+^{2} dz\otimes d\bar z + \frac{dw\otimes d\bar w}{|w|^2}
			\right)
			\]
			to $h$ and extends to an isometry of complete Riemannian orbifolds
			\begin{equation*}
			(\hat M/S^1,\hat g_{\mathrm{hor}})\simeq (N(2/k_+,0), h).
			\end{equation*}
			
			\item Case $a_+=2/k_+, a_-=2/k_-$. Assume first that $\mathrm{gcd}(k_+,k_-)=1$. Consider $\hat M=\C^2\backslash\{0\}$ with coordinates $z,w$ and define an almost free circle action $S^1\times \hat M\to \hat M$ as $u\cdot (z,w)=(u^{k_+}z,u^{k_-}w)$. Then there is an isomorphism between the smooth parts of $\hat M/S^1$ and $N(2/k_+,2/k_-)$ given by the coordinate change
			\begin{equation*}
			\begin{split}
			k_-&\arg z-k_+\arg w=\mu_1,\\
			|z|&=\frac{e^{-\mu_-/k_-}}{4(k_-^{-2}e^{2\mu_+/k_+}+k_+^{-2}e^{-2\mu_-/k_-})},\\
			|w|&=\frac{e^{\mu_+/k_+}}{4(k_-^{-2}e^{2\mu_+/k_+}+k_+^{-2}e^{-2\mu_-/k_-})}.
			\end{split}
			\end{equation*}
			As in the proof of Proposition~\ref{p:completion_aneq0}, this isomorphism transforms the horizontal component $\hat g_{\mathrm{hor}}$ of the metric
			\[
			\hat g=\frac{4}{|z|^2+|w|^2}\Re\left(k_-^2dz\otimes d\bar z+k_+^2dw\otimes d\bar w\right)
			\]
			into $h$, and extends to an isomorphism of the underlying orbifolds:
			\begin{equation*}
			(\hat M/S^1,\hat g_{\mathrm{hor}})\simeq (N(2/k_+,2/k_-), h).
			\end{equation*}
			
			Finally, if $\mathrm{gcd}(k_-,k_+)=d>1$, then orbifold $(N(2/k_+,2/k_-), h)$ is a $\Z_d$ quotient of $(N(2d/k_+,2d/k_-), h)$, where $\Z_d$ acts on $N(2d/k_+,2d/k_-)$ via rotations in $\mu_1$ coordinate.
		\end{enumerate}
	\end{rmk}
	
	Finally we are ready to show the surjectivity of the moment map.  Assume for concreteness that in Proposition~\ref{p:deck_transform_mu} we have group $\Gamma=\Gamma_1$. We aim to show that the natural isometric embedding of the orbit space $N\to N(a_+,a_-)$ is an isometry. We know that the map $(M\backslash\mathbf T)/S^1\to \R^3_{\pmb\mu}/\Gamma$ is a diffeomorphism, and the metric $h$ can be intrinsically defined both on $N_0=(M\backslash M^{S^1})/S^1$ and on $N(a_+,a_-)$. Since $N\to N(a_+,a_-)$ is an isometric embedding, smooth on an open dense part, it must be smooth on the whole $N_0$. This can be seen, e.g., by considering exponential charts centered outside of the locus $\{p=\pm 1\}$.
	
	To prove surjectivity of $N\to N(a_+,a_-)$, we are going to invoke Proposition~\ref{p:s1_bundle_complete} once again. Consider $\mathbf{T}_{+,\ge}:=\{p\geq 1-\ge\}\subset M$. For a generic constant $\ge>0$, this is a connected complete manifold with boundary.  As in Remark \ref{r:Norbifold1}, $\beta_0$ and $h$ are well-defined smooth tensors on $N_0$. Using the same argument as before in the proof of Proposition~\ref{p:deck_transform_mu}, we add a multiple of the baseline solution $\til W$ to function $W$, and modify the metric $g$ and connection $\eta$ accordingly. Repeating the argument in the proof of Proposition~\ref{p:deck_transform_mu}, we conclude that $N\cap \{p\geq 1-\ge\}$ is complete with respect to the distance function $d_h$ induced by the metric $h$. Similarly the pieces $N\cap \{p\leq -1+\ge\}$ and $N\cap \{|p|\leq 1-\ge\}$ are $d_h$-complete. Therefore $N$ is also complete with respect to $d_h$, hence $N\to N(a_+,a_-)$ is an isometry.

	The orbifold $N(a_+,a_-)$ comes equipped with a distinguished metric $h=h_{a_+,a_-}$, function $p$ and 2-form $\beta_0$ on an open dense part. Using the coordinates of Proposition~\ref{p:completion_a0} and~\ref{p:completion_aneq0}, it is straightforward to check that $h$, $p$, $\beta_0$ (see equation~\eqref{f:beta0_def}), and $\theta_I$ (see equation~\eqref{f:theta_through_p}) extend to smooth (in the orbifold sense) tensors on $N(a_+,a_-)$.
	
	Collecting Theorem~\ref{t:nondegenerate_gk_description}, Proposition~\ref{p:deck_transform_mu} and Proposition~\ref{p:quantization}, we obtain the following.
	\begin{thm}\label{t:M_orbifold_quotient}
		Let $(M,g,I,J)$ be a complete regular rank one soliton. Then there exist constants $a_+,a_-\in \{2/k\ |\ k\in\Z \}\cup \{0\}$, $(a_+,a_-)\neq (0,0)$ such that either
		\[
		M/S^1\simeq N(a_+,a_-),
		\]
		or
		\[
		M/S^1\simeq N(a_+,0)/\Z,
		\]
		where $\Z$ acts on $N(a_+,0)$ by translations in $(\mu_1, \mu_-)$ coordinates preserving $a_+\mu_++a_-\mu_-$.
		
		Furthermore on the complement of the fixed point set $M^{S^1}$ there exists an $S^1$-invariant function $W$ such that the Seifert fibration
		\[
		M\backslash M^{S^1}\to(M\backslash M^{S^1})/S^1
		\]
		has curvature $\beta=*_hdW+W\beta_0$ with connection $\eta$ and the underlying GK structure is given by
		\begin{equation*}
		\begin{split}
		g=&\ Wh+W^{-1}\eta^2\\
		\Omega_{I}=&
		(-d\mu_1+\sqrt{-1}d\mu_2)\wedge(\eta+\sqrt{-1}W(d\mu_3-pd\mu_2)),\\
		\Omega_J=&
		(-d\mu_1+\sqrt{-1}d\mu_3)\wedge(\eta+\sqrt{-1}W(-d\mu_2+pd\mu_3)).
		\end{split}
		\end{equation*}
	\end{thm}
	\begin{rmk}
		Assume as before $a_+\neq 0$.  The topology of the orbit space $N(a_+,a_-)/H$ depends on $a_+,a_-$ and $H$ as follows:
		\begin{enumerate}
			\item $H=\{\mathrm{id}\}$, $a_+=2/k_+$, $a_-=0$. In this case orbifold $N(a_+,a_-)$ is diffeomorphic to a global quotient of $\C\times \R$ by a linear $\Z_{k_+}$-action in the first coordinate. The underlying smooth manifold is diffeomorphic to $\R^3$;
			\item $H=\Z$, $a_+=2/k_+$, $a_-=0$. In this case $N(a_+,a_-)$ is diffeomorphic to a quotient $(\C/\Z_{k_+}\times \R)/\Z$, where $\Z$ acts via rotation on the first  factor and via nontrivial translation on the second. The underlying smooth manifold is diffeomorphic to $\R^2\times S^1$;
			\item $H=\{\mathrm{id}\}$, $a_+=2/k_+$, $a_-=2/k_-$. In this case orbifold $N(a_+,a_-)$ is diffeomorphic to a product $S^2(k_+,k_-)\times \R$, where $S^2(k_+,k_-)$ is a \emph{spindle} 2-sphere with cone angles $2\pi/k_+$ and $2\pi/k_-$.
			The underlying smooth manifold is diffeomorphic to $S^2\times \R$;
		\end{enumerate}
	\end{rmk}
	\begin{cor}\label{c:fixed_smooth}
		Let $(M,g,I,J)$ be a complete regular rank one soliton, and denote by $\pi\colon M\to N(a_+,a_-)/H$ the projection onto the orbit space. If $y\in M$ is a fixed point of the $S^1$ action, then the representation of $S^1$ in $T_yM$ has weights $(\pm 1,\pm 1)$ and $\pi(y)$ is a smooth point of orbifold $N(a_+,a_-)/H$.
	\end{cor}
	\begin{proof}
		We already know by Lemma~\ref{l:free_nondegenerate} that any fixed point $y$ in the nondegenerate part $M\backslash\mathbf T$ must have weights $(\pm 1,\pm 1)$. Now assume that $y\in\mathbf{T}\subset M$ is a fixed point with weights $\mathbf{w}=(w_1,w_2)$, $\mathrm{gcd}(w_1,w_2)=1$. Let $y\in U\simeq \C^2$ be a neighbourhood of $y$ provided by the slice theorem. Then $(U\backslash\{y\})/S^1$ is an orbifold isomorphic to a product $\R\times S^2(|w_1|,|w_2|)$ of $\R$ and a spindle $S^2$ with two orbifold points of orders $|w_1|$ and $|w_2|$.
		
		On the other hand, given a point $z\in N(a_+,a_-)$ there are two options. Either $z$ is a smooth point, and its punctured neighbourhood is smooth. Or a punctured neighbourhood of $z$ is isomorphic as an orbifold to a product $\R\times S^2(k,k)$, $k\in\{k_+,k_-\}$. Hence at $y\in M$ we must have $|w_1|=|w_2|$ which is possible only if $w_i=\pm 1$.
	\end{proof}
	The above corollary implies that if there is a fixed point on the degeneracy divisor $\mathbf T_+$ (resp.\,$\mathbf T_-$), then $k_+=\pm 1$ (resp.\,$k_-=\pm 1$).

\section{Construction and classification of solutions} \label{s:construction}

In this section we give a construction of complete rank one generalized K\"ahler-Ricci solitons, proving Theorem \ref{t:existence}.  Having already determined the function $p$ using Proposition \ref{p:invGKsoliton}, our task is to find suitable functions $W$ solving (\ref{f:W_laplace2}), which extend naturally over the degeneracy loci.  Note that in Lemma \ref{l:W0_solution} we already identified the baseline solution $\til{W}$ to equation (\ref{f:W_laplace2}).  We first observe that for any other solution $W$, the function $V = W/\til{W}$ satisfies a Laplace equation for a certain metric conformal to to the induced metric on the quotient space.    Using this fact and analyzing the explicit constructions of the spaces $N(a_+, a_-)$ in \S \ref{s:completion}, we classify the possible solutions to equation (\ref{f:W_laplace2}).  With the explicit construction of the horizontal geometry and the classification of possible functions $W$ in hand, we give the proof of the main existence theorem (Theorem \ref{t:existence}), and the classification theorem (Theorem \ref{t:uniqueness}).

\subsection{Reduction of the equation for \texorpdfstring{$W$}{} to a Laplace equation} \label{ss:Wlaplace}

The first step in our analysis is to better understand the structure of the equation for $W$ in the case of solitons. Recall that by Lemma~\ref{l:W0_solution} the function
\[
\til W:=\left(a_+^2(1+p)+a_-^2(1-p)\right)^{-1}
\]
solves equation~\eqref{f:W_laplace2} provided $p$ corresponds to a rank one GK soliton as in Proposition~\ref{p:invGKsoliton}. With the use of this baseline solution we can reduce the equation for $W$ to the Laplace equation for a metric $\til h$ conformally related to $h$.

\begin{prop} \label{p:laplace_reduction} Equation~\eqref{f:W_laplace2} for $W$ is equivalent to the Laplace equation
	\[
	\Delta_{\til h}V=0
	\]
	for the function $V:=W/\til W$, where
	\[
	\til h:= \left(\frac{2\til W^2}{e^{a_+\mu_+}+e^{-a_-\mu_-}} \right)^2 h.
	\]
\end{prop}
\begin{proof}
	Using that $\til W$ is a solution, we can rewrite equation~\eqref{f:W_laplace2} as an equation for $V = W/\til W$:
	\begin{equation}\label{f:V_laplace2}
	\begin{split}
	\til W\bigl(V_{11}+\tfrac{1}{2}(1+p)V_{++}+\tfrac{1}{2}(1-p)V_{--}\bigr)
	+(\til W(1+p))_+V_++(\til W(1-p))_-V_-=0.
	\end{split}
	\end{equation}
	Now let 
	\[
	\psi = \frac{2\til W^2}{e^{a_+\mu_+}+e^{-a_-\mu_-}} = \til W^2(1-p)e^{-a_+\mu_+} = \til W^2(1+p)e^{a_-\mu_-},
	\]
	so that the conformal metric is $\til h=\psi^2h$.  Then we find that
	\begin{equation*}
	\begin{split}
	\Delta_{\til h}V&=\frac{1}{2\psi^3(1-p^2)}
	\left(
	2\psi V_{11}+(\psi(1+p)V_{+})_{+}+(\psi(1-p)V_{-})_{-}
	\right)\\
	&=\frac{1}{\psi^2(1-p^2)}\left(V_{11}+\tfrac{1}{2}(1+p)V_{++}+\tfrac{1}{2}(1-p)V_{--}+
	\frac{(\psi(1+p))_{+}}{2\psi}V_{+}+
	\frac{(\psi(1-p))_{-}}{2\psi}V_{-}\right).
	\end{split}
	\end{equation*}
	Using the different expressions for $\psi$ above it is straightforward to check that
	\[
	\frac{(\psi(1+p))_{+}}{\psi}=2\frac{(\til W(1+p))_+}{\til W}, \qquad \frac{(\psi(1-p))_{-}}{\psi}=2\frac{(\til W(1-p))_-}{\til W},
	\]
	which implies that \eqref{f:V_laplace2} is equivalent to $\Delta_{\til h}V=0$.
\end{proof}

The function $\psi:=2\til W^2/(e^{a_+\mu_+}+e^{a_-\mu_-})$ is defined initially on $\R^3_{\pmb\mu}$, and extends to a smooth function on the orbifold $N(a_+,a_-)$. Thus $\til h$ can be thought of as an orbifold metric on $N(a_+,a_-)$. If $\{a_+,a_-,W\}$ corresponds to a complete rank one regular GK soliton $(M,g,I,J)$, then $V=W/\til W$ is a positive smooth solution to the Laplace equation $\Delta_{\til h} V=0$ on $N(a_+,a_-)\backslash\{z_i\}$, where $\{z_i\}$ is the set of fixed point orbits. Furthermore, since the vector field $X$ generating the $S^1$ action has a simple zero at each $z_i$, we have $V(x)\sim d_h(x,z_i)^{-1}$ in a neighbourhood of $z_i$.

Thus, to proceed with the construction and classification of GK solitons, we need to construct and classify positive solutions to the Laplace equation $\Delta_{\til h}V=0$ on $N(a_+,a_-)$ with poles on a finite set $\{z_i\}$. There are two difficulties with answering this question. The first one is technical: $N(a_+,a_-)$ is an orbifold rather than a manifold, therefore we have to be somewhat cautious in developing the elliptic theory on $N(a_+,a_-)$. The second one is more substantial~--- $N(a_+,a_-)$ equipped with the metric $\til h$ is not necessarily complete, its completion is not an orbifold, and we do not have any a priori knowledge about the behavior of $V$ near the completion. Thus a large part of the general well-developed elliptic theory is not available to us.

Luckily, thanks to a very special structure of orbifold $N(a_+,a_-)$ and metric $\til h$, we are able to surpass these difficulties by `lifting' the Laplace equation $\Delta_{\til h}V=0$ on an incomplete 3-dimensional orbifold to an equivalent Laplace equation on a complete flat 4-dimensional manifold. This idea works equally well if either $a_-=0$ or $a_-\neq 0$, however, since the details differ slightly, we consider the two cases separately.

\subsection{Equation \texorpdfstring{$\Delta_{\tilde h}V=0$}{} in the case \texorpdfstring{$a_- = 0$}{}}

\begin{prop}\label{p:elliptic_a0}
	Let $N(a_+,0)$ be the orbifold from Definition~\ref{d:N(a+,a-)}. 
	\begin{enumerate}
		\item Given any point $z\in N(a_+,0)$ there exists a unique smooth positive function 
		\[G_z(x)\colon N(a_+,0)\backslash\{z\}\to (0;\infty)\] such that:
		\begin{enumerate}
			\item $\Delta_{\til h}G_z(x)=-2\pi\delta_z(x)$ in the sense of distributions;
			\item $\inf_{x\in N(a_+,0)} G_z(x)=0$;
		\end{enumerate}
		\item If $V$ is a positive solution to the equation $\Delta_{\til h}V=0$ on the complement of a finite set $\{z_i\}$, then there exist unique positive constants $c_V,c_i$ such that
		\[
		V(x)=c_V+\sum_{i} c_i G_{z_i}(x).
		\]
		\item Any positive superharmonic function on $N(a_+,0)/\Z$ is constant.
	\end{enumerate}
\end{prop}
\begin{rmk}
	Borrowing terminology from the standard elliptic theory on smooth complete manifolds, we can say that $(N(a_+,0),\til h)$ is \emph{non-parabolic}, i.e., admits a positive Green's function and satisfies a \emph{Liouville property}, i.e., any positive harmonic function is constant; while $(N(a_+,0)/\Z,\til h)$ is \emph{parabolic}~--- any positive superharmonic function on it is constant.
\end{rmk}
\begin{proof}
	We start by recalling an observation from the proof of Proposition~\ref{p:s1_bundle_complete}. If $M\to N$ is a principal $S^1$ bundle with a connection one-form $\eta$, and the metric $g$ on $M$ is of the form
	\[
	g=Wh+W^{-1}\eta^2,
	\]
	where $W\colon M\to N$ is a smooth function and $h$ is a metric on $N$, then $f\colon N\to \R$ solves the Laplace equation $\Delta_h f=0$ on $M$ if and only if its lift to $M$ solves equation $\Delta_g f=0$, and the same is true for inequality. Hence there is a bijection:
	\begin{equation}\label{f:laplace_equivalence}
	\{\mbox{harmonic functions on }(N,h)\}\longleftrightarrow \{S^1\mbox{ invariant harmonic functions on }(M,g)\}.
	\end{equation}
	
	We use the constructions of Remark~\ref{r:N_as_quotient} and realize $N(a_+,0)$ as a quotient of $\hat M=\C\times \C^*$ by a circle action $u\cdot (z,w)=(uz,u^{k_+}w)$. Under this identification the conformal factor $\psi^2$ of Proposition~\ref{p:laplace_reduction} becomes
	\[
	\psi^2=\frac{(1+|z|^2)^2}{a_+^8}.
	\]
	Let $X\in \Gamma(T\hat M)$ be the vector field generating the $S^1$ action, and denote by $\eta$ the connection induced by the metric $\hat g$. Then $|X|^2_{\hat g}=4k_+^2=16a_+^{-2}$, and on $\hat M$ we can express the metric $\hat g$ as
	\[
	\hat g=h+\frac{16}{a_+^2}\eta^2.
	\]
	Thus the flat metric on $\hat M$,
	\[
	g_f:=\frac{1+|z|^2}{4}\hat g=\Re(k_+^2dz\otimes d\bar z+\frac{dw\otimes d\bar w}{|w|^2}),
	\]
	can be expressed as
	\[
	g_f=\frac{1+|z|^2}{4}h+\frac{4(1+|z|^2)}{a_+^2}\eta^2=\frac{a_+^2}{4(1+|z|^2)} \underbrace{\frac{(1+|z|^2)^2}{4a_+^2}h}_{\mbox{const} \cdot \til h}\ +\ \frac{4(1+|z|^2)}{a_+^2}\eta^2.
	\]
	Now, using observation~\eqref{f:laplace_equivalence} we find that $V$ on $N(a_+,0)$ solves equation $\Delta_{\til h}V=0$ if and only if its lift to $\hat M\simeq \R^3\times S^1$ solves the Laplace equation
	\[
	\Delta_{g_f}V=0
	\]
	with respect to the flat metric $g_f$.
	
	Now we use the well-developed elliptic theory on the flat $\R^3\times S^1$ to conclude the proof. It is known $\R^3\times S^1$ is non-parabolic, i.e., given any point $y\in \R^3\times S^1$ there exists an everywhere positive Green's function $\Gamma_y(x)$ solving the equation
	\[
	\Delta_{g_f}\Gamma_y(x)=-2\pi\delta_y(x).
	\]
	Such $\Gamma_y(x)$ can be either constructed explicitly as a convergent sum of Green's functions of $\R^4$, or its existence can be proved using a general result, e.g.,~\cite[\S 3]{GrigoryanParabolicity}. Furthermore the flat $\R^4$ satisfies the Liouville property~--- any entire positive harmonic function must be constant, therefore its isometric quotient $\R^3\times S^1$ also satisfies the Liouville property.
	
	Let $z=[S^1y]\in \hat M/S^1$ be the corresponding point in the orbit space. Averaging $\Gamma_y(x)$ with respect to the $S^1$ action
	\[
	G_{z}(x):=\frac{1}{2\pi}\int_{S^1} \Gamma_{y}(u\cdot x)du
	\]
	we obtain an $S^1$-invariant function $G_{z}(x)$ which descends to the orbifold quotient $\hat M/S^1$ and solves $\Delta_{g_f} G_z(x)=-2\pi\delta_{S^1y}(x)$ in $\hat M$ or equivalently
	\[
	\Delta_{\til h} G_z(x)=-2\pi\delta_z(x),
	\]
	in $\hat M/S^1$. Moreover $\lim G_z(x)=0$ as $d_{g_f}(x,z)\to \infty$. If $G'_z(x)$ is another function solving the same equation with $\inf G'_z(x)=0$, then the difference $G'_z(x)-G_z(x)$ is an entire harmonic function on $\hat M$. By the maximum principle, the difference is bounded from below by $0$, therefore by the Liouville property the two functions must differ by a constant. Since $\inf G'_z(x)=0$, the constant is zero.
	
	Let $V$ be any $S^1$-invariant positive harmonic function on the complement of finitely many $S^1$-orbits $\hat M\backslash \{S^1z_i\}$. Using the Riesz representation theorem and the $S^1$ invariance of $V$, we can decompose $V$ as
	\[
	V=V_0+\sum c_i G_{z_i},
	\]
	where $c_i$ are constants and $V_0$ is an entire harmonic function. Since $V$ is positive, $c_i>0$. Hence by the maximum principle $V_0$ is positive, and must be constant.
	
	Now consider $N(a_+,0)/\Z$. Action of $\Z$ on $N(a_+,0)$ lifts to a $g_f$-isometric action on $\hat M$. Using the same argument as above we have a bijection between superharmonic functions with poles at $\{z_i\}$ on $(N(a_+,0)/\Z,\til h)$ and $S^1$-invariant superharmonic functions on $(\hat M/\Z, g_f)$ with poles along $S^1$ orbits. It remains to note that manifold $\hat M/\Z=(\C\times \C^*)/\Z$ with metric $g_f$ has quadratic volume growth thus by a result of Cheng and Yau~\cite{ChengYau} does not admit any non-constant superharmonic functions.
\end{proof}

%
%
%
%
%
%
%
%

%
%
\subsection{Equation \texorpdfstring{$\Delta_{\tilde h}V=0$}{} in the case \texorpdfstring{$a_- \neq 0$}{}}

\begin{prop}\label{p:elliptic_aneq0}
	Let $N(a_+,a_-)$ be the orbifold from Definition~\ref{d:N(a+,a-)}. Consider a finite collection of points $\{z_i\}\subset N(a_+,a_-)$.
	\begin{enumerate}
		\item Given any point $z\in N(a_+,a_-)$ there exists a unique smooth positive function \[G_z(x)\colon N(a_+,a_-)\backslash\{z\}\to (0;\infty)\] such that:
		\begin{enumerate}
			\item $\Delta_{\til h}G_z(x)=-2\pi\delta_z(x)$ in the sense of distributions;
			\item $\inf_{x\in N(a_+,a_-)} G_z(x)=0$;
		\end{enumerate}

		\item If $V$ is a positive solution to the equation $\Delta_{\til h}V=0$ on the complement of a finite collection of points $\{z_i\}$, then there exist unique positive constants $c'_V,c_V'',c_i$ such that
		\[
		V(x)=c_V'+c_V''G_0(x)+\sum_{i} c_i G_{z_i}(x),
		\]
		where $G_0(x)=k_+^2e^{2\mu_+/k_+}+k_-^2e^{-2\mu_-/k_-}$.
		\item Any positive superharmonic function on $N(a_+,a_-)/\Z$ is constant.
	\end{enumerate}
\end{prop}
\begin{proof}
	The proof is essentially the same as the proof in the case $a_-=0$. For simplicity, we assume that $\mathrm{gcd}(k_+,k_-)=1$. The proof in the general case is analogous.
	
	As before, we use Remark~\ref{r:N_as_quotient} to realize $N(a_+,a_-)$ as the quotient of $\hat M=\C^2\backslash\{0\}$ by a circle action $u\cdot (z,w)=(u^{k_+}z,u^{k_-}w)$. Under this identification, the conformal factor $\psi^2$ becomes
	\[
	\psi^2=\frac{(|z|^2+|w|^2)^2}{4}.
	\]
	Let $X$ be the vector field generating the $S^1$ action on $\hat M$, and denote by $\eta$ the principal connection induced by $\hat g$. Then $|X|_{\hat g}=2k_+k_-$ and on $\hat M$ we can express metric $\hat g$ as
	\[
	\hat g=h+(2k_+k_-)^2\eta^2.
	\]
	Thus the flat metric on $\hat M$,
	\[
	g_f:=\frac{|z|^2+|w|^2}{4}\hat g=\Re\left(k_-^2dz\otimes d\bar z+k_+^2dw\otimes d\bar w\right),
	\]
	can be expressed as
	\[
	\begin{split}
	g_f&=\frac{|z|^2+|w|^2}{4}h+k_+^2k_-^2(|z|^2+|w|^2)\eta^2\\&=
	\frac{1}{k_+^2k_-^2(|z|^2+|w|^2)} \underbrace{\frac{k_+^2k_-^2(|z|^2+|w|^2)^2}{4} h}_{\mbox{const}\cdot \til h}\ +\ k_+^2k_-^2(|z|^2+|w|^2)\eta^2
	\end{split}	
	\]
	
	Using the same observation~\eqref{f:laplace_equivalence}, we find that $V$ on $N(a_+,a_-)$ solves $\Delta_{\til h}V=0$ if and only if its lift to $\hat M$ solves the Laplace equation with respect to the flat metric $g_f$.
	
	At this point we can repeat the argument of Proposition~\ref{p:elliptic_a0} with one difference~--- $\hat M=\C^2\backslash\{0\}$ is not complete with respect to the flat metric $g_f$. To this end, we have to consider $V$ on its completion $\C^2\supset \C^2\backslash\{0\}$. We observe that $V$ is a positive harmonic function (possibly with poles along a finite collection of $S^1$ orbits), undefined at the origin. Therefore we can decompose it as a linear combination of the Green's function of $\C^2$ with the pole at $\{0\}$ and a positive harmonic function (with the same poles as $V$) well-defined at $\{0\}$. The rest of the proof proceeds verbatim as the proof of Proposition~\ref{p:elliptic_a0}. Thus we have to add an extra term $G_0(x)$ in the decomposition of $V(x)$.  It remains to note that up to a constant multiple the Green's function on $\R^4$ is given by $G_0(x)=d_{g_f}(x,0)^{-2}=  (k_-^2|z|^2+k_+^2|w|^2)^{-1}$, which (again up to a constant multiple) equals $k_+^2e^{2\mu_+/k_+}+k_-^2e^{-2\mu_-/k_-}$ as in the statement of the proposition.
	
	In the second case of $N(a_+,a_-)/\Z$, the action of group $\Z$ lifts to a $\Z$-action on $\C^2\backslash \{0\}$ via contractions fixing the origin:
	\[
	|z|\mapsto \gl|z|, \quad |w|\mapsto \gl |w|,\quad \arg z\mapsto \arg z+\alpha_z, \quad \arg w\mapsto \arg w+\alpha_w
	\]
	It remains to note if a positive superharmonic function on $\R^4$ is invariant under these contractions, then it attains a minimum at an interior point and must be constant by the maximum princple.
	\end{proof}
	
	\begin{rmk}
	In either case, whether $a_-=0$ or $a_-\neq 0$, let $G_{z_i}$ be a Green's function on $N(a_+,a_-)\backslash\{z_i\}$ constructed in Propositions~\ref{p:elliptic_a0} and~\ref{p:elliptic_aneq0}. For $W=\til W G_{z_i}$ we have
	\[
	\beta=*_hdW+W\beta_0=\frac{\psi}{\til W}*_{\til h}dG_{z_i}+G_{z_i}(*_hd\til W+\til W\beta_0).
	\]
	Consider $z_i\in N(a_+,a_-)$~--- a smooth point, and a sphere $S^2_\ge$ of $\til h$-radius $\ge$ enclosing $z_i$ oriented using the outer normal. Then using the closedness of $\beta$ and the standard estimate $G_{z_i}(z)\sim d_{\til h}(z,z_i)^{-1}$ near $z_i$  we compute
	\[
	\int_{S^2_\ge}\beta=\lim_{\ge\to 0}\int_{S^2_\ge}\beta=\lim_{\ge\to 0}\int \frac{\psi}{\til W}*_{\til h}dG_{z_i}=-2\pi \frac{\psi(z_i)}{\til W(z_i)}.
	\]
	In particular, $c_{z_i}\beta\in H^2(N(a_+,a_-)\backslash\{z_i\},2\pi\Z)$ is a generator, where the normalization constant is given by
	\begin{equation}\label{f:c_zi_normalization}
	c_{z_i}:=\til W(z_i)/\psi(z_i).
	\end{equation}
	\end{rmk}

\subsection{Existence and uniqueness}

We conclude with the proofs of the main existence and classification results.  As above our discussion splits into the cases where $a_-$ is either vanishing or nonvanishing.  Furthermore, the case $a_- = 0$ splits into the cases according to the possible quotient spaces $N(a_+, 0)$ and $N(a_+, 0) / \mathbb Z$ as in Theorem \ref{t:nondegenerate_gk_description}.  We point to Figures \ref{f:fig1} and \ref{f:fig2} for a brief summary.

\begin{prop} \label{p:gensoln1} Let $(k_+,l_+)$ be a pair of coprime integers, $0\leq l_+<|k_+|$ and denote $a_+ = \frac{2}{k_+}$. Given $\gl\geq0$ and a finite collection of points $\{z_1, \dots, z_n\}$ in the smooth locus of $N(a_+, 0)$, let
\[
W=\til{W} \Bigl( \gl + \sum_{i=1}^nc_{z_i}G_{z_i} \Bigr),
\]
where $G_{z_i}$ is the Green's function centered at $z_i$ constructed in Proposition \ref{p:elliptic_a0} and $c_{z_i}>0$ is the normalization constant as in~\eqref{f:c_zi_normalization}.

There exists a unique complete regular rank one soliton $(M^4, g, I, J)$ such that
\begin{enumerate}
\item the orbit space $\pi\colon M\to M/{S^1}$ is homeomorphic to $N(a_+, 0)$,
\item $\pi^{-1} \{z_1, \dots, z_n\} = M^{S^1}$,
\item $\mathbf{T_+}\backslash \pi^{-1}\{z_1, \dots, z_n\}$ is the set of points of type $(k_+, l_+)$ (cf.\,Remark \ref{r:local_s1_quotient}),
\item In the image of the nondegeneracy locus we have
\begin{align*}
\Phi = a_+ \mu_+,
\end{align*}
\item With the above choice of functions $W$ and $p$, locally in the nondegeneracy locus the GK structure on $(M,g,I,J)$ is given by the construction of Theorem~\ref{t:nondegenerate_gk_description_v2}.
\end{enumerate}
\begin{proof} 

Condition (5) uniquely prescribes the curvature form $\beta=*_hdW+W\beta_0$ on a dense open part of $N(a_+,0)$, and given the explicit form of $\beta_0$ (see Remark \ref{r:Norbifold1}) and since $W$ is a smooth function on orbifold $N(a_+,0)$, the curvature form also extends to a smooth two-form on $N(a_+,0)$. Since $H^2(N(a_+,0),\Z)=0$, by our normalization of $G_{z_i}$, $\beta$ represents a class in
\[
H^2(N(a_+,0)\backslash\{z_1,\dots,z_n\},2\pi\Z).
\]
Thus the assumptions of Proposition~\ref{p:s1_bundle_seifert_dg} are satisfied, and there exists a unique Siefert fibration
\[
\pi : M_0 \to N(a_+,0)\backslash\{z_1,\dots,z_n\}
\]
together with a smooth connection form $\eta$ with curvature $\beta$. Since $N(a_+,0)$ is simply connected, $\eta$ is also unique up to a gauge transform.

Let $\Sigma = \{p = 1\}  = \{\rho = 0\} \subset N(a_+, 0)$. We set $\Phi = a_+ \mu_+$, which in turn defines a smooth function $p$ on $N(a_+,0)$:
\[
p=\frac{1-e^{a_+\mu_+}}{1+e^{a_+\mu_+}}.
\]
In the case $k_+ = \pm 1$, $\Sigma$ consists of smooth points, and by reordering we assume $\{z_1, \dots, z_m\} \in \Sigma^c, \{z_{m+1}, \dots, z_n\} \in \Sigma$. Otherwise if $|k_{\pm}|>1$, we have $m=n$.

By definition, the complement $N(a_+, 0) \backslash \Sigma$ is isomorphic to $\R^3_{\pmb\mu}/\Gamma_1$, where $\Gamma_1$ acts by translations in $\mu_1$ coordinate.
In particular, we can identify a simply connected open subset $U$ of $N(a_+, 0) \backslash \Sigma$ with an open subset in $\R^3_{\pmb\mu}$. By our choice of function $W$ and the corresponding form $\beta$, we can apply Theorem~\ref{t:nondegenerate_gk_description_v2} and endow $\pi^{-1}(U)$ with a GK structure determined by $p$ and $W$. Since $p$ does not depend on $\mu_1$, this GK structure is independent of the identification between $U$ and a subset of $\R^3_{\pmb\mu}$, yielding a GK structure on 
\begin{align*}
M_1 := \left(M_0 \backslash \pi^{-1} (\Sigma) \right) \cup \{y_1, \dots, y_m\},
\end{align*}
where $\{y_i\}$ are the fixed points of the natural $S^1$ action on $M_1$. The orbit space of this action is homeomorphic to $N(a_+, 0) \backslash \Sigma$ with $\pi(y_i)=z_i$, $1\leq i\leq m$. The constructed GK structure is smooth away from the points $\{y_1, \dots, y_m\}$ and is $C^{1,1}$ globally.

We now prove that $(g, I, J)$ extends smoothly across $\pi^{-1}(\Sigma)$.  As explained above, we know that $W, h$, and $\gb$ are smooth on the orbifold $N(a_+, 0) \backslash \{z_1, \dots, z_n\}$, so the metric $g$ will extend smoothly across $\pi^{-1} \left(\Sigma \backslash \{z_{m+1}, \dots, z_n\} \right)$. Next for our specific choice of $p$, the Lee form $\theta_I$ (see~\eqref{f:theta_through_p}) descends to $N(a_+,0)$ and extends smoothly across $\Sigma$, thus $H=-*_g\theta_I$ is a smooth 3-form on $M_0$. Arguing using parallel transport as in Proposition \ref{p:removable_singularity_GK}, it follows that the complex structures $I$ and $J$ extend smoothly (in fact real analytically) across this locus as well. This gives a possibly incomplete smooth GK structure on
\[
M_0\cup\{y_1,\dots,y_m\}
\]
with an orbit space $N(a_+,0)\backslash\{z_{m+1},\dots, z_n\}$.

Next, we prove that we can glue in fixed points $\{y_{m+1},\dots y_n\}$ such that
\[
M_0\cup \{y_1,\dots,y_{n}\}
\]
is a complete $C^{1,1}$ GK manifold with orbit space $N(a_+,0)$. Using the normalization of the Green's functions at $\{z_{m+1}, \dots, z_{n}\}$, and the fact that the locus $\Sigma\subset N(a_+,0)$ must be smooth as long as $m<n$ (see Corollary~\ref{c:fixed_smooth}), we observe that preimage $\pi^{-1}(B^3\backslash \{z_i\})$ of a small punctured ball around $z_i$ is diffeomorphic to $B^4\backslash \{0\}$. The function $W$ evidently satisfies the required local behavior of Proposition~\ref{p:removable_singularity} near each $x_i\in\{x_1,\dots,x_m\}$, thus we can extend the metric across $B^4$ after gluing in a point $y_i$.  
Therefore $H=-*_g\theta_I$ admits a $C^{1,1}$ extension on $M_0\cup \{y_1,\dots,y_{n}\}$, and we can apply the first part of Proposition \ref{p:removable_singularity_GK} to show that the GK structure extends in $C^{1,1}$ sense. Thus we have constructed a $C^{1,1}$ GK structure $(M,g,I,J)$ with an orbit space
\[
\pi\colon M\to N(a_+,0)
\]
which is smooth outside of a fixed point set $\{y_i\}$.

We now show that this GK structure extends smoothly across $\{y_i\}$. Given our choice of function $\Phi$, Proposition~\ref{p:invGKsoliton} implies that $(M,g,I,J)$ is a GK soliton on the nondegerate locus. Since the nondegenerate locus is open and dense, $(M,g,I,J)$ is a soliton globally. Given the explicit form of $df$ in Proposition \ref{p:invGKsoliton}, we know that the extension of $f$ is smooth.  With $C^{1,1}$ regularity for the metric, we can construct a harmonic coordinate system for $g$, in which the soliton system takes the form
\begin{align*}
0 =&\ g^{kl} \frac{\del^2 g_{ij}}{\del x^k \del x^l} + Q(g, \del g, H) + \del^2 f + \del g \star  \del f,\\
0 =&\ g^{kl} \frac{\del^2 H_{ijk}}{\del x^k \del x^l} + \del^2 g \star H + \del g \star \del H + d (i_{\N f} H).
\end{align*}
The first equation is a strictly elliptic equation for $g$ with $C^{1,\ga}$ coefficients and a $C^{\ga}$ inhomogeneous term, so by Schauder estimates we conclude $C^{2,\ga}$ estimates for $g$.  The second equation is then a strictly elliptic equation for $H$ with  $C^{2,\ga}$ coefficients and a $C^{\ga}$ inhomogeneous term, thus we conclude $C^{2,\ga}$ estimates for $H$.  Differentiating this system and applying a standard bootstrap argument gives $C^{\infty}$ regularity for both $g$ and $H$.  By the parallel transport argument for extending $I$ and $J$ from Proposition \ref{p:removable_singularity_GK}, it follows that $I$ and $J$ are smooth as well. Thus $(M,g,I,J)$ is a smooth regular rank one soliton.

It remains to address the completeness of $(M,g)$. The projection onto the orbit space
\[
\pi\colon M\to N(a_+,0)
\]
has compact fibers, therefore $(M,g)$ is complete if and only if $N(a_+,0)$ with the distance function induced by the horizontal metric is complete. By the construction, in the complement of the finite set $\{z_i\}$ the horizontal metric is given by
\[
g_{\mathrm{hor}}=Wh.
\]
Since $h$ is complete, an appropriate lower bound on $W$ will imply the completeness of $Wh$. Let us identify $(N(a_+,0), h)$ with a quotient of $(\C\times \C^*, \hat g)$ by a circle action (see Remark~\ref{r:N_as_quotient}) and denote coordinates on $\C\times \C^*$ by $z$ and $w$. Then
\[
\hat g=\frac{4}{1+|z|^2}\Re\left(
k_+^{2} dz\otimes d\bar z + \frac{dw\otimes d\bar w}{|w|^2}
\right)
\]
and
\[
\til W= C(|z|^2+1).
\]
On the other hand, using the explicit form of the Green's function on $\C\times \C^*$ with respect to a flat metric $g_f$ (see the proof of Proposition~\ref{p:elliptic_a0}) we observe that for any $z_i\in N(a_+,0)$, there is a constant $C'>0$ such that outside a compact subset $K$
\[
G_{z_i}\geq C'(|z|^2+\log|w|^2)^{-1/2}.
\]
This implies that
\[
W\hat g=C(\gl+c_i\sum G_{z_0})\Re\left(
k_+^{2} dz\otimes d\bar z + \frac{dw\otimes d\bar w}{|w|^2}
\right)
\]
is a complete metric on $\C\times \C^*$ as long as at least one of the numbers $(\gl,c_1,\dots,c_n)$ is positive. Since $Wh$ on $N(a_+,0)$ corresponds to the horizontal component of $W\hat g$ under the projection onto the orbit space $\C\times \C^*\to N(a_+,0)$, this implies the completeness of $Wh$, which in turn is equivalent to the completeness of $(M,g)$.
\end{proof}
\end{prop}

\begin{rmk} \label{r:polesonT} In the result above, it is necessary to restrict the points $\{z_1, \dots, z_n\}$ to the smooth locus of $N(a_+, 0)$.  Given a point $z$ on the orbifold locus, the Seifert fibration in the complement of $z$, near $z$, is not locally trivial, thus no neighborhood of $p$ can be homeomorphic to a punctured ball, and thus the metric completion is not a smooth manifold.  The same restrictions apply to the construction in the case $a_- \neq 0$ below individually on the two components $\mathbf{T_{\pm}}$.
\end{rmk}

Fix an action $\Z\colon N(a_+,0)\to N(a_+,0)$ induced by~\eqref{f:Z_action}. Now we are going to describe the set of GK solitons fibered over $N(a_+,0)/\Z$.

\begin{prop} \label{p:gensoln2}
	Let $(k_+,l_+)$ be a pair of coprime integers, $0\leq l_+<|k_+|$, and denote $a_+ = \frac{2}{k_+}$. Given $\gl>0$ let
	\[
	W=\lambda \til{W}.
	\]
	There exists an $S^1$ worth of complete regular rank one solitons $(M^4, g, I, J)$ such that
	\begin{enumerate}
		\item the orbit space $\pi\colon M\to M/{S^1}$ is homeomorphic to $N(a_+, 0)/\Z$,
		\item the action of $S^1$ on $M$ is free,
		\item $\mathbf{T_+}$ is the set of points of type $(k_+, l_+)$ (cf.\,Remark \ref{r:local_s1_quotient}),
		\item in the image of the nondegeneracy locus we have
		\begin{align*}
		\Phi = a_+ \mu_+,
		\end{align*}
		\item with the above choice of functions $W$ and $p$, locally in the nondegeneracy locus the GK structure on $(M,g,I,J)$ is given by the construction of Theorem~\ref{t:nondegenerate_gk_description_v2}.
	\end{enumerate}
	\begin{proof}
		By Proposition~\ref{p:gensoln1}, we know that there is a unique GK soliton $(\til M,g,I,J)$ such that
		\begin{enumerate}
		\item the orbit space $\pi\colon \til M\to \til M/{S^1}$ is homeomorphic to $N(a_+, 0)$,
		\item the action of $S^1$ on $\til M$ is free,
		\item $\mathbf{T_+}$ is the set of points of type $(k_+, l_+)$,
		\item In the image of the nondegeneracy locus we have
		\begin{align*}
		\Phi = a_+ \mu_+,
		\end{align*}
		\item With the above choice of functions $\til W$ and $p$, locally in the nondegeneracy locus the GK structure on $(\til M,g,I,J)$ is given by the construction of Theorem~\ref{t:nondegenerate_gk_description_v2}.
		\end{enumerate}
		Now, since the action of $\Z$ on $N(a_+,0)$ preserves $\Phi$, the curvature form $\beta$ is also invariant under $\Z$, and there is a family of lifts of this action to an action
		\begin{equation}\label{f:Z_action_lift}
		\Z\times \til M\to \til M
		\end{equation}
		parametrized by $\mathrm{Hom}(\pi_1(N(a_+,0)/\Z), S^1)\simeq S^1$. Using again the invariance of $\Phi$ and $W$ under $\Z$ we see that the lifted action~\eqref{f:Z_action_lift} preserves $(g,I,J)$, thus induces a GK structure on $M:=\til M/\Z$, which we denote by the same symbols $(g,I,J)$. Finally, since the one-form
		\[
		 \tfrac{1}{2}\left(p a_+ d\mu_+-a_+d\mu_+\right)
		\]
		from Proposition~\ref{p:invGKsoliton} is exact on $N(a_+,0)/\Z$, we conclude that $(M,g,I,J)$ is a GK soliton.
	\end{proof}
\end{prop}

Now we focus on the case $a_-\neq 0$. Recall that $N(a_+,a_-)$ is isomorphic to the quotient of $\C^2\backslash\{0\}$ by a linear circle action with weights $k_{\pm}=2/a_\pm$. In particular, $N(a_+,a_-)$ is homeomorphic to $S^2\times \R$. Let $\{z_1,\dots,z_n\}\subset N(a_+,a_-)$ be a finite set of points in the smooth locus, and choose a two-cycle $S_0\in C_2(N(a_+,a_-),\Z)$, $S_0=S^3/S^1$, where $S^3\subset \C^2$ is a sphere of small radius enclosing the origin. As before, let $W$ be a function on $N(a_+,a_-)\backslash\{z_1,\dots, z_n\}$ such that
\[
\beta=*_hdW+W\beta_0
\]
is closed and denote
\[
S(W):=\frac{1}{2\pi}\int_{S_0}\beta\in \R.
\]
Let $\Sigma_{\pm}=\{p=\pm 1\}\subset N(a_+,a_-)$ be a disjoint union of two curves. We orient each of the components so that $S_0\cap \Sigma_{\pm}=+1$.

\begin{prop} \label{p:gensoln3} 
Let $(k_+,l_+)$, $(k_-,l_-)$ be two pairs of coprime integers, $0\leq l_\pm<|k_\pm|$ and denote $a_{\pm} = \frac{2}{k_{\pm}}$. Given $\lambda>0$, $\lambda_0\geq 0$ and a finite collection of points $\{z_1, \dots, z_n\}$ in the smooth locus of $N(a_+, a_-)$, let
\[
W=\til W\left(\lambda+\lambda_0G_0+\sum_{i=1}^n c_{z_i}G_{z_i}\right),
\]
where $G_{z_i}$ is the Green's function centered at $z_i$ constructed in Proposition~\ref{p:elliptic_aneq0} and $c_{z_i}>0$ is the normalization constant as in~\eqref{f:c_zi_normalization}. Assume that
\begin{equation}\label{f:W_quantization}
S(W)-\frac{l_+}{k_+}-\frac{l_-}{k_-}\in \Z.
\end{equation}

Then there exists a unique complete regular rank one soliton $(M^4, g, I, J)$ such that 
\begin{enumerate}
\item the orbit space $\pi\colon M\to M/{S^1}$ is homeomorphic to $N(a_+, a_-)$,
\item $\pi^{-1} \{z_1, \dots, z_n\} = M^{S^1}$,
\item $\mathbf{T_{\pm}}\backslash \pi^{-1}\{z_1, \dots, z_n\}$ is the set of points of type $(k_{\pm}, l_{\pm})$ (cf.\,Remark \ref{r:local_s1_quotient}),
\item in the image of the nondegeneracy locus we have
\begin{align*}
\Phi = a_+ \mu_+ + a_- \mu_-,
\end{align*}
\item with the above choice of functions $W$ and $p$, locally in the nondegeneracy locus the GK structure on $(M,g,I,J)$ is given by construction of Theorem~\ref{t:nondegenerate_gk_description_v2}.
\end{enumerate}
\begin{proof}
As in the proof of Proposition~\ref{p:gensoln1} we start by constructing the smooth manifold $M$. We are given an orbifold $N(a_+,a_-)$ with cooriented submanifolds $\Sigma_{\pm}=\{p=\pm 1 \}$ of codimension two. Each submanifold $\Sigma_{\pm}$ comes equipped with a pair of integers $(k_\pm,l_\pm)$. We claim that under assumption~\eqref{f:W_quantization},
\[
[\beta/2\pi]-\frac{l_+}{k_+}[\Sigma_+]-\frac{l_-}{k_-}[\Sigma_-]\in H^2(N(a_+,a_-)\backslash\{z_1,\dots,z_n\},\Z).
\]
Indeed, $H_2(N(a_+,a_-)\backslash\{z_1,\dots,z_n\},\Z)\simeq \Z^{n+1}$ is generated by $S_0$ and a collection of $n$ small spheres $\{S_i\}_{i=1}^n$ each enclosing one point $z_i$. If $z_i\not\in \Sigma_{\pm}$, we can assume that the curves $\Sigma_{\pm}$ do not intersect the spheres $S_i$. Otherwise, if $z_i\in \Sigma_+$ ($z_i\in \Sigma_-$), we necessarily have $k_+=\pm 1$, $l_+=0$ (resp.\,$k_-=\pm 1$, $l_-=0$). In either case, by our choice of normalization constants $c_{z_i}$, we have
\[
\left\langle [\beta/2\pi]-\frac{l_+}{k_+}[\Sigma_+]-\frac{l_-}{k_-}[\Sigma_-], S_i \right\rangle=\frac{1}{2\pi}\int_{S_i} \beta=-1.
\]
It remains to check that
\[
\left\langle [\beta/2\pi]-\frac{l_+}{k_+}[\Sigma_+]-\frac{l_-}{k_-}[\Sigma_-], S_0 \right\rangle\in \Z
\]
but this is exactly the statement of equation~\eqref{f:W_quantization}.

Thus we have all the ingredients in place to apply Proposition~\ref{p:s1_bundle_seifert_dg} and construct a smooth manifold $M_0$ and a Seifert fibtration
\[
\pi\colon M_0\to N(a_+,a_-)\backslash\{z_1,\dots,z_n\},
\]
with connection $\eta$ and curvature form $\beta$ (if either of $k_\pm$ is negative, we reverse the coorientation of the corresponding submanifold $\Sigma_{\pm}$). The rest of the construction of the GK soliton $(M,g,I,J)$ goes exactly as in the proof of Proposition~\ref{p:gensoln1}.

Up to this point the construction works for any constants $\lambda,\lambda_0\geq 0$. It remains to understand if the constructed manifold $(M,g)$ is complete. Repeating the argument in Proposition~\ref{p:gensoln1}, completeness of $(M,g)$ is equivalent to the completeness of $\C^2\backslash\{0\}$ equipped with the metric $W\hat g$, where
\[
\hat g=\frac{4}{|z|^2+|w|^2}\Re\left(k_-^2dz\otimes d\bar z+k_+^2dw\otimes d\bar w\right)
\]
is a complete metric on $\C^2\backslash \{0\}$ (see part (2) of Remark~\ref{r:N_as_quotient}). Since $\til W$ is bounded below by a positive constant, function
\[
W=\til W\left(\lambda+\lambda_0G_0+\sum_{i=1}^n c_{z_i}G_{z_i}\right)
\]
is also bounded from below by a positive constant, as long as $\lambda>0$. Thus, for $\lambda>0$ the metric $W\hat g$ induces a complete distance function on $\C^2\backslash\{0\}$, and $(M,g)$ is also complete.

We claim that if $\lambda=0$, then $g$ is not complete. To this end, we estimate $W$ from above. Choose a compact set $K\subset \C^2\supset \C^2\backslash\{0\}$ containing preimages of all the points $\{z_i\}$. Using the explicit form of Green's function on $\C^2$ equipped with a flat metric, we observe that there exists a constant $C=C(K)>0$ such that in the complement of $K$,
\[
G_{z_i}<C(|z|^2+|w^2|)^{-1}.
\]
Therefore, as long as $\lambda=0$
\[
W<(n+1)C(|z|^2+|w^2|)^{-1}.
\]
Thus $(\C^2\backslash\{0\}, W\hat g)$ is incomplete, as any radial curve $(tz_0,tw_0)$, $t\in[1,+\infty)$ has a finite length.
\end{proof}
\end{prop}

With the above propositions we have Theorem~\ref{t:existence}:

\begin{proof}[Proof of Theorem~\ref{t:existence}]
	Propositions~\ref{p:gensoln1},~\ref{p:gensoln2} and~\ref{p:gensoln3} give explicit existence statements for complete regular rank one solitons as stated in Theorem~\ref{t:existence}.
\end{proof}

Finally we give the proof of the classification statement of Theorem~\ref{t:uniqueness}:

\begin{proof}[Proof of Theorem~\ref{t:uniqueness}]
	Let $(M,g,I,J)$ be a 4-dimensional complete generalized K\"ahler-Ricci soliton, such that its Poisson tensor does not vanish identically. Assume that the Ricci curvature of $(M,g)$ is bounded from below and $\dim\Isom(g)\leq 1$. As before, denote by $V_I=\tfrac{1}{2}(\theta_I^\#-\N f)$ and $V_J=\tfrac{1}{2}(-\theta_I^\#-\N f)$ the corresponding soliton vector fields. By Proposition~\ref{p:solitonVF}, $IV_I$ and $JV_J$ are Killing.
	
	If $\dim\Isom(g)=0$, then both Killing vector fields $IV_I$ and $JV_J$ must vanish, thus $\theta_I^\#-\N f=-\theta_I^\#-\N f=0$. Hence $\theta_I=0$ and $\N f=0$, so $(M,g,I,J)$ is a usual Calabi-Yau metric  with respect to either of complex structure $I$ and $J$. Since $I$ and $J$ are linearly independent, this forces the holonomy of $g$ to be contained in $SU(2)$, so $(M,g,I)$ is indeed hyperK\"ahler.
	
	From now on we can assume that $\mathrm{span}(IV_I, JV_J)$ is generically one-dimensional, so $(M,g,I,J)$ is a rank one soliton. By Proposition~\ref{p:one-diml_isometry}, we have two possibilities:
	
	\begin{enumerate}
		\item $\Isom(g)\simeq S^1$,
		\item $\Isom(g)\simeq \R$.
	\end{enumerate}
	
	In the first case there is a vector field $X$ generating an $S^1$ action on $(M,g,I,J)$ preserving the GK structure, such that $IV_I$ and $JV_J$ are constant multiples of $X$. Furthermore, since $M$ has finite topology~--- $\dim H^*(M,\R)<\infty$~--- the action of $S^1$ has finitely many fixed points. Thus $(M,g,I,J)$ is a complete regular rank one soliton (cf.\,Definition~\ref{d:regular_rank_one}). Therefore, we can invoke Theorem~\ref{t:M_orbifold_quotient} and conclude that the orbit space $\pi\colon M\to N$ is given by one of the spaces $N(a_+,0)$, $N(a_+,0)/\Z$ or $N(a_+,a_-)$ and in the complement of the degeneracy locus the GK structure is locally given by Theorem~\ref{t:nondegenerate_gk_description_v2}. Let $\pi\colon M_0\to N_0$ be the Seifert fibration in the complement of the isolated fixed points. By Propositions~\ref{p:elliptic_a0} and~\ref{p:elliptic_aneq0}, the function $W$ solving the equation $d(*_hdW+W\beta_0)=0$ must be a linear combination of $\til W$ and $\til W G_{z_i}$, and if $N=N(a_+,0)/\Z$, $W$ must be a multiple of $\til W$. Since $G_{z_i}(x)\to +\infty$ as $x\to z_i$, and $W$ is positive, the constants $c_{z_i}$ in front of $\til W G_{z_i}$ in the expansion of $W$ must be positive. Furthermore, the local structure of $(M,g,I,J)$ near fixed points (Lemma~\ref{l:free_nondegenerate} and Corollary~\ref{c:fixed_smooth}) uniquely determines the exact values of constants $c_{z_i}$, $z_i\in N$. If $N=N(a_+,a_-)$, $a_-\neq 0$, there is also a necessary integrality condition
	\[
	\frac{1}{2\pi}\int_{S^2_0}\beta - \frac{l_+}{k_+} -\frac{l_-}{k_-}\in \Z
	\]
	where $S^2_0$ is a 2-sphere generating $H^2(N(a_+,a_-),\Z)$. Due to Theorem~\ref{t:M_orbifold_quotient}, the choice of $W$ and the prescribed form of $\Phi$ uniquely determine the GK structure $(M,g,I,J)$ up to an isomorphism generated by translations in $\mu$-coordinates.
	
	In the second case, the $\Z$-quotient $M/\Z$ of $M$ admits a complete regular rank one GK soliton structure $(M/\Z,g,I,J)$ with a free $S^1$ action such that the bundle $M/\Z\to (M/\Z)/S^1$ is trivial. Thus $(M/\Z, g,I,J)$ is given by one of the constructions of Propositions~\ref{p:gensoln1}, $\ref{p:gensoln2}$ and $\ref{p:gensoln3}$ with $k_{\pm 1}=\pm 1$, $\{z_i\} = \emptyset$, and $[\beta]$ representing the zero class in the second cohomology. The initial manifold $(M,g,I,J)$ is then a $\Z$ cover of $(M/\Z,g,I,J)$.
\end{proof}



\begin{thebibliography}{10}
	
	\bibitem{ad-le-ru-07}
	A.\,Adem, J.\,Leida, and Y.\,Ruan.
	\newblock {\em Orbifolds and stringy topology}, volume 171 of {\em Cambridge
		Tracts in Mathematics}.
	\newblock Cambridge University Press, Cambridge, 2007.
	
	\bibitem{an-kr-le-89}
	M.\,Anderson, P.\,Kronheimer, and C.\,LeBrun.
	\newblock Complete {R}icci-flat {K}\"{a}hler manifolds of infinite topological
	type.
	\newblock {\em Comm. Math. Phys.}, 125(4):637--642, 1989.
	
	\bibitem{AGG}
	V.\,Apostolov, P.\,Gauduchon, and G.\,Grantcharov.
	\newblock Bi-{H}ermitian structures on complex surfaces.
	\newblock {\em Proc. London Math. Soc. (3)}, 79(2):414--428, 1999.
	
	\bibitem{ap-gu-07}
	V.\,Apostolov and M.\,Gualtieri.
	\newblock Generalized {K}\"{a}hler manifolds, commuting complex structures, and
	split tangent bundles.
	\newblock {\em Comm. Math. Phys.}, 271(2):561--575, 2007.
	
	\bibitem{ASNDGKCY}
	V.\,Apostolov and J.\,Streets.
	\newblock The nondegenerate generalized {K}\"{a}hler {C}alabi-{Y}au problem.
	\newblock {\em J. Reine Angew. Math.}, 777:1--48, 2021.
	
	\bibitem{Bielawski}
	R.\,Bielawski.
	\newblock Complete hyper-{K}\"{a}hler {$4n$}-manifolds with a local
	tri-{H}amiltonian {$\mathbf R^n$}-action.
	\newblock {\em Math. Ann.}, 314(3):505--528, 1999.
	
	\bibitem{Bismut}
	J.-M.\,Bismut.
	\newblock A local index theorem for non-{K}\"ahler manifolds.
	\newblock {\em Math. Ann.}, 284(4):681--699, 1989.
	
	\bibitem{BoyerConformalDuality}
	C.~P.\,Boyer.
	\newblock Conformal duality and compact complex surfaces.
	\newblock {\em Math. Ann.}, 274(3):517--526, 1986.
	
	\bibitem{strings-85}
	C.\,Callan, D.\,Friedan, E.\,Martinec, and M.\,Perry.
	\newblock Strings in background fields.
	\newblock {\em Nuclear Phys. B}, 262(4):593--609, 1985.
	
	\bibitem{ca-gu-10}
	G.\,Cavalcanti and M.\,Gualtieri.
	\newblock Generalized complex geometry and {$T$}-duality.
	\newblock In {\em A celebration of the mathematical legacy of {R}aoul {B}ott},
	volume~50 of {\em CRM Proc. Lecture Notes}, pages 341--365. Amer. Math. Soc.,
	Providence, RI, 2010.
	
	\bibitem{ChengYau}
	S.~Y.\,Cheng and S.~T.\,Yau.
	\newblock Differential equations on {R}iemannian manifolds and their geometric
	applications.
	\newblock {\em Comm. Pure Appl. Math.}, 28(3):333--354, 1975.
	
	\bibitem{Carmo}
	M.\,do~Carmo.
	\newblock {\em Differential Geometry of Curves and Surfaces: Revised and
		Updated Second Edition}.
	\newblock Dover Books on Mathematics. Dover Publications, 2016.
	
	\bibitem{gf-st-20}
	M.\,Garcia-Fernandez and J.\,Streets.
	\newblock {\em Generalized Ricci Flow}.
	\newblock University Lecture Series. American Mathematical Society, 2021.
	
	\bibitem{ga-hu-ro-84}
	S.\,Gates, Jr., C.\,Hull, and M.\,Ro\v{c}ek.
	\newblock Twisted multiplets and new supersymmetric nonlinear
	{$\sigma$}-models.
	\newblock {\em Nuclear Phys. B}, 248(1):157--186, 1984.
	
	\bibitem{GauduchonIvanov}
	P.\,Gauduchon and S.\,Ivanov.
	\newblock Einstein-{H}ermitian surfaces and {H}ermitian {E}instein-{W}eyl
	structures in dimension {$4$}.
	\newblock {\em Math. Z.}, 226(2):317--326, 1997.
	
	\bibitem{GauduchonWeyl}
	P.\,Gauduchon.
	\newblock Structures de {W}eyl-{E}instein, espaces de twisteurs et vari\'et\'es
	de type {$S^1\times S^3$}.
	\newblock {\em J. Reine Angew. Math.}, 469:1--50, 1995.
	
	\bibitem{GibbonsHawking}
	G.\,Gibbons and S.\,Hawking.
	\newblock Gravitational multi-instantons.
	\newblock {\em Physics Letters B}, 78(4):430--432, 1978.
	
	\bibitem{Gilkey-98}
	P.\,Gilkey, J.\,Park, and W.\,Tuschmann.
	\newblock Invariant metrics of positive {R}icci curvature on principal bundles.
	\newblock {\em Math. Z.}, 227(3):455--463, 1998.
	
	\bibitem{GrigoryanParabolicity}
	A.\,Grigoryan.
	\newblock The existence of positive fundamental solutions of the {L}aplace
	equation on {R}iemannian manifolds.
	\newblock {\em Mat. Sb. (N.S.)}, 128(170)(3):354--363, 446, 1985.
	
	\bibitem{gu-10}
	M.\,Gualtieri.
	\newblock Branes on {P}oisson varieties.
	\newblock In {\em The many facets of geometry}, pages 368--394. Oxford Univ.
	Press, Oxford, 2010.
	
	\bibitem{gu-14}
	M.\,Gualtieri.
	\newblock Generalized {K}\"{a}hler geometry.
	\newblock {\em Comm. Math. Phys.}, 331(1):297--331, 2014.
	
	\bibitem{gu-gi-ka}
	V.\,Guillemin, V.\,Ginzburg, and Y.\,Karshon.
	\newblock {\em Moment maps, cobordisms, and {H}amiltonian group actions},
	volume~98 of {\em Mathematical Surveys and Monographs}.
	\newblock American Mathematical Society, Providence, RI, 2002.
	\newblock Appendix J by Maxim Braverman.
	
	\bibitem{HitchinPoisson}
	N.\,Hitchin.
	\newblock Instantons, {P}oisson structures and generalized {K}\"ahler geometry.
	\newblock {\em Comm. Math. Phys.}, 265(1):131--164, 2006.
	
	\bibitem{hi-10}
	N.\,Hitchin.
	\newblock Lectures on generalized geometry.
	\newblock In {\em Surveys in differential geometry. {V}olume {XVI}. {G}eometry
		of special holonomy and related topics}, volume~16 of {\em Surv. Differ.
		Geom.}, pages 79--124. Int. Press, Somerville, MA, 2011.
	
	\bibitem{Fritz}
	F.\,John.
	\newblock The fundamental solution of linear elliptic differential equations
	with analytic coefficients.
	\newblock {\em Comm. Pure Appl. Math.}, 3:273--304, 1950.
	
	\bibitem{ko-05}
	J.\,Koll\'{a}r.
	\newblock Circle actions on simply connected 5-manifolds.
	\newblock {\em Topology}, 45(3):643--671, 2006.
	
	\bibitem{kosz}
	J.\,Koszul.
	\newblock Sur certains groupes de transformations de {L}ie.
	\newblock In {\em G\'{e}om\'{e}trie diff\'{e}rentielle. {C}olloques
		{I}nternationaux du {C}entre {N}ational de la {R}echerche {S}cientifique,
		{S}trasbourg, 1953}, pages 137--141. Centre National de la Recherche
	Scientifique, Paris, 1953.
	
	\bibitem{LeBrun}
	C.\,LeBrun.
	\newblock Anti-self-dual {H}ermitian metrics on blown-up {H}opf surfaces.
	\newblock {\em Math. Ann.}, 289(3):383--392, 1991.
	
	\bibitem{PontecorvoCS}
	M.\,Pontecorvo.
	\newblock Complex structures on {R}iemannian four-manifolds.
	\newblock {\em Math. Ann.}, 309(1):159--177, 1997.
	
	\bibitem{SchoenYau}
	R.\,Schoen and S.-T.\,Yau.
	\newblock Conformally flat manifolds, {K}leinian groups and scalar curvature.
	\newblock {\em Invent. Math.}, 92(1):47--71, 1988.
	
	\bibitem{StreetsNDGKS}
	J.\,Streets.
	\newblock Generalized {K}\"ahler-{R}icci flow and the classification of
	nondegenerate generalized {K}\"ahler surfaces.
	\newblock {\em Adv. Math.}, 316:187--215, 2017.
	
	\bibitem{st-19-soliton}
	J.\,Streets.
	\newblock Classification of solitons for pluriclosed flow on complex surfaces.
	\newblock {\em Math. Ann.}, 375(3-4):1555--1595, 2019.
	
	\bibitem{PCF}
	J.\,Streets and G.\,Tian.
	\newblock A parabolic flow of pluriclosed metrics.
	\newblock {\em Int. Math. Res. Not. IMRN}, (16):3101--3133, 2010.
	
	\bibitem{GKRF}
	J.\,Streets and G.\,Tian.
	\newblock Generalized {K}\"ahler geometry and the pluriclosed flow.
	\newblock {\em Nuclear Phys. B}, 858(2):366--376, 2012.
	
	\bibitem{PCFReg}
	J.\,Streets and G.\,Tian.
	\newblock Regularity results for pluriclosed flow.
	\newblock {\em Geom. Topol.}, 17(4):2389--2429, 2013.
	
	\bibitem{SU}
	J.\,Streets and Y.\,Ustinovskiy.
	\newblock Classification of generalized {K}\"{a}hler-{R}icci solitons on
	complex surfaces.
	\newblock {\em Comm. Pure Appl. Math.}, 74(9):1896--1914, 2021.
	
	\bibitem{YauIsometries}
	S.~T.\,Yau.
	\newblock Remarks on the group of isometries of a {R}iemannian manifold.
	\newblock {\em Topology}, 16(3):239--247, 1977.
	
\end{thebibliography}

\end{document}